\documentclass[a4paper]{amsart}

\usepackage{amsmath, tabularx}
\usepackage{amsthm}
\usepackage{amssymb}
\usepackage{amsfonts}
\usepackage{paralist}
\usepackage{aliascnt}
\usepackage[initials]{amsrefs}
\usepackage{amscd}
\usepackage{blkarray}
\usepackage{mathbbol}
\usepackage[inner=3cm,outer=3cm, bottom=3.5cm]{geometry}
\usepackage{microtype}
\usepackage{setspace}
\usepackage[colorlinks=true,linkcolor=blue,urlcolor=blue]{hyperref}
\usepackage{tikz}
\usetikzlibrary{matrix}

\newcommand{\kk}{\mathbb{K}}

\newcommand{\ZZ}{\mathbb{Z}}
\newcommand{\ZZC}{\normalfont\mathcal{Z}}
\newcommand{\MM}{\normalfont\mathcal{M}}
\newcommand{\PP}{\normalfont\mathbb{P}}
\newcommand{\xx}{\normalfont\mathbf{x}}
\newcommand{\yy}{\normalfont\mathbf{y}}
\newcommand{\dd}{\normalfont\mathbf{d}}
\newcommand{\cc}{\normalfont\mathbf{c}}
\newcommand{\mm}{\normalfont\mathfrak{m}}
\newcommand{\xy}{(x_0,x_1)}
\newcommand{\SSS}{\widetilde{S}}
\newcommand{\pp}{\normalfont\mathfrak{p}}

\newcommand{\qqq}{\normalfont\mathfrak{q}}
\newcommand{\nn}{\normalfont\mathfrak{N}}
\newcommand{\nnn}{\normalfont\mathfrak{n}}
\newcommand{\rank}{\normalfont\text{rank}}

\newcommand{\sat}{\normalfont\text{sat}}
\newcommand{\depth}{\normalfont\text{depth}}
\newcommand{\grade}{\normalfont\text{grade}}

\newcommand{\Ext}{\normalfont\text{Ext}}

\newcommand{\Ker}{\normalfont\text{Ker}}
\newcommand{\Coker}{\normalfont\text{Coker}}
\newcommand{\Quot}{\normalfont\text{Quot}}

\newcommand{\HT}{\normalfont\text{ht}}
\newcommand{\Syz}{\normalfont\text{Syz}}

\newcommand{\Supp}{\normalfont\text{Supp}}
\newcommand{\Ass}{\normalfont\text{Ass}}
\newcommand{\Sym}{\normalfont\text{Sym}}
\newcommand{\Rees}{\mathcal{R}}
\newcommand{\Hom}{\normalfont\text{Hom}}

\newcommand{\EEQ}{\mathcal{K}}
\newcommand{\REQ}{\mathcal{I}}
\newcommand{\BB}{\mathcal{B}}
\newcommand{\OO}{\mathcal{O}}
\newcommand{\T}{\mathcal{T}}

\newcommand{\Cc}{\mathcal{C}}
\newcommand{\FF}{\normalfont\mathcal{F}}
\newcommand{\HL}{\normalfont\text{H}_{\mm}}
\newcommand{\HH}{\normalfont\text{H}}

\newcommand{\AAA}{\mathcal{A}}

\newcommand{\bideg}{\normalfont\text{bideg}}
\newcommand{\Proj}{\normalfont\text{Proj}}
\newcommand{\Spec}{\normalfont\text{Spec}}
\newcommand{\multProj}{\normalfont\text{MultiProj}}
\newcommand{\biProj}{\normalfont{\text{MultiProj}}_{\AAA\text{-gr}}}

\newcommand{\xvars}[1]{x_{#1,0}, x_{#1,1},\ldots,x_{#1,r_{#1}}}
\newcommand{\Sylv}{{\mathcal{S}\textit{ylv}}_{(x_0,x_1)}(\varphi)}

\newtheorem{theorem}{Theorem}[section]

\newaliascnt{corollary}{theorem}
\newtheorem{corollary}[corollary]{Corollary}
\aliascntresetthe{corollary}

\newaliascnt{lemma}{theorem}
\newtheorem{lemma}[lemma]{Lemma}
\aliascntresetthe{lemma}

\newaliascnt{conjecture}{theorem}

\aliascntresetthe{conjecture}

\newaliascnt{proposition}{theorem}
\newtheorem{proposition}[proposition]{Proposition}
\aliascntresetthe{proposition}

\newaliascnt{definition}{theorem}
\newtheorem{definition}[definition]{Definition}
\aliascntresetthe{definition}

\newaliascnt{notation}{theorem}
\newtheorem{notation}[notation]{Notation}
\aliascntresetthe{notation}

\newaliascnt{example}{theorem}
\newtheorem{example}[example]{Example}
\aliascntresetthe{example}

\newaliascnt{remark}{theorem}
\newtheorem{remark}[remark]{Remark}
\aliascntresetthe{remark}

\newaliascnt{problem}{theorem}

\aliascntresetthe{problem}

\newaliascnt{construction}{theorem}

\aliascntresetthe{construction}

\newaliascnt{algorithm}{theorem}
\newtheorem{algorithm}[algorithm]{Algorithm}
\aliascntresetthe{algorithm}

\newaliascnt{defprop}{theorem}

\aliascntresetthe{defprop}

\def\equationautorefname~#1\null{(#1)\null}
\def\sectionautorefname~#1\null{Section #1\null}
\def\subsectionautorefname~#1\null{\S #1\null}

\begin{document}


\title{Degree and birationality of multi-graded rational maps}

\author{Laurent Bus\'e}
\address{Universit\'e C\^{o}te d'Azur, Inria,  2004 route des Lucioles, 06902 Sophia Antipolis, France}
\email{Laurent.Buse@inria.fr}
\urladdr{http://www-sop.inria.fr/members/Laurent.Buse/}

	
\author{Yairon Cid-Ruiz}
\address{Max Planck Institute for Mathematics in the Sciences, Inselstrasse 22, 04103 Leipzig, Germany}
\email{cidruiz@mis.mpg.de}	
\urladdr{https://ycid.github.io/}
	
\author{Carlos D'Andrea}
\address{Departament de Matem\`atiques i Inform\`atica, Universitat de Barcelona. Gran Via 585, 08007 Barcelona, Spain}
\email{cdandrea@ub.edu}
\urladdr{http://www.ub.edu/arcades/cdandrea.html}
	



	

\date{\today}

\begin{abstract}
		We give formulas and effective sharp bounds for the degree of multi-graded rational maps and provide some effective and computable criteria for birationality in terms of their algebraic and geometric properties. 
		We also extend the Jacobian dual criterion  to the multi-graded setting.
		Our approach is  based on the study of  blow-up algebras, including syzygies, of the ideal generated by the defining polynomials of the rational map. A key ingredient is a new algebra that we call the \textit{saturated special fiber ring}, which turns out to be a fundamental tool to analyze the degree of a rational map.  
		We also provide a very effective birationality criterion and a complete description of the equations of the associated Rees algebra of a particular class of plane rational maps.		
\end{abstract}	

\maketitle	
	



\section{Introduction}

Questions and results concerning the degree and birationality of rational maps are classical in the literature from both an algebraic and geometric point of view. 
These problems have been extensively studied since the work of Cremona in 1863 and are still very active research topics (see e.g.~\cite{ACBook,HULEK_KATZ_SCHREYER_SYZ,Simis_cremona,dolgachev-notes,AB_INITIO,HACON} and the references therein). 
However, there seems to be very few results and no general theory available for multi-graded rational maps, i.e.~maps that are defined by a collection of multi-homogeneous polynomials over a subvariety of a product of projective spaces.   
At the same time, there is an increasing interest in those maps, for both theoretic and applied purposes (see e.g.~\cite{SEDERBERG20151,SEDERBERG20161,EFFECTIVE_BIGRAD, HWH, BOTBOL_IMPLICIT_SURFACE}). 
This paper is the first step towards the development of a general theory of rational and birational multi-graded maps, providing also new insights in the single-graded case. 
We give formulas and effective sharp bounds for the degree of multi-graded rational maps and provide some effective criteria for birationality in terms of their algebraic and geometric properties. Sometimes we also improve known results in the single-graded case. Our approach is based on the study of blow-up algebras, including syzygies, of the ideal generated by the defining polynomials of a rational map, which is called the base ideal of the map. This idea goes back to \cite{HULEK_KATZ_SCHREYER_SYZ} and since then a large amount of papers has blossomed in this direction (see e.g.~\cite{AB_INITIO, Simis_cremona,KPU_blowup_fibers,EISENBUD_ULRICH_ROW_IDEALS,Hassanzadeh_Simis_Cremona_Sat,SIMIS_RUSSO_BIRAT,EFFECTIVE_BIGRAD,SIMIS_PAN_JONQUIERES,HASSANZADEH_SIMIS_DEGREES}).

\medskip

One of our main contributions is the introduction of a new algebra that we call the \textit{saturated special fiber ring} (see \autoref{saturated_special_fiber_ring}). 
The fact that saturation plays a key role in this kind of problems has been already observed in previous works in the single-graded setting (see \cite{Hassanzadeh_Simis_Cremona_Sat}, \cite[Theorem 3.1]{JEFFRIES_MONTANO_VARBARO}, \cite[Proposition 1.2]{PAN_RUSSO}). 
Based on that, we define this new algebra by taking certain multi-graded parts of the saturation of all the powers of the base ideal. It can be seen as an extension of the more classical special fiber ring. 
We show that this new algebra turns out to be a fundamental ingredient for computing the degree of a rational map (see \autoref{theorem_factorize_rat_map}), emphasizing in this way that saturation has actually a prominent role to control the degree and birationality of a rational map. 
In particular, it allows us to prove a new general numerical formula (see \autoref{NEW_CRITERION}) from which the degree of a rational map can be extracted.

\medskip

Next, we refine the above results for rational maps whose base locus is  zero-dimensional. We begin by providing a purely algebraic proof (see \autoref{thm_degree_formula_base_points}) of the classical degree formula in intersection theory (see e.g.~\cite[Section 4.4]{FULTON_INTERSECTION_THEORY}) adapted to our setting. 
Then, we investigate the properties that can be extracted if the syzygies of the base ideal of a rational map are used.
This idea amounts to use the symmetric algebra to approximate the Rees algebra and it allows us to obtain sharp upper bounds (see \autoref{degree_in_terms_of_Sym}) for the degree of multi-graded rational maps. 
Some applications to the more specific cases of multi-projective rational maps (see \autoref{birationality_implies_sat}) and projective rational maps (see \autoref{graded_case_degree_in_terms_of_Sym}) are also discussed, with the goal of providing efficient birationality criteria in the low-degree cases (see e.g.~\autoref{upper_bound_degree_P1_P1}, 
\autoref{thm_bigrad_map_1_n} and \autoref{P^2_case}).

\medskip

The problem of detecting whether a rational map is birational has attracted a lot of attention in the past thirty years. 
A typical example is the class of Cremona maps (i.e. birational) in the projective plane that have been studied extensively (see e.g.~\cite{ACBook}). 
Obviously, the computation of the degree of a rational map yields a way of testing its birationality. 
Nevertheless, this approach is not very efficient and various  techniques have been developed in order to improve it and to obtain finer properties of birational maps. 
Among these specific techniques, the Jacobian dual criterion introduced and fully developed in \cite{Simis_cremona,SIMIS_RUSSO_BIRAT,AB_INITIO} has its own interest. In \autoref{sec:jacdualcriterion} of this paper, we will extend the theory of the Jacobian dual criterion to the multi-graded setting (see \autoref{jacob_main_crit}) and will derive some consequences where the syzygies of the base ideal are used instead of the higher-order equations of the Rees algebra   (see \autoref{thm:syzygy-jacobiandual}). 

\medskip

We finish the paper with a study of the particular class of plane rational maps whose base ideal is saturated and has a syzygy of degree one. In this setting, we provide a very effective birationality criterion (see \autoref{thm:birat-mu1=1}) and a complete description of the equations of the associated Rees algebra (see \autoref{Rees_Eq_mu_1}). This latter result is of interest on its own because finding the equations of the Rees algebra is currently a very active research topic (see e.g.~\cite{HONG_SIMIS_VASC_ELIM,KPU_BIGRAD_STRUCT,KPU_GOR3,KPU_NORMAL_SCROLL,COX_REES_MU_1,VASC_EQ_REES}).

\medskip

\noindent \textbf{Acknowledgments.} All the examples in this paper were computed with the help of the computer algebra software \textit{Macaulay2} \cite{MACAULAY2}. A package containing some computational results is publicly available \cite{Yairon} and included in the last version of \textit{Macaulay2}. The three authors were supported by the European Union's Horizon 2020 research and innovation programme under the Marie Sk\l{}odowska-Curie grant agreement No. 675789. The last two authors also acknowledge financial support from the Spanish MINECO, through the research project MTM2015-65361-P, and  the ``Mar\'ia de Maeztu'' Programme for Units of Excellence in R\&D (MDM-2014-0445).

\section{The degree of a multi-graded  rational map}

In this section we focus on the degree of a rational map between an integral multi-projective variety and an integral projective variety. 
Our main tool is the introduction of a new algebra which is a saturated version of the special fiber ring. 
The study of this algebra yields an answer for the degree of a rational map.

Our main result here is \autoref{theorem_factorize_rat_map} where we show that the saturated special fiber ring (\autoref{saturated_special_fiber_ring}) carries very important information of a rational map. 
Another fundamental result is \autoref{NEW_CRITERION}, which is the main tool for making specific computations. 

\subsection{Preliminaries on multi-graded rational maps}
\label{subsection_prelim}

Let $\kk$ be a field, $X_1 \subset \PP^{r_1}_\kk, X_2 \subset \PP^{r_2}_\kk, \ldots, X_m \subset \PP^{r_m}_\kk$ and $Y \subset \PP^{s}_\kk$ be integral projective varieties over $\kk$. 
For $i = 1,\ldots,m$, the homogeneous coordinate ring of $X_i$ is denoted by $
A_i = \kk[\mathbf{x_i}]/\mathfrak{a}_i=\kk[\xvars{i}]/\mathfrak{a}_i,$ and  $S = \kk[y_0,y_1,\ldots,y_s]/\mathfrak{b}$ stands for the homogeneous coordinate ring of $Y.$ 
We set $R = A_1 \otimes_{\kk} A_2 \otimes_{\kk} \cdots \otimes_{\kk} A_m \cong \kk[\mathbf{x}]/\left(\mathfrak{a}_1, \mathfrak{a}_2, \ldots, \mathfrak{a}_m\right)$.
We also assume that the fiber product $X=X_1\times_{\kk} X_2 \times_{\kk} \cdots \times_{\kk} X_m$ is an integral scheme (a condition that is satisfied, for instance, when $\kk$ is algebraically closed; see e.g. \cite[Lemma 4.23]{GORTZ_WEDHORN}).

Since we will always work over the field $\kk$, for any two $\kk$-schemes $W_1$ and $W_2$ their fiber product $W_1\times_{\kk} W_2$ will simply be denoted by $W_1\times W_2$.
Similarly, for any $k \ge 1$, $\PP^k$ will denote the $k$-th dimensional projective space $\PP_\kk^k$ over $\kk$.

Let $\FF : X = X_1 \times X_2 \times \cdots \times X_m \dashrightarrow Y \subset \PP^s$ be a rational map defined by $s+1$ multi-homogeneous elements $\mathbf{f}=\{f_0, f_1, \ldots, f_s\} \subset R$ of the same multi-degree. 
We say that $\mathcal{F}$ is birational if it is dominant and it has an inverse rational map given by a tuple of rational maps 
$$
\mathcal{G}: Y  \dashrightarrow (X_1, X_2, \ldots, X_m).
$$
For $i=1,\ldots,m$, each rational map $Y \dashrightarrow X_i \subset \PP^{r_i}$ is defined by $r_i+1$ homogeneous forms $\mathbf{g}_i=\{g_{i,0}, g_{i,1},\ldots, g_{i,r_i}\} \subset S$ of the same degree.

\begin{definition}
	The degree of a dominant rational map $\FF: X = X_1 \times X_2 \times \cdots \times X_m \dashrightarrow Y$ is defined as $\deg(\mathcal{F})=[K(X):K(Y)]$, where $K(X)$ and $K(Y)$ represent the fields of rational functions of $X$ and $Y$, respectively.
\end{definition}

Now, we recall some basic facts about multi-graded rings; we refer  the reader to \cite{HYRY_MULTIGRAD}, \cite{GRADED_RINGS_II} and \cite{GRADED_RINGS_I} for more details.

The ring $R = A_1 \otimes_{\kk} A_2 \otimes_{\kk} \cdots \otimes_{\kk} A_m$ has a natural multi-grading given by
$$
R = \bigoplus_{(j_1,\ldots,j_m) \in \ZZ^m} {(A_1)}_{j_1} \otimes_{\kk} {(A_2)}_{j_2} \otimes_{\kk} \cdots \otimes_{\kk} {(A_m)}_{j_m}.
$$
Let $\nn$ be the multi-homogeneous irrelevant ideal of $R$, that is 
$$
\nn = \bigoplus_{j_1>0,\ldots,j_m>0} R_{j_1,\ldots,j_m}.
$$
Similarly to the single-graded case, we can define a multi-projective scheme from $R$.
The multi-projective scheme  $\multProj(R)$ is given by 
the set of all multi-homogeneous prime ideals in $R$ which do not contain $\nn$,
and its scheme structure is obtained by using multi-homogeneous localizations. 

For any vector $\mathbf{c}=(c_1,\ldots,c_m)$ of positive integers we can define the multi-Veronese subring 
$$
R^{(\mathbf{c})} = \bigoplus_{j =0}^\infty R_{j\cdot \mathbf{c}},
$$
which we see as a standard graded $\kk$-algebra.
The canonical injection $R^{(\mathbf{c})} \hookrightarrow R$ induces an isomorphism of schemes 
$\multProj(R) \xrightarrow{\cong} \Proj\left(R^{(\mathbf{c})}\right)$ (see e.g.~\cite[Exercise II.5.11, Exercise II.5.13]{HARTSHORNE}).
In particular, if we set $\Delta=(1,\ldots,1),$ then $\Proj(R^{(\Delta)})$ corresponds with the homogeneous coordinate ring of the image of $X$ under the Segre embedding 
$$
\PP^{r_1} \times \PP^{r_2} \times \cdots \times \PP^{r_m} \longrightarrow \PP^{N},
$$
with $N = (r_1+1)(r_2+1)\cdots(r_m+1)-1$.
Therefore, for any positive vector $\mathbf{c}$ we have the following isomorphisms 
\begin{equation}
	\label{isom_multProj_Proj}
	X \cong \multProj(R) \cong \Proj\big(R^{(\Delta)}\big) \cong \Proj\big(R^{(\mathbf{c})}\big).
\end{equation}

Given a multi-graded $R$-module $M,$ we get an associated quasi-coherent sheaf $\mathcal{M}=\widetilde{M}$ of $\mathcal{O}_X$-modules. 
We have the following relations between sheaf and local cohomology modules (see e.g.~\cite{HYRY_MULTIGRAD}, \cite[Appendix A4.1]{EISEN_COMM}):
\begin{enumerate}[(i)]
	\item  There is an exact sequence of multi-graded $R$-modules 
	\begin{equation}
		\label{relation_local_sheaf_cohom_zero}
		0 \rightarrow \HH_{\nn}^0(M) \rightarrow M \rightarrow \bigoplus_{\mathbf{n} \in \ZZ^m} \HH^0(X, \mathcal{M}(\mathbf{n})) \rightarrow \HH_{\nn}^1(M) \rightarrow 0.
	\end{equation}
	\item For $j\ge 1$, there is an isomorphism of multi-graded $R$-modules 
	\begin{equation}
		\label{relation_local_sheaf_cohom_positive}
		\HH_{\nn}^{j+1}(M) \cong \bigoplus_{\mathbf{n} \in \ZZ^m}\HH^j(X, \mathcal{M}(\mathbf{n})).
	\end{equation}
\end{enumerate}

Let $I$ be the multi-homogeneous ideal $I=\left(f_0,\ldots,f_s\right)$.
Since $\FF$ is a dominant rational map, the homogeneous coordinate ring $S$ is often called the special fiber ring in the literature. Using the canonical graded morphism associated to $\FF$
\begin{eqnarray*}
	\kk[y_0,\ldots,y_s]/\mathfrak{b} & \rightarrow & R \\
	y_i & \mapsto & f_i,
\end{eqnarray*}
we classically identify $S$ with the standard graded $\kk$-algebra $\kk[I_{\dd}]$, which can be decomposed as 
$$
S = \kk[I_{\dd}] = \bigoplus_{n=0}^{\infty} {\left[I^n\right]}_{n\cdot \dd}
$$
(see the next subsection for more details). For the rest of this section we shall assume the following. 
\begin{notation}
	\label{notations_section_2}
	Let  $\mathcal{F} : X = X_1 \times X_2 \times \cdots \times X_m \dashrightarrow Y \subset \PP^s$ be a dominant rational map defined by $s+1$ multi-homogeneous forms $\mathbf{f}=\{f_0, f_1, \ldots, f_s\} \subset R$ of the same multi-degree $\mathbf{d}=(d_1,\ldots,d_m)$.
 	Let $\delta_i$ be the dimension $\delta_i=\dim(X_i)$ of the projective variety $X_i$, and  $\delta=\delta_1+\cdots+\delta_m$  the dimension of $X$.
	Let $I \subset R$ be the multi-homogeneous ideal generated by $f_0,f_1,\ldots,f_s$.
	Let $S$ be the homogeneous coordinate ring $S = \kk[f_0,f_1,\ldots,f_s]=\kk[I_{\mathbf{d}}]$ of $Y$ and let $T$ be the multi-Veronese subring $T = R^{(\dd)}=\kk[R_{\dd}]$.
	After regrading, we regard $S \subset T$ as standard graded $\kk$-algebras.
\end{notation}

\subsection{The saturated special fiber ring}

In this section, we introduce and study an algebra which is the saturated version of the special fiber ring.  
For any ideal $J \subset R$, the saturation ideal $\left(J:\nn^{\infty}\right)$ with respect to $\nn$ will be written as $J^{\sat}$.

\begin{definition}
	\label{saturated_special_fiber_ring}
	The saturated special fiber ring of $I$ is the graded $S$-algebra
	$$
	\SSS = \bigoplus_{n=0}^{\infty} {\left[{(I^n)}^{\sat}\right]}_{n\cdot \dd}.
	$$
\end{definition}
Interestingly, the algebra $\SSS$ turns out to be finitely generated as an $S$-module  (\autoref{lem_S_hat_S_tilde}$(ii)$).

The central point of our approach is a comparison between the two algebras $S$  and $\SSS$. 
Perhaps surprisingly, we show that assuming the condition of $S$ being integrally closed, then $\FF$ is birational if and only if these two algebras coincide, and more generally, we show that the difference between them yields the degree of $\FF$. 
In addition, we also prove that the algebra $\SSS$ reduces the study of the rational map $\FF$ to the study of a finite morphism.

\begin{theorem}
	\label{theorem_factorize_rat_map}
	Let $\FF: X = X_1 \times X_2 \times \cdots \times X_m \dashrightarrow Y$ be a dominant rational map. If $\dim(Y)=\delta$, then we have the following commutative diagram
	\begin{center}		
		\begin{tikzpicture}
		\matrix (m) [matrix of math nodes,row sep=3em,column sep=4em,minimum width=2em, text height=1.5ex, text depth=0.25ex]
		{
			X & &\\
			\Proj(T) & & Y \\
			& \Proj(\SSS) & \\
		};
		\path[-stealth]
		(m-1-1) edge node [left]	{$\cong$} (m-2-1)
		(m-3-2) edge node [above]	{$\mathcal{H}$} (m-2-3)
		;				
		\draw[->,dashed] (m-1-1)--(m-2-3) node [midway,above] {$\FF$};	
		\draw[->,dashed] (m-2-1)--(m-2-3) node [midway,above] {$\FF^{\prime}$};	
		\draw[->,dashed] (m-2-1)--(m-3-2) node [midway,above] {$\mathcal{G}$};	
		\end{tikzpicture}	
	\end{center}
	where the maps $\FF^{\prime}:\Proj(T) \dashrightarrow Y$, $\mathcal{G}: \Proj(T) \dashrightarrow \Proj(\SSS)$ and $\mathcal{H}: \Proj(\SSS) \rightarrow Y$ are induced from the inclusions $S \hookrightarrow T$, $\SSS \hookrightarrow T$ and $S \hookrightarrow \SSS$, respectively.
	
	\noindent
	Also, the statements below are satisfied:
	\begin{enumerate}[(i)]
		\item $\mathcal{H}: \Proj(\SSS) \rightarrow Y$ is a finite morphism with $\deg(\FF)=\deg(\mathcal{H})$.
		\item $\mathcal{G}$ is a  birational  map.
		\item $e(\SSS) = \deg(\FF)\cdot e(S)$.
		\item 
		Under the additional condition of $S$ being integrally closed, then $\FF$ is birational if and only if $\SSS=S$.		
	\end{enumerate}
\end{theorem}

The rest of this section is dedicated to the proof of this theorem.
 Before, we need some intermediate results and definitions. 
 We begin with the following lemma that has its roots in a similar result for the single-graded case (see e.g.~\cite[Proof of Theorem 6.6]{Sim_Ulr_Vasc_mult}, \cite[Proof of Proposition 2.11]{AB_INITIO}, \cite[Remark 2.10]{KPU_blowup_fibers}).

\begin{lemma}
	\label{lem_deg_map_field_deg}
	 The degree of $\mathcal{F}$ is given by $\deg(\mathcal{F})=[T:S]$.
	\begin{proof}
		By definition we have that $\deg(\FF)=\left[K(X):K(Y)\right]=\left[\big(R^{(\dd)}\big)_{(0)} : S_{(0)}\right]$. Let $0 \neq f \in I_{\dd}.$ Then, we get 
		$$
		\Quot(S) = S_{(0)}(f) \qquad \text{and} \qquad \Quot(T)=\big(R^{(\dd)}\big)_{(0)}(f).
		$$
		Finally, since $f$ is transcendental over $\big(R^{(\dd)}\big)_{(0)}$ and $S_{(0)}$, then it follows that
		$$
		\deg(\FF)=\left[\big(R^{(\dd)}\big)_{(0)} : S_{(0)}\right]=\left[\big(R^{(\dd)}\big)_{(0)}(f) : S_{(0)}(f)\right]=[T:S],
		$$
		as claimed.
	\end{proof}
\end{lemma}

Now, we introduce a new multi-graded algebra $\AAA$ defined by $\AAA=R[y_0,y_1,\ldots,y_s]$.
By an abuse of notation, for any multi-homogeneous element $g \in \AAA$ we will write $\bideg(g)=(\mathbf{a},b)$ if $\mathbf{a} \in \ZZ^m$ corresponds with the multi-degree part  in $R$ and $b \in \ZZ$ with the degree part in $\kk[\mathbf{y}]$.
We give the multi-degrees $\bideg(x_i)=(\deg(x_i), 0)$, $\bideg(y_i)=(\mathbf{0},1)$, where $\mathbf{0} \in \ZZ^m$ denotes a vector $\mathbf{0}=(0,\ldots,0)$ of $m$ copies of $0$.

Given a multi-graded $\AAA$-module $M$ and a multi-degree vector $\cc \in \ZZ^m$, then $M_{\mathbf{c}}$ will denote the $\mathbf{c}$-th multi-graded part in $R$, that is 
$$
M_{\mathbf{c}} = \bigoplus_{n  \in \ZZ} M_{\mathbf{c},n}.
$$
We remark that $M_{\mathbf{c}}$ has a natural structure as a graded $\kk[\yy]$-module.

We can present the Rees algebra $\Rees(I)=\bigoplus_{n=0}^{\infty}I^nt^n \subset R[t]$ as a quotient of the multi-graded algebra $\AAA = R[y_0,y_1,\ldots,y_s]$ by means of the map
\begin{eqnarray}
	\label{presentation_Rees}
	\Psi: \AAA & \longrightarrow & \Rees(I) \subset R[t] \\ \nonumber
	 y_i & \mapsto & f_it.
\end{eqnarray}
We set $\bideg(t)=(-\mathbf{d}, 1)$, which implies that $\Psi$ is multi-homogeneous of degree zero. 
Thus, the multi-graded structure of $\Rees(I)$ is given by 
$$
\Rees(I) = \bigoplus_{\mathbf{c} \in \ZZ^m, n \in \ZZ} {\left[\Rees(I)\right]}_{\mathbf{c}, n} \quad \text{ and } \quad {\left[\Rees(I)\right]}_{\mathbf{c}, n} = {\left[I^n\right]}_{\mathbf{c} + n\cdot\mathbf{d}}.
$$
We shall denote ${\left[\Rees(I)\right]}_{\mathbf{c}}=\bigoplus_{n=0}^\infty {\left[\Rees(I)\right]}_{\mathbf{c},n}$, and of particular interest is the case ${\left[\Rees(I)\right]}_{\mathbf{0}}=\bigoplus_{n=0}^\infty {\left[I^n\right]}_{n\cdot\mathbf{d}}$.
We note that each local cohomology module $\HH_{\nn}^i(\Rees(I))$ has a natural structure of multi-graded $\AAA$-module (see e.g.~\cite[Lemma 2.1]{DMOD}), and also that ${\left[\HH_{\nn}^i(\Rees(I))\right]}_{\cc}=\bigoplus_{n \in \ZZ} {\left[\HH_{\nn}^i(I^n)\right]}_{\cc+n\cdot\mathbf{d}}$ has a natural structure as a graded $\kk[\mathbf{y}]$-module.
Let $\mathfrak{b}=\Ker(\kk[\mathbf{y}] \rightarrow \kk[\mathbf{f}t])$ be the kernel of the map
\begin{equation*}
	\label{mapping_y_to_ft}
	\kk[\mathbf{y}] \rightarrow \kk[\mathbf{f}t] \subset \Rees(I), \qquad y_i \mapsto f_it,
\end{equation*}
then we have that $S \cong \kk[\mathbf{y}]/\mathfrak{b}$ and that for any $h \in \mathfrak{b}$ the multiplication map ${\Rees(I)} \xrightarrow{\cdot h} {\Rees(I)}$ is zero.
So the induced map on local cohomology ${\left[\HH_{\nn}^i(\Rees(I))\right]}_{\cc} \xrightarrow{\cdot h} {\left[\HH_{\nn}^i(\Rees(I))\right]}_{\cc}$ is also zero for any $h \in \mathfrak{b}.$ This implies that ${\left[\HH_{\nn}^i(\Rees(I))\right]}_{\cc}$ has a natural structure as a graded $S$-module.

\begin{remark}
	The blow-up $\widetilde{X}=Bl_{I}(X)$ of $X$ along $V(I)$ is defined as the multi-projective scheme obtained by considering the Rees algebra $\Rees(I)$ as a multi-graded $\AAA$-algebra. We shall use the notation
	$$
	\widetilde{X} = \biProj(\Rees(I)) \subset X \times \PP^s,
	$$
	where $\widetilde{X}$ can be canonically embedded in $X \times \PP^s$.
	
	By considering the Rees algebra $\Rees(I)$ only as a multi-graded $R$-algebra, then we obtain a multi-projective scheme which is an ``affine version'' of the blow-up $\widetilde{X}$, and that we shall denote by
	$$
	\multProj_{R\text{-gr}}\left(\Rees(I)\right) \subset X \times \mathbb{A}^{s+1},
	$$ 
	where $\multProj_{R\text{-gr}}\left(\Rees(I)\right)$ can be canonically embedded in $X \times \mathbb{A}^{s+1}$.
\end{remark}

\begin{proposition}
	\label{prop_finite_dim_Hn_Rees}
	For each $i \ge 0$ and $\mathbf{c}=(c_1,\ldots,c_m)$, we have the following statements:
	\begin{enumerate}[(i)]
		\item 	${\left[\HH_{\nn}^i(\Rees(I))\right]}_{\cc}$ is a finitely generated graded $S$-module.
		\item $\HH^{i}\left(\multProj_{R\text{-gr}}({\Rees(I)}), \mathcal{O}_{\multProj_{R\text{-gr}}({\Rees(I)})}(\cc)\right)$ is a finitely generated graded $S$-module.
	\end{enumerate}
	\begin{proof}
		$(i)$ It is enough to prove that ${\left[\HH_{\nn}^i(\Rees(I))\right]}_{\cc}$ is a finitely generated $\kk[\mathbf{y}]$-module.
		Suppose that $F_{\bullet}: \cdots \rightarrow F_2 \rightarrow F_1 \rightarrow F_0 \rightarrow 0$ represents the minimal free resolution of $\Rees(I)$ as an $\AAA$-module.
		Let $\mathcal{C}_{\nn}^{\bullet}$ be the \v{C}ech complex with respect to the ideal $\nn.$ We consider the double complex $F_{\bullet} \otimes_R \mathcal{C}_{\nn}^{\bullet}$. 
		Computing cohomology by rows and then by columns, gives us a spectral sequence that collapses on the first column and the terms are equal to $\HH_{\nn}^i(\Rees(I))$.
		On the other hand, by computing cohomology via columns, we get the spectral sequence 
		$$
		E_1^{-p,q} = \HH_{\nn}^q(F_p) \Longrightarrow \HH_{\nn}^{q-p}(\Rees(I)).
		$$
		Since each ${\left[\HH_{\nn}^q(F_p)\right]}_{\cc}$ is a finitely generated (free) $\kk[\mathbf{y}]$-module, then it follows that ${\left[\HH_{\nn}^{q-p}(\Rees(I))\right]}_{\cc}$ is also finitely generated as a $\kk[\mathbf{y}]$-module.
		
		$(ii)$ When $i \ge 1$,  we get the result from \autoref{relation_local_sheaf_cohom_positive} and the previous part $(i)$.
		By \autoref{relation_local_sheaf_cohom_zero} we have the exact sequence 
		$$
		0 \rightarrow {\left[\Rees(I)\right]}_{\cc} \rightarrow \HH^{0}\left(\multProj_{R\text{-gr}}({\Rees(I)}), \mathcal{O}_{\multProj_{R\text{-gr}}({\Rees(I)})}(\cc)\right) \rightarrow {\left[\HH_{\nn}^1(\Rees(I))\right]}_{\cc} \rightarrow 0.
		$$
		The $S$-module ${\left[\Rees(I)\right]}_{\cc}$ is clearly finitely generated, and from part $(i)$ the $S$-module  ${\left[\HH_{\nn}^1(\Rees(I))\right]}_{\cc}$  is also finitely generated.
		Therefore, the exact sequence above gives us the result for $i=0$.
		\end{proof}
\end{proposition}

\begin{remark}
The degree of the generators of $\widetilde{S}$ can be effectively bounded by means of the argument given in \autoref{prop_finite_dim_Hn_Rees} and using the relations between $\widetilde{S}$ and $\widehat{S}$. Using this idea,  an algorithm to compute the saturated special fiber ring $\widetilde{S}$ has been designed and implemented by the second author \cite{Yairon}  in  a  package included in the latest version of {\it Macaulay2}   \cite{MACAULAY2}. It computes the saturated special fiber ring of an ideal in a multi-graded setting and provides new algorithms to compute the degree of a rational map and test birationality. 
\end{remark}

By \autoref{prop_finite_dim_Hn_Rees}, the multi-projective scheme $\multProj_{R\text{-gr}}\left(\Rees(I)\right)$ yields the following finitely generated $S$-algebra
$$
\widehat{S} := \HH^{0}\left(\multProj_{R\text{-gr}}({\Rees(I)}), \mathcal{O}_{\multProj_{R\text{-gr}}({\Rees(I)})}\right).
$$
We will see that $\widehat{S}$ carries the same information as $\widetilde{S}$, but has the advantage of having some geometrical content as the global sections of an ``affine version'' of the blow-up $\widetilde{X}$, a fact which is going to be fundamental in the proofs of \autoref{theorem_factorize_rat_map} and \autoref{NEW_CRITERION}.

Let $\{\vartheta_1,\ldots,\vartheta_r \}$ be a set of generators of $\nn$, 
then $\multProj_{R\text{-gr}}({\Rees(I)})$  has an affine open cover 
\begin{equation}
	\label{affine_cover_R_gr}
	\mathcal{U} = {\left(U_i\right)}_{i=1,\ldots,r},\qquad\qquad
	U_i = \Spec\left({\Rees(I)}_{(\vartheta_i)}\right),
\end{equation}
where ${\Rees(I)}_{(\vartheta_i)}$ denotes the graded $S$-module
$$
{\Rees(I)}_{(\vartheta_i)}= {\left[{\Rees(I)}_{\vartheta_i}\right]}_{\mathbf{0},*}
$$
obtained by restricting to multi-degree $\mathbf{0}$ in the multi-grading corresponding with $R$.
Since 
$$
\Rees(I)\;\cong\; \bigoplus_{n=0}^\infty {I^n}(n\cdot\dd),
$$
 then computing \v{C}ech cohomology with respect to $\mathcal{U}$ gives us the following equality
$$
\widehat{S} = \bigoplus_{n=0}^{\infty} \HH^0\left(X,{(I^n)}^\sim (n\cdot\dd)\right).
$$
In the next lemma, with some simple remarks, we show that $\SSS$ and $\widehat{S}$ are almost the same. 

\begin{lemma}
		\label{lem_S_hat_S_tilde}
		The following statements hold:
		\begin{enumerate}[(i)]
			\item There is an inclusion $\SSS \subset \widehat{S}$, which becomes an equality $\left[\SSS\right]_n=\left[\widehat{S}\right]_n$ for $n\gg 0$.
			\item $\SSS$ is a finitely generated $S$-module.
			\item The two algebras have the same multiplicity, that is $e(\SSS)=e(\widehat{S})$.
			\item $\Proj(\SSS) \cong \Proj(\widehat{S})$.
			\item If $\grade(\nn)\ge 2$, then $\SSS=\widehat{S}$.
		\end{enumerate}
		\begin{proof}
			$(i)$ 
			Since we have an isomorphism of sheaves ${\left(I^n\right)}^{\sim}(n\cdot\dd) \cong {\left({\left(I^n\right)}^{\sat}\right)}^{\sim}(n\cdot\dd)$, then from \autoref{relation_local_sheaf_cohom_zero} we get the short exact sequence
			\begin{equation}
				\label{comparison_S_hat_S_tilde1}
				0 \rightarrow {\left[{(I^n)}^{\sat}\right]}_{n\cdot\dd} \rightarrow \HH^0\left(X,{(I^n)}^\sim(n\cdot\dd)\right) \rightarrow {\left[\HH_{\nn}^1\left({(I^n)}^{\sat}\right)\right]}_{n\cdot\dd} \rightarrow 0
			\end{equation}
			for each $n\ge 1$.
			The short exact sequence $0 \rightarrow {(I^n)}^{\sat} \rightarrow R \rightarrow R/{(I^n)}^{\sat} \rightarrow 0$ yields the long exact sequence 
			\begin{equation}
				\label{comparison_S_hat_S_tilde2}
				{\left[\HH_{\nn}^0\left(R/{(I^n)}^{\sat}\right)\right]}_{n\cdot\dd} \rightarrow {\left[\HH_{\nn}^1\left({(I^n)}^{\sat}\right)\right]}_{n\cdot\dd} \rightarrow  {\left[\HH_{\nn}^1(R)\right]}_{n\cdot\dd} \rightarrow 			{\left[\HH_{\nn}^1\left(R/{(I^n)}^{\sat}\right)\right]}_{n\cdot\dd}.
			\end{equation}
			We always have that			$\HH_{\nn}^0\left(R/{(I^n)}^{\sat}\right)=0$, and that ${\left[\HH_{\nn}^1(R)\right]}_{n\cdot\dd}=0$ for $n\gg0$ (see e.g. \cite[Lemma 4.2]{Kleiman_geom_mult}).
			Hence, we get that ${\left[\HH_{\nn}^1\left({(I^n)}^{\sat}\right)\right]}_{n\cdot\dd}=0$ for $n\gg 0$, and so the result follows. 
			
			$(ii)$ Straightforward from part $(i)$ and \autoref{prop_finite_dim_Hn_Rees}$(ii)$.
			
			$(iii)$ It is clear from part $(i)$.
			
			$(iv)$ Follows from part $(i)$ (see e.g.~\cite[Exercise II.2.14]{HARTSHORNE}).
			
			$(v)$ The condition $\grade(\nn)\ge 2$ implies that $\HH_{\nn}^1(R)=0$, then the required equality is obtained from \autoref{comparison_S_hat_S_tilde1}	and \autoref{comparison_S_hat_S_tilde2}.
		\end{proof}
\end{lemma}

The next example shows that the condition of $S$ being integrally closed cannot be dropped in \autoref{theorem_factorize_rat_map}$(iv)$.
	
\begin{example}
		Let $R = \kk[x_0,x_1]$ and consider the parametrization of the rational quartic 
		$$
		\FF: \PP^1 \rightarrow \PP^3, \quad (x_0:x_1) \rightarrow (x_0^4: x_0^3x_1: x_0x_1^3: x_1^4).
		$$ 
		The homogeneous coordinate ring of the rational quartic, i.e.~the image of $\FF$, is 
		$$
		S \cong \frac{\kk[y_0,y_1,y_2,y_3]}
		{\left({y}_{1} {y}_{2}-{y}_{0} {y}_{3},{y}_{2}^{3}-{y}_{1} {y}_{3}^{2},{y}_{0} {y}_{2}^{2}-{y}_{1}^{2} {y}_{3},{y}_{1}^{3}-{y}_{0}^{2} {y}_{2}
			\right)}.
		$$
		It is known that $S$ is not normal (i.e, the rational quartic curve is not projectively normal; see e.g. \cite[Exercise I.3.18]{HARTSHORNE}, \cite[Exercise 18.8]{EISEN_COMM}).
		It then follows that the condition of $S$ being integrally closed cannot be dropped in \autoref{theorem_factorize_rat_map}$(iv)$. Indeed, 
		by setting $I=(x_0^4, x_0^3x_1, x_0x_1^3, x_1^4)$, one notices that $I^{\sat} = R$ and so one obtains that $\left[I^{\sat}\right]_4 \neq \left[I\right]_4$ whereas $\FF$ is birational onto its image $\Proj(S)$.
		
\end{example}

	All the rational maps that are considered in specific applications are usually such that $\grade(\nn)\ge 2$.
	So, in practice, we always have $\SSS=\widehat{S}$. Nevertheless, we give an example where $\SSS$ and $\widehat{S}$ are different.
	
	\begin{example} Take again the rational quartic curve $\Cc\subset \PP^3$ and consider the map
	$$
		\FF : \mathcal{C}  \rightarrow  \PP^2, \quad 
		(x_0:x_1:x_2:x_3)  \mapsto  \left(x_0:x_1:x_3\right).			
	$$
	In order to be consistent with our general notation in the section, we now denote the homogeneous coordinate ring of $\Cc$ by 
		$$
		R \cong \frac{\kk[x_0,x_1,x_2,x_3]}
		{\left({x}_{1} {x}_{2}-{x}_{0} {x}_{3},{x}_{2}^{3}-{x}_{1} {x}_{3}^{2},{x}_{0} {x}_{2}^{2}-{x}_{1}^{2} {x}_{3},{x}_{1}^{3}-{x}_{0}^{2} {x}_{2}
			\right)}
		$$
		and consider the ideal $I=\left(x_0,x_1,x_3\right) \subset R$.
%
		
		Observe that $R$ is not Cohen-Macaulay with $\dim(R)=2$ and $\depth(R)=1$. Moreover, by setting  $B=\kk[x_0,x_1,x_2,x_3]$ and $\mm=\left(x_0,x_1,x_2,x_3\right) \subset B$, computations with Macaulay2 \cite{MACAULAY2} 
%
show that
		$$
		\Ext_B^3\big(R, B(-4)\big) \cong \big(B/\mm\big)(1),
		$$
		and so the graded local duality theorem (see e.g.~\cite[Theorem 3.6.19]{BRUNS_HERZOG}) yields the isomorphism
		$$
		\HH_{\mm}^1(R) \cong \big(B/\mm\big)(-1).
		$$
		Since we have $I^{\sat} = R$, 
		it follows that
		$$
		{\left[\HH_{\mm R}^1\left(I^{\sat}\right)\right]}_1\cong{\left[\HH_{\mm}^1\left(I^{\sat}\right)\right]}_1={\left[\HH_{\mm}^1(R)\right]}_1 \cong {\left[\big(B/\mm\big)(-1)\right]}_1\cong\kk\neq 0.
		$$
		Therefore, from \autoref{comparison_S_hat_S_tilde1} we deduce that $\left[\SSS\right]_1 \neq \left[\widehat{S}\right]_1$.
	\end{example}

We are now ready to give the proof of the main result of this section.

\begin{proof}[Proof of \autoref{theorem_factorize_rat_map}]
		From our previous discussions \autoref{isom_multProj_Proj}, we have the following commutative diagram
		\begin{center}		
			\begin{tikzpicture}
			\matrix (m) [matrix of math nodes,row sep=3em,column sep=4em,minimum width=2em, text height=1.5ex, text depth=0.25ex]
			{
				X &   \\
				\Proj(T) & Y. \\
			};
			\path[-stealth]
			(m-1-1) edge node [left]	{$\cong$} (m-2-1)
			;				
			\draw[->,dashed] (m-1-1)--(m-2-2) node [midway,above] {$\FF$};	
			\draw[->,dashed] (m-2-1)--(m-2-2) node [midway,below] {$\FF^{\prime}$};	
			\end{tikzpicture}	
		\end{center}
		The rational map $\FF^{\prime}$ can be given by the tuple $(f_0:f_1:\cdots:f_s)$ because each $f_i \in T$. 
		From a geometrical point of view, here we are embedding $X$ in the ``right'' projective space 
		$$
		\PP^{M} \quad\text{ of dimension }\quad M = \prod_{i=1}^m\binom{r_i+d_i}{r_i}-1,
		$$
		where the $f_i$'s are actually linear forms. 
		Then the rational map $\FF^{\prime}:\Proj(T) \dashrightarrow Y=\Proj(S)$ is induced from the canonical inclusion $S \hookrightarrow T$.
		Since we have the canonical inclusions 
		$
		S \hookrightarrow \SSS \hookrightarrow T
		$
		then $\FF^{\prime}$ is given by the composition of the rational maps 
		$$
		\Proj(T) \dashrightarrow \Proj(\SSS) \dashrightarrow \Proj(S).
		$$
		From the condition $\dim(X)= \dim(Y),$ we have that 
		$$
		\Quot(S) \subset \Quot(\SSS) \subset \Quot(T)
		$$
		are algebraic extensions.
		Therefore, by using the same argument of \autoref{lem_deg_map_field_deg}, we get the equalities $\deg(\FF)=\left[T:S\right]$, $\deg(\mathcal{G})=[T:\SSS]$ and $\deg(\mathcal{H})=[\SSS:S]$.
		
		$(i)$
		First we check that the rational map $\mathcal{H}: \Proj(\SSS) \dashrightarrow Y$ is actually a finite morphism.		
		Since $\SSS$ is a finitely generated $S$-module (\autoref{lem_S_hat_S_tilde}$(ii)$),  we even have that $\SSS$ is integral over $S$.
		By the Incomparability theorem (see e.g.~\cite[Corollary 4.18]{EISEN_COMM}), the inclusion $S \hookrightarrow \SSS$ induces a (well defined everywhere) morphism 
		$$
		\mathcal{H}: \Proj(\SSS) \rightarrow \Proj(S).
		$$
		Indeed, for any $\qqq \subsetneq \SSS_+$ we necessarily have that $\qqq \cap S \subsetneq S_{+}$.
		The finiteness  of $\SSS$ as an $S$-module yields that $\mathcal{H}$ is a finite morphism.	
		
		Next we will prove that $\deg(\mathcal{H})=\deg(\FF)$.
		Let us denote by $\widetilde{X}=\biProj(\Rees(I))$ the blow-up of $X$ along $\BB=V(I)$, which can also be seen as the closure of the graph of $\FF$. 
		We then have the commutative diagram 
		\begin{center}		
			\begin{tikzpicture}
			\matrix (m) [matrix of math nodes,row sep=3em,column sep=4em,minimum width=2em, text height=1.5ex, text depth=0.25ex]
			{
				\widetilde{X} \subset X \times \PP^s &  \\
				X &  Y\subset \PP^s\\
			};			
			\draw[->] (m-1-1)--(m-2-1) node [midway,left] {$\pi_1$};	
			\draw[->,dashed] (m-2-1)--(m-2-2) node [midway,above] {$\FF$};	
			\draw[->] (m-1-1)--(m-2-2) node [midway,above] {$\pi_2$};	
			\end{tikzpicture}	
		\end{center}
		with $\pi_1$ being an isomorphism outside $\BB$ (see e.g.~\cite[Proposition 2.3]{Laurent_Jouanolou_Closed_Image}, \cite[Proposition II.7.13]{HARTSHORNE}). 		

		Let $\xi$ be the generic point of $Y$ and consider the fiber product  $W:=\widetilde{X} \times_{Y} \Spec\left(\OO_{Y,\xi}\right)$.
		Denoting $\eta$ as the generic point of $\widetilde{X}$, since $\pi_2$ is assumed to be generically finite, then we have the isomorphism $\Spec\left(\OO_{\widetilde{X},\eta}\right) \cong W$; this is a classical result, for a detailed proof see \cite[\href{https://stacks.math.columbia.edu/tag/02NV}{Tag 02NV}]{stacks-project}. 
		
		Even though $W$ is just a point, we will consider a convenient (and trivial) affine open cover of it. 
		The scheme $Y$ has an affine open cover given by $Y_j=\Spec\left(S_{(y_j)}\right)$ for $j=0,\ldots,s$.
		The open set $\pi_2^{-1}(Y_j)$ is isomorphic to $\multProj_{R\text{-gr}}\left({\Rees(I)}_{(y_j)}\right)$, where ${\Rees(I)}_{(y_j)}$ denotes the multi-graded $R$-module 
		$$
		{\Rees(I)}_{(y_j)}={\left[{\Rees(I)}_{y_j}\right]}_{\mathbf{*},0}
		$$
		defined by restricting to elements of degree zero in the grading corresponding with $S$.
		Then $W$ can be obtained by glueing the open cover 
		$$
		W_j \;:=\;\multProj_{R\text{-gr}}\left({\Rees(I)}_{(y_j)}\right)\; \times_{\Spec\left(S_{(y_j)}\right)}\; \Spec\left(\OO_{Y,\xi}\right)
		$$
		for $j=0,\ldots,s$.
		
		Fix $0 \le j \le s$.
		Similarly to \autoref{affine_cover_R_gr}, the scheme $\multProj_{R\text{-gr}}\left({\Rees(I)}_{(y_j)}\right)$ has an affine open cover 
		$$ {\left(\Spec\left({\Rees(I)}_{(\vartheta_iy_j)}\right)\right)}_{i=1,\ldots,r},
		$$
		where we are using the similar notation ${\Rees(I)}_{(\vartheta_iy_j)}={\left[{\Rees(I)}_{\vartheta_i y_j}\right]}_{\mathbf{0},0}$.
		Now, $W_j$ is obtained by glueing the affine open cover given by
		\begin{equation}
			\label{affine_cover_Rees_full_grad}
			{\Big(\Spec\left({\Rees(I)}_{(\vartheta_iy_j)} \; \otimes_{S_{(y_j)}}\; \OO_{Y,\xi} \right)\Big)}_{i=1,\ldots,r}.
		\end{equation}
		Since $\OO_{Y,\xi}\cong S_{(0)}$, then we have that the ring ${\Rees(I)}_{(\vartheta_iy_j)} \; \otimes_{S_{(y_j)}}\; \OO_{Y,\xi}$ does not depend on $j$. 
		Therefore we obtain that $W=W_j$. 
				
		Let $K$ be the multiplicative set of homogeneous elements of $S$ and $B$ 	be the localized ring $B=K^{-1}S$.
		We denote by $\mathcal{W}={\left(\mathcal{W}_i\right)}_{i=1,\ldots,r}$ the affine open cover of \autoref{affine_cover_Rees_full_grad}. Since we have the following isomorphisms of multi-graded $\AAA$-modules
		$$
		{\Rees(I)}_{\vartheta_{i_1}\vartheta_{i_2}\cdots\vartheta_{i_t} y_j}  \otimes_{S_{y_j}} B \;\cong\; 				{\Rees(I)}_{\vartheta_{i_1}\vartheta_{i_2}\cdots\vartheta_{i_t}} \otimes_{S} B \qquad \text{ for any } 1\le i_1 <i_2<\cdots<i_t \le r,
		$$
		 the corresponding \v{C}ech complex is given by
		$$
		C^{\bullet}(\mathcal{W}):\qquad0 \rightarrow \bigoplus_i {\Rees(I)}_{\vartheta_i} \otimes_S B \rightarrow \bigoplus_{i<k} {\Rees(I)}_{\vartheta_i\vartheta_k} \otimes_S B \rightarrow \cdots \rightarrow {\Rees(I)}_{\vartheta_1\cdots\vartheta_r} \otimes_S B \rightarrow 0.
		$$
		Using the affine open cover  \autoref{affine_cover_R_gr} of $\multProj_{R\text{-gr}}({\Rees(I)})$, we get the similar \v{C}ech complex
		$$
		C^{\bullet}(\mathcal{U}):\qquad0 \rightarrow \bigoplus_i {\Rees(I)}_{\vartheta_i}  \rightarrow \bigoplus_{i<k} {\Rees(I)}_{\vartheta_i\vartheta_k}  \rightarrow \cdots \rightarrow {\Rees(I)}_{\vartheta_1\cdots\vartheta_r}  \rightarrow 0.
		$$
		
		Since $B$ is flat over $S$, we get the isomorphism of multi-graded $\AAA$-modules
		$$
		\HH^0(C^\bullet(\mathcal{U})) \otimes_S B \;\cong\; \HH^0(C^{\bullet}(\mathcal{W})),
		$$
		and restricting to the multi-degree $\mathbf{0}$ part in $R$, we get the following isomorphisms of graded $S$-modules
		\begin{align*}
			\widehat{S} \otimes_S B &= \HH^{0}\left(\multProj_{R\text{-gr}}({\Rees(I)}), \mathcal{O}_{\multProj_{R\text{-gr}}({\Rees(I)})}\right) \otimes_S B \\
			&\cong {\left[\HH^0(C^\bullet(\mathcal{U}))\right]}_{\mathbf{0}} \otimes_S B \;\cong\; {\left[\HH^0(C^{\bullet}(\mathcal{W}))\right]}_{\mathbf{0}}. 
		\end{align*}
		From the fact that $S \hookrightarrow \widehat{S}$ is an algebraic extension, we have $\Quot(\widehat{S}) = \widehat{S} \otimes_S \Quot(S).$ So,  by restricting to the degree zero part, we get the following isomorphisms of rings
		$$
		\widehat{S}_{(0)} = {\left[\widehat{S} \otimes_S B\right]}_0 \cong {\Big[{\left[\HH^0(C^\bullet(\mathcal{U}))\right]}_{\mathbf{0}}\otimes_S B\Big]}_{0} \cong {\left[ \HH^0(C^{\bullet}(\mathcal{W}))\right]}_{\mathbf{0},0} \cong \HH^0(W,\OO_W) = \OO_{\widetilde{X},\eta}.
		$$
		
		Finally, since $\pi_1$ is a birational morphism and $\SSS_{(0)}=\widehat{S}_{(0)}$ (\autoref{lem_S_hat_S_tilde}$(iv)$), we obtain
		$$
		\deg(\mathcal{H})=\left[\SSS_{(0)}:S_{(0)}\right]=\left[\widehat{S}_{(0)}:S_{(0)}\right] = \left[\OO_{\widetilde{X},\eta}:\OO_{Y,\xi}\right] = \deg(\pi_2) = \deg(\FF).
		$$
		
		$(ii)$ From part $(i)$ we have $\deg(\FF)=\deg(\mathcal{H}).$ Then, the equality $\deg(\FF) = \deg(\mathcal{G})\deg(\mathcal{H})$ gives us that $\deg(\mathcal{G})=1$.
		
		$(iii)$ From the associative formula for multiplicity \cite[Corollary 4.6.9]{BRUNS_HERZOG} and part $(i)$, we get 
		$$
		e(\SSS) = \left[\SSS:S\right]e(S) = \deg(\FF)e(S).
		$$
		
		$(iv)$ We only need prove that assuming the birationality of $\FF$ and that $S$ is integrally closed, then we get $\SSS=S$.
		The equality $\deg(\FF)=\deg(\mathcal{H})=\left[\SSS:S\right]$ and the birationality of $\FF$ imply that 
		$\Quot(\SSS)=\Quot(S)$.
		Therefore we have the following canonical inclusions
		$$
		S \subset \SSS \subset \Quot(\SSS) = \Quot(S),
		$$ 
		and so from the fact that $S$ is integrally closed and that $S \hookrightarrow \SSS$ is an integral extension, we  obtain $\SSS=S$.
\end{proof}

We end this subsection by providing a relation between the $j$-multiplicity of an ideal and the multiplicity of the corresponding saturated special fiber ring.
The $j$-multiplicity of an ideal was introduced in \cite{ACHILLES_MANARESI_J_MULT}. It serves as a generalization of the Hilbert-Samuel multiplicity, and has applications in intersection theory (see e.g.~\cite{FLENNER_O_CARROLL_VOGEL}).

Let $A$ be a standard graded $\kk$-algebra of dimension $\delta+1$ which is an integral domain.
Let $\nnn$ be its maximal irrelevant ideal $\nnn=A_{+}$.
For a non necessarily $\nnn$-primary ideal $J\subset A$ its $j$-multiplicity is given by
$$
j(J) := \delta
!\lim\limits_{n\rightarrow\infty} \frac{\dim_{\kk}\left(\HH_{\nnn}^0\left(J^n/J^{n+1}\right)\right)}{n^\delta}.
$$

\begin{lemma}
	\label{lem_j_mult}
	Let $J\subset A$ be a homogeneous ideal equally generated in degree $d$.
	Suppose $J$ has maximal analytic spread $\ell(J)=\delta+1$.
	Then, we have the equality
	$$
	j(J) = d\cdot e(D),
	$$
	where $D = \bigoplus_{n=0}^{\infty} {\left[{(J^n)}^{\sat}\right]}_{nd}$  is the saturated special fiber ring of $J$.
	\begin{proof}
		We consider the associated dominant rational map $\mathcal{G}: \Proj\left(A\right) \dashrightarrow \Proj\left(\kk[J_d]\right)$, that satisfies $\dim(A)=\dim(\kk[J_d])$ because $\ell(J)=\delta+1$.
		From \cite[Theorem 5.3]{KPU_blowup_fibers} and \autoref{theorem_factorize_rat_map}$(iii)$ we obtain 
		$$
		j(J)=d\cdot\deg(\mathcal{G})\cdot e\left(\kk[J_d]\right) \quad\text{ and }\quad e(D)=\deg(\mathcal{G})\cdot e\left(\kk[J_d]\right),
		$$ 
		respectively.
		So the result follows by comparing both equations.
	\end{proof}
\end{lemma}

As a direct consequence of this lemma we obtain a refined version of \cite[Theorem 3.1$(iii)$]{JEFFRIES_MONTANO_VARBARO}.

\begin{corollary}
		Let $J\subset A$ be a homogeneous ideal equally generated in degree $d$.
		Suppose $J$ has maximal analytic spread $\ell(J)=\delta+1$.
		If ${\left[{(J^n)}^{\sat}\right]}_{nd}={\left[J^n\right]}_{nd}$ for all $n \gg 0$, then $$
		j(J)=d\cdot e(\kk[J_d]).
		$$
\end{corollary}

\begin{remark}
	To emphasize the relevance of \autoref{theorem_factorize_rat_map},  we mention that recently the multiplicity of the saturated special fiber ring has been computed explicitly for some classes of rational maps, illustrating the fact that \autoref{theorem_factorize_rat_map} provides a much better understanding and control of the degree and birationality of those rational maps.
	\begin{itemize}
		\item In \cite{SAT_FIB_PERF_HT_2}, the case of rational maps with \textit{perfect base ideals of height two} was treated.
		\item In \cite{MULT_GOR_HT_3}, the case of rational maps with \textit{Gorenstein base ideals of height three} was treated.
	\end{itemize}
\end{remark}

\subsection{Formula for the degree of multi-graded rational maps}

In this subsection, we prove a new formula that relates the degree of $\FF$ with the multiplicity of the $S$-module ${\left[\HH_{\nn}^1(\Rees(I))\right]}_{\mathbf{0}}$ and the degree of the image. This result will be our main tool for making specific computations. To state it, we will need the following additional notation: for any finitely generated graded $S$-module $N$, the $(\delta+1)$-th multiplicity is defined by (see e.g.~\cite[\S 4.7]{BRUNS_HERZOG})
$$
e_{\delta+1}(N) = \begin{cases}
	e(N) \qquad \text{if } \dim(N) = \delta +1,\\
	0 \quad\;\;\qquad  \text{otherwise.}
\end{cases}
$$

\begin{corollary}
		\label{NEW_CRITERION}
		Let $\FF: X = X_1 \times X_2 \times \cdots \times X_m \dashrightarrow Y$ be a dominant rational map. 
		If $\dim(Y)=\delta$, then the degree of $\FF$ can be computed by 
		$$
		\deg_{\PP^s}(Y)(\deg(\FF)-1) =e_{\delta+1}\Big({\left[\HH_{\nn}^1(\Rees(I))\right]}_{\mathbf{0}}\Big) =  \delta!\lim\limits_{n\rightarrow \infty} \frac{\dim_{\kk}\big({\left[\HH_{\nn}^1(I^n)\right]}_{n\cdot\mathbf{d}}\big)}{n^\delta}.
		$$
		In particular, we have that $\FF$ is birational if and only if $\dim_S\left({\left[\HH_{\nn}^1\left(\Rees(I)\right)\right]}_{\mathbf{0}}\right) < \delta+1$.
		\begin{proof}
			From \autoref{relation_local_sheaf_cohom_zero} we have the exact sequence 
			$$
			0 \rightarrow {\left[\Rees(I)\right]}_{\mathbf{0}} \rightarrow \HH^{0}\left(\multProj_{R\text{-gr}}({\Rees(I)}), \mathcal{O}_{\multProj_{R\text{-gr}}({\Rees(I)})}\right) \rightarrow {\left[\HH_{\nn}^1(\Rees(I))\right]}_{\mathbf{0}} \rightarrow 0			
			$$
			which by using our previous notations can be written as 
			$$
			0 \rightarrow S \rightarrow \widehat{S} \rightarrow {\left[\HH_{\nn}^1(\Rees(I))\right]}_{\mathbf{0}} \rightarrow 0.			
			$$
			We clearly have  $e_{\delta+1}(S)=\deg_{\PP^s}(Y)$, then it follows that
			\begin{align*}
				e_{\delta+1}(\widehat{S}) &= e_{\delta+1}(\SSS)  \qquad\qquad\qquad\text{(by \autoref{lem_S_hat_S_tilde}$(iii)$)}\\
				&= \deg(\FF)\cdot e_{\delta+1}(S)  \qquad\text{\,(by \autoref{theorem_factorize_rat_map}$(iii)$)}\\
				&= \deg(\FF)\cdot\deg_{\PP^s}(Y). 
			\end{align*}
			Therefore, the previous exact sequence yields the equality
			$$
			e_{\delta+1}\Big({\left[\HH_{\nn}^1(\Rees(I))\right]}_{\mathbf{0}}\Big) = e_{\delta+1}(\widehat{S}) - e_{\delta+1}(S) = \deg_{\PP^s}(Y)(\deg(\FF)-1),
			$$
			as claimed.
		\end{proof}
\end{corollary}

\begin{remark}
		Let $J$ be an ideal in the polynomial ring $\kk[x_1,\ldots,x_p],$ and $\mm$ the maximal irrelevant ideal $(x_1,\ldots,x_p)$. In  \cite{Cutkosky_Ha_Asymp}, it was shown that the limit 
		$$
		\lim\limits_{n \rightarrow \infty} \frac{\lambda(\HH_{\mm}^1(J^n))}{n^p} = 		\lim\limits_{n \rightarrow \infty} \frac{\lambda(\HH_{\mm}^0(R/J^n))}{n^p}
		$$
		always exists under the assumption that $\kk$ is a field of characteristic zero, but, interestingly, it is proved that it is not necessarily a rational number.
		Later, in \cite{Herzog_hilb_polynom} it was obtained that when $J$ is a monomial ideal this limit is a rational number. 
		From the previous \autoref{NEW_CRITERION} we have that a similar limit obtained by restricting to certain graded strands, is always rational and also can give  valuable information for a (multi-graded) rational map. 
\end{remark}

\section{Rational maps with zero-dimensional base locus}

In this section we restrict ourselves to the case where the base locus $\BB=V(I)$ has dimension zero, i.e. that $\BB$ is finite over $\kk$.
In this case, we obtain four main different lines of results, that we gather in four subsections.
Firstly, in
 \autoref{subsection_deg_formula}, we provide an algebraic proof of the degree formula in the general multi-graded case. 
Then, in \autoref{subsection_Sym_algebra}, we derive bounds for the degree of a rational map from \autoref{NEW_CRITERION}, in terms of the symmetric algebra.
Thirdly, in \autoref{subsection_mult_proj_spaces}, we apply our methods in the case of rational maps defined over multi-projective spaces.
And we conclude by providing an upper bound for the degree of a single-graded rational map in terms of certain values of the Hilbert function of the base ideal in \autoref{subsection_single_grad_rat_maps}.

We shall see that these upper bounds are sharp in some cases, and also that we obtain new effective birationality criteria under certain conditions.

\subsection{The degree formula}
\label{subsection_deg_formula}

We give a formula for the degree of a multi-graded rational map, which depends on the degrees of the source and the image, and the multiplicity of the base points. 
This known formula can also be obtained with more geometric techniques (see \cite[Section 4.4]{FULTON_INTERSECTION_THEORY}).
It can be seen as a generalization of the same result in the single-graded case (see \cite[Theorem 2.5]{Laurent_Jouanolou_Closed_Image} and \cite[Theorem 6.6]{Sim_Ulr_Vasc_mult}). Hereafter we use the same notations and conventions of \autoref{subsection_prelim}. We begin with two preliminary results.

\begin{proposition}\label{prop:dimvanish}
	Assume that $\FF:X=X_1\times\cdots\times X_m \dashrightarrow Y$ is generically finite. Then, we have that  $\dim_S\big({\left[\HH_{\nn}^i(\Rees(I))\right]}_{\mathbf{0}}\big) < \dim(S)$  for all $i \ge 2$.
	\begin{proof}
			We have defined $\multProj_{R\text{-gr}}(\Rees(I))$ by considering $\Rees(I)$ as a multi-graded $R$-algebra, and so we have the following morphisms
			\begin{align*}
				\pi_2:&\; \biProj(\Rees(I)) \subset X \times \PP^s \longrightarrow \Proj(S) \subset \PP^s\\
				v:&\; \multProj_{R\text{-gr}}(\Rees(I)) \subset X \times \mathbb{A}^{s+1} \longrightarrow\Spec(S) \subset \mathbb{A}^{s+1}
			\end{align*}
			where both $\pi_2$ and $v$ are determined by the inclusion $S=\kk[\mathbf{y}]/\mathfrak{b} \hookrightarrow \Rees(I)$ that sends $y_i$ into $f_it$, and the only difference  consists on whether we take into account the grading in $\mathbf{y}$ or not.
			Therefore, we have that $v$ is also generically finite, and there exists some $0 \neq L \in S$ for which the morphism
			$$
			v_L : \multProj_{R\text{-gr}}({\Rees(I)}_L) \rightarrow \Spec(S_L)
			$$
			is finite (see \cite[Exercise II.3.7]{HARTSHORNE}).
			Thus, it follows that $\multProj_{R\text{-gr}}({\Rees(I)}_L)$ is affine (see \cite[Exercise II.5.17]{HARTSHORNE}).
			
			From the vanishing of sheaf cohomology (see  \cite[III.3]{HARTSHORNE}) and \autoref{relation_local_sheaf_cohom_positive}, we get 	
			$$
			{\big({\left[\HH_{\nn}^i(\Rees(I))\right]}_{\mathbf{0}}\big)}_{L} \cong {\left[{\HH_{\nn}^i\big({\Rees(I)}_{L}\big)}\right]}_{\mathbf{0}} \cong \HH^{i-1}\left(\multProj_{R\text{-gr}}({\Rees(I)}_L), \mathcal{O}_{\multProj_{R\text{-gr}}({\Rees(I)}_L)}\right) = 0
			$$
			for all $i\ge 2$.
			Since ${\left[\HH_{\nn}^i(\Rees(I))\right]}_{\mathbf{0}}$ is a finitely generated graded $S$-module then it is annihilated by some power of $L$, and the claimed result follows.
	\end{proof}
\end{proposition}

We define the degree of $X$ as the degree of its corresponding projectively embedded variety in $\PP^N$ by means of the Segre embedding. We have the following relation between the degree of $X$ and the degrees of the projective varieties $X_i \subset \PP^{r_i}$, $i=1,\ldots,m$.  

\begin{lemma}
	\label{lem_degree_Segre}
	The degree of $X=X_1\times\cdots\times X_m$ can be computed as
	$$
	\deg_{\PP^N}(X) = \frac{\delta!}{\delta_1!\delta_2!\cdots \delta_m!} \deg_{\PP^{r_1}}(X_1)\deg_{\PP^{r_2}}(X_2)\cdots \deg_{\PP^{r_m}}(X_m).
	$$
	\begin{proof}
		Since  the homogeneous coordinate ring of the image of $X$ in the Segre embedding is given by $R^{(\Delta)}$, then we have the following equality 
		$$
		P_{R^{(\Delta)}}(t) = P_{A_1}(t)P_{A_2}(t)\cdots P_{A_m}(t)
		$$
		between the Hilbert polynomials of the standard graded $\kk$-algebras $A_1,\ldots,A_m$ and $R^{(\Delta)}$.
		By comparing the leading terms of both sides of the equation we get the claimed result.
	\end{proof}
\end{lemma}

Under the present condition $\dim(\BB)=0$, we define the multiplicity of $\BB$ in $X$ by the following formula
\begin{equation}
\label{mult_base_points}
e(\BB) := \delta!\lim\limits_{n\rightarrow \infty} \frac{\dim_{\kk}\Big(\HH^0\Big(X, \mathcal{O}_X/(I^n)^\sim\Big)\Big)}{n^\delta}.
\end{equation}
Since we have the equalities 
\begin{align*}
\dim_{\kk}\Big(\HH^0\Big(X, \mathcal{O}_X/(I^n)^\sim\Big)\Big)  &= \sum_{\pp \in \BB} \dim_{\kk}\Big({\big(\mathcal{O}_X/(I^n)^\sim\big)}_{\pp}\Big) \\
&=\sum_{\pp \in \BB} \left[\kk(\pp):\kk\right]\cdot \text{length}_{\mathcal{O}_{X,\pp}}\Big({\big(\mathcal{O}_X/(I^n)^\sim\big)}_{\pp}\Big) \\
&= \sum_{\pp \in \BB} \left[\kk(\pp):\kk\right]\cdot \text{length}_{R_{\pp}}\Big(R_{\pp}/I_{\pp}^n\Big),
\end{align*}
the expression $\dim_{\kk}\big(\HH^0\big(X, \mathcal{O}_X/(I^n)^\sim\big)\big)$ becomes a polynomial  for $n\gg0$.
Also,  we can compute \autoref{mult_base_points} with the following equation
$$
e(\BB) = \sum_{\pp \in \BB} \left[\kk(\pp):\kk\right] \cdot e_{I_{\pp}}(R_{\pp}),
$$
where $e_{I_{\pp}}(R_{\pp})$ denotes the Hilbert-Samuel multiplicity of the local ring $R_{\pp}$ with respect to the $\pp R_{\pp}$-primary ideal $I_{\pp}$ (see \cite[Section 4.5]{BRUNS_HERZOG}).

The degree of the base locus $\BB=V(I)$ is defined in a similar way to its multiplicity \autoref{mult_base_points}.
When $\dim(\BB)=0, \ \deg(\BB)$ is given by the formula
\begin{align*}
\deg(\BB) := 
\dim_{\kk}\Big(\HH^0\Big(X, \mathcal{O}_X/I^\sim\Big)\Big)  &= \sum_{\pp \in \BB} \dim_{\kk}\Big({\big(\mathcal{O}_X/I^\sim\big)}_{\pp}\Big) \\
&=\sum_{\pp \in \BB} \left[\kk(\pp):\kk\right]\cdot \text{length}_{\mathcal{O}_{X,\pp}}\Big({\big(\mathcal{O}_X/I^\sim\big)}_{\pp}\Big) \\
&= \sum_{\pp \in \BB} \left[\kk(\pp):\kk\right]\cdot \text{length}_{R_{\pp}}\Big(R_{\pp}/I_{\pp}\Big).
\end{align*}

The theorem below provides a new algebraic proof of the degree formula for a multi-graded rational map with finitely many base points.

\begin{theorem}
	\label{thm_degree_formula_base_points}
	Let $\FF: X = X_1 \times X_2 \times \cdots \times X_m \dashrightarrow Y$ be a dominant rational  map. 
	If $\dim(Y)=\delta$ and $\dim(\BB)=0$, then 
	$$
	d_1^{\delta_1}\cdots d_m^{\delta_m}\deg_{\PP^N}(X) = \deg_{\PP^s}(Y)\deg(\FF) + e(\BB), 
	$$
	or equivalently
	$$
	d_1^{\delta_1}\cdots d_m^{\delta_m} \frac{\delta!}{\delta_1!\cdots \delta_m!} \deg_{\PP^{r_1}}(X_1)\cdots \deg_{\PP^{r_m}}(X_m) = \deg_{\PP^s}(Y)\deg(\FF) + e(\BB).
	$$
	\begin{proof}
						For $n\ge 1$ we have the exact sequence of sheaves
		$$
		0\rightarrow {(I^n)}^\sim(n\cdot\mathbf{d}) \rightarrow \mathcal{O}_X(n\cdot\mathbf{d}) \rightarrow \frac{\mathcal{O}_X}{{(I^n)}^\sim}(n\cdot\mathbf{d}) \rightarrow 0,
		$$
		that gives us the following equation relating Euler characteristics
		$$
		\chi\big(X,{(I^n)}^\sim(n\cdot\mathbf{d})\big) + \chi\Big(X, \frac{\mathcal{O}_X}{{(I^n)}^\sim}(n\cdot\mathbf{d})\Big) = \chi(X, \mathcal{O}_X(n\cdot\mathbf{d})).
		$$
		
		The term $\chi(X, \mathcal{O}_X(n\cdot\mathbf{d}))$ for sufficiently large $n$ becomes
		$$
		\chi(X, \mathcal{O}_X(n\cdot\mathbf{d})) = \dim_{\kk}\big( \HH^0(X, \mathcal{O}_X(n\cdot\mathbf{d}))\big) = \dim_{\kk}(R_{n\cdot\mathbf{d}}) = \frac{d_1^{\delta_1}\cdots d_m^{\delta_m}\deg_{\PP^N}(X)}{\delta!}n^\delta+\text{lower terms}				
		$$
		the Hilbert polynomial of the standard graded $\kk$-algebra $T (=R^{(\mathbf{d})})$ (recall that $\HH^i(X, \mathcal{O}_X(n\cdot\mathbf{d}))=0$ for $i\ge 1$ and $n\gg0$; see \cite[Theorem 1.6]{HYRY_MULTIGRAD}).
		
		Since $\dim(\BB)=0$, the summand $\chi\Big(X, \frac{\mathcal{O}_X}{{(I^n)}^\sim}(n\cdot\mathbf{d})\Big)$ for all $n\gg0$ is a polynomial 
		$$
		\chi\Big(X, \frac{\mathcal{O}_X}{{(I^n)}^\sim}(n\cdot\mathbf{d})\Big)=\dim_{\kk}\Big(\HH^0\Big(X, \frac{\mathcal{O}_X}{{(I^n)}^\sim}\Big)\Big) = \frac{e(\BB)}{\delta!}n^\delta + \text{lower terms}
		$$
		whose leading coefficient is equal to the multiplicity of the base points.
		
		We clearly have that $\FF$ is a generically finite map, so \autoref{prop:dimvanish} yields that for any $i\ge 1$  and $n\gg 0$, the expression
		$$
		\dim_{\kk}\big(\HH^i(X,{(I^n)}^\sim(n\cdot\mathbf{d}))\big)=\dim_{\kk}\big( {\left[\HH_{\nn}^{i+1}(\Rees(I))\right]}_{\mathbf{0},n} \big)
		$$
		becomes a polynomial of degree strictly less than $\delta$.		
		This implies that the leading coefficient of the polynomial determined by 
		$
		\chi\big(X,{(I^n)}^\sim(n\cdot\mathbf{d})\big)
		$
		coincides with the one of the polynomial determined by 
		$$
		\dim_{\kk}\big(\HH^0(X,{(I^n)}^\sim(n\cdot\mathbf{d}))\big).
		$$		
		Therefore, from \autoref{theorem_factorize_rat_map}$(iii)$, for $n\gg0$ the function  $\chi\big(X,{(I^n)}^\sim(n\cdot\mathbf{d})\big)$ is also a polynomial that has the form 
		$$
		\chi\big(X,{(I^n)}^\sim(n\cdot\mathbf{d})\big) = \frac{\deg_{\PP^s}(Y)\deg(\FF)}{\delta!}n^\delta+\text{lower terms}.
		$$
		Finally, comparing the leading coefficients of these polynomials, the equation  
		$$
		d_1^{\delta_1}\cdots d_m^{\delta_m}\deg_{\PP^N}(X) = \deg_{\PP^s}(Y)\deg(\FF) + e(\BB)
		$$
		follows.
		The other formula is equivalent from \autoref{lem_degree_Segre}.
	\end{proof}
\end{theorem}

\subsection{Degree and syzygies of the base ideal}
\label{subsection_Sym_algebra}

In this subsection, using the close link between the Rees and the symmetric algebras, we derive some consequences of \autoref{NEW_CRITERION} in terms of the symmetric algebra of the base ideal of a rational map.  
Under the assumption of having a zero dimensional base locus, we bound the multiplicity $e_{\delta+1}\left({\left[\HH_{\nn}^1(\Rees(I))\right]}_{\mathbf{0}}\right)$ of the Rees algebra with the corresponding multiplicity $e_{\delta+1}\left( {\left[\HH_{\nn}^1\big(\Sym(I)\big)\right]}_\mathbf{0} \right)$ of the symmetric algebra, and the later one is bounded by using the $\ZZC_\bullet$ approximation complex.

We keep here similar notations with respect to the previous ones, but we assume that the image $Y$ is the projective space $\PP^{\delta}$.
We take this assumption because in general the symmetric algebra $\Sym(I)$ is only a $\kk[\yy]$-module and not an $S$-module (\autoref{notations_section_2}).
To be precise, we restate the notations that we use in this subsection.

\begin{notation}
	\label{notations_section_3}
	Let  $\mathcal{F} : X = X_1 \times X_2 \times \cdots \times X_m \dashrightarrow \PP^\delta$ be a dominant rational map defined by $\delta+1$ multi-homogeneous forms $\mathbf{f}=\{f_0, f_1, \ldots, f_\delta\} \subset R$ of the same multi-degree $\mathbf{d}=(d_1,\ldots,d_m)$, where $\delta$ is the dimension of $X$.
	Let $I \subset R$ be the multi-homogeneous ideal generated by $f_0,f_1,\ldots,f_\delta$.
	Let $S$ be the homogeneous coordinate ring $S = \kk[y_0,y_1,\ldots,y_\delta]$ of $\PP^\delta$.
\end{notation}

\begin{remark}
	\label{mult_eq_rank}
		Given a finitely generated $S$-module $N$, from the associative formula for multiplicity \cite[Corollary 4.6.9]{BRUNS_HERZOG}, we get 
		$$
		e_{\delta+1}\left(N\right) = \rank\left(N\right).
		$$
\end{remark}

The Rees algebra $\Rees(I)$ has a natural structure of multi-graded $\AAA$-module by \autoref{presentation_Rees}.
Also, from the minimal graded presentation of $I$
$$
F_1 \xrightarrow{\varphi} F_0 \xrightarrow{\left(f_0,\ldots,f_\delta\right)} I \rightarrow 0,
$$
the symmetric algebra 
$$
\Sym(I) \cong \AAA / I_1 \big( \left(y_0,\cdots,y_\delta\right) \cdot \varphi \big)
$$
has a natural structure of multi-graded $\AAA$-module.
Therefore, we have a canonical exact sequence of multi-graded $\AAA$-modules relating both algebras 
\begin{equation}
		\label{canonical_map_Sym_Rees}
		0 \rightarrow \EEQ \rightarrow \Sym(I) \rightarrow \Rees(I) \rightarrow 0.		
\end{equation}

\begin{remark}
	\label{torsion_multi_sym_alg}
	From the pioneering work \cite{MICALI_REES} (see also \autoref{lem_torsion_multi_sym_alg}) we have that the torsion submodule $\EEQ$ is given by
	$$
	\EEQ = \HH_{I}^0\left(\Sym(I)\right).
	$$	
\end{remark}

The following result is likely part of the folklore, but we include a proof for the sake of completeness.
\begin{lemma}
	\label{vanish_local_cohom}
	Let $M$ be a multi-graded $R$-module (not necessarily finitely generated) and $Z\subset X$ be a closed subset of dimension zero.
	If $\left(\Supp_R(M) \cap X \right)\subset Z$, then we have $\HH_{\nn}^j(M)=0$ for any $j \ge 2$.
	\begin{proof}
			Let $i:Z \rightarrow X$ be the inclusion of the closed set $Z$,  $\MM$ the sheafification $\MM=\widetilde{M}(\mathbf{n})$ of $M$ twisted by $\mathbf{n} \in \ZZ^m$, and ${\MM\mid}_{Z}$ the restriction of $\MM$ to $Z$.
			Since the support of $\MM$  is contained in $Z$, then extending ${\MM\mid}_{Z}$ by zero gives the isomorphism $\MM \cong i_{*}\left({\MM\mid}_{Z}\right)$ (see \cite[Exercise II.1.19(c)]{HARTSHORNE}).
			Using \autoref{relation_local_sheaf_cohom_positive}, \cite[Lemma III.2.10]{HARTSHORNE} and the Grothendieck vanishing theorem \cite[Theorem III.2.7]{HARTSHORNE}, we get 
			$$
			{\left[\HH_{\nn}^{j}\big(M\big)\right]}_{\mathbf{n}} \cong \HH^{j-1}\left(X,\MM\right) \cong			\HH^{j-1}\left(X,i_{*}\left({\MM\mid}_{Z}\right)\right) = \HH^{j-1}\left(Z, {\MM\mid}_{Z}\right)=0
			$$
			for any $j\ge 2$ and $\mathbf{n} \in \ZZ^m$.
	\end{proof}
\end{lemma}

\begin{lemma}
	\label{lem_mult_H_m_1_Sym_and_Rees}
	The following statements hold:
	\begin{enumerate}[(i)]
		\item For each $i \ge 0,\ {\left[\HH_{\nn}^i(\Sym(I))\right]}_\mathbf{0}$ is a finitely generated graded $S$-module.
		\item If $\dim(\BB)=0$, then 
		$$
		\rank\Big( {\left[\HH_{\nn}^1(\Sym(I))\right]}_\mathbf{0} \Big) = \rank\Big( {\left[\HH_{\nn}^1(\Rees(I))\right]}_\mathbf{0} \Big) + \rank\Big( {\left[\HH_{\nn}^1\big(\HH_I^0(\Sym(I))\big)\right]}_\mathbf{0} \Big).
		$$
	\end{enumerate}
	\begin{proof}
			$(i)$ The proof of \autoref{prop_finite_dim_Hn_Rees}$(i)$ applies verbatim.
			
			$(ii)$ From \autoref{torsion_multi_sym_alg}, we can make the identification $\EEQ = \HH_I^0(\Sym(I))$ in the short exact sequence \autoref{canonical_map_Sym_Rees}.
			Hence, we can obtain the following long exact sequence in local cohomology
			$$
			\HH_{\nn}^0(\Rees(I)) \rightarrow \HH_{\nn}^1\big(\HH_I^0(\Sym(I))\big) \rightarrow \HH_{\nn}^1(\Sym(I)) \rightarrow  \HH_{\nn}^1(\Rees(I))
			\rightarrow
			\HH_{\nn}^2\big(\HH_I^0(\Sym(I))\big).
			$$
			We clearly have that $\HH_{\nn}^0(\Rees(I))=0$, and from \autoref{vanish_local_cohom} we get that $\HH_{\nn}^2\big(\HH_I^0(\Sym(I))\big)=0$.
			Therefore, the assertion follows.
	\end{proof}
\end{lemma}

In the rest of this subsection one of the main tools to be used will be the so-called approximation complexes.
These complexes were introduced in \cite{Simis_Vasc_Syz_Conormal_Mod}, and extensively developed in \cite{HSV_Approx_Complexes_I}, \cite{HSV_Approx_Complexes_II} and \cite{HSV_TRENTO_SCHOOL}.
In particular, we will consider the $\ZZC_\bullet$ complex in order to obtain an approximation of a resolution of $\Sym(I)$.

We fix some notations regarding the  approximation complexes, and for more details we refer the reader to \cite{HSV_TRENTO_SCHOOL}.
Let $K_\bullet=K(f_0,\ldots,f_\delta; R)$ be the graded Koszul complex  of $R$-modules
\begin{equation*}
	K_\bullet: 0 \rightarrow K_{\delta+1} \xrightarrow{d_{\delta+1}} K_\delta \xrightarrow{d_\delta}\ldots \xrightarrow{d_2} K_1 \xrightarrow{d_1} K_0 \xrightarrow{d_0} 0
\end{equation*}
associated to the sequence $\{f_0,\ldots,f_\delta\}$.
For each $i\ge0$, let $Z_i$ be the $i$-th Koszul cycle and $H_i$ be the $i$-th Koszul homology, that is
$Z_i = \Ker(d_i)$ and $H_i=\HH_i(K_\bullet)$.
Using the Koszul complex $K(y_0,\ldots,y_\delta; \AAA)$, one can construct the approximation complexes $\ZZC_\bullet$ and $\mathcal{M}_\bullet$ (see \cite[Section 4]{HSV_TRENTO_SCHOOL}). 
The $\ZZC_\bullet$ complex is given by  
\begin{equation*}
		\label{Z_complex}
		\ZZC_\bullet: 0 \rightarrow \ZZC_{\delta+1} \rightarrow \ZZC_{\delta} \rightarrow \cdots \rightarrow \ZZC_{1} \rightarrow \ZZC_0 \rightarrow 0,	
\end{equation*}
where $\ZZC_{i} = \big[Z_i \otimes_R \AAA\big](i\cdot\mathbf{d},-i)$ for all $1\le i \le \delta+1$. 
We have that $\HH_0(\ZZC_\bullet) \cong \Sym(I)$ and $\ZZC_{\delta+1}=0$.
Similarly, the $\MM_\bullet$ complex is given by 
\begin{equation*}
\label{M_complex}
\MM_\bullet: 0 \rightarrow \MM_{\delta+1} \rightarrow \MM_{\delta} \rightarrow \cdots \rightarrow \MM_{1} \rightarrow \MM_0 \rightarrow 0,	
\end{equation*}
where $\MM_{i} = \big[H_i \otimes_R \AAA\big](i\cdot\mathbf{d},-i)$ for all $1\le i \le \delta+1$. 

\medskip

The next theorem contains the main results of this subsection.

\begin{theorem}
	\label{degree_in_terms_of_Sym}
	Let $\FF: X = X_1 \times X_2 \times \cdots \times X_m \dashrightarrow \PP^{\delta}$ be a dominant rational map. 
	If $\dim(\BB)=0$,  then the following statements hold:
	\begin{enumerate}[(i)]
		\item $\deg(\FF) = \rank\Big( {\left[\HH_{\nn}^0\big(\HH_I^1(\Sym(I))\big)\right]}_\mathbf{0}  \Big) + 1$.
		\item We have  
		$$\deg(\FF) \le \rank\Big( {\left[\HH_{\nn}^1\big(\Sym(I)\big)\right]}_\mathbf{0}  \Big) + 1,$$
		with equality if $I$ is of linear type.
		\item In terms of the Koszul cycles $Z_i$, we get the following upper bound
		$$
		\deg(\FF) \le 1+ \sum_{i=0}^{\delta} \dim_{\kk}\Big({\left[\HH_{\nn}^{i+1}(Z_i)\right]}_{i\cdot \mathbf{d}} \Big).
		$$
	\end{enumerate} 
	\begin{proof}
		$(i)$ We will consider the double complex $F^{\bullet,\bullet}=C_{\nn}^{\bullet} \otimes_R C_I^\bullet \otimes_R \Sym(I)$, where $C_{\nn}^{\bullet}$ and $C_I^\bullet$ are the \v{C}ech complexes corresponding with $\nn$ and $I$, respectively.
		We have the spectral sequence 
		$$
		E_2^{p,q} = \HH_{\nn}^p\big( \HH_I^q(\Sym(I)) \big) \Longrightarrow \HH^{p+q}(\text{Tot}(F^{\bullet,\bullet})) \cong \HH_{\nn}^{p+q}(\Sym(I)).
		$$
		From \autoref{vanish_local_cohom} we obtain that
		$E_2^{p,q} = 0$ for $p \ge 2$.
		Therefore, the
		spectral sequence converges with  $E_2^{p,q}=E_{\infty}^{p,q}$.
		
		 The filtration of the term $\HH^{1}(\text{Tot}(F^{\bullet,\bullet})) \cong \HH_{\nn}^1(\Sym(I))$ yields the equality 
		$$
		\rank \Big( {\left[\HH_{\nn}^1(\Sym(I))\right]}_\mathbf{0} \Big) = \rank \Big( {\left[\HH_{\nn}^0\big( \HH_I^1(\Sym(I)) \big)\right]}_\mathbf{0} \Big) + \rank \Big( {\left[\HH_{\nn}^1\big( \HH_I^0(\Sym(I)) \big)\right]}_\mathbf{0} \Big),
		$$
		and assembling with \autoref{mult_eq_rank}, \autoref{NEW_CRITERION},  and \autoref{lem_mult_H_m_1_Sym_and_Rees}$(ii)$ we get 
		$$
		\deg(\FF) = \rank \Big( {\left[\HH_{\nn}^0\big( \HH_I^1(\Sym(I)) \big)\right]}_\mathbf{0} \Big) + 1.
		$$
		
		$(ii)$ It follows from \autoref{mult_eq_rank}, \autoref{NEW_CRITERION} and \autoref{lem_mult_H_m_1_Sym_and_Rees}$(ii)$.
		
		$(iii)$ 
		For any $i \ge 0$, we have that $I\cdot H_i=0$ and so the support of $H_i$ is contained in $\BB=V(I)$.
		Hence, for any $\pp \not\in \BB$ we have $(\MM_\bullet)_{\pp}=0$.
		Applying basic properties of approximation complexes (see e.g.~\cite[Corollary 4.6]{HSV_TRENTO_SCHOOL}), we can obtain that $\HH_i(\ZZC_{\bullet})_{\pp}=0$ for any $\pp \not\in \BB$ and $i \ge 1$.
		Therefore, from \autoref{vanish_local_cohom} we get that $\HH_{\nn}^j(\HH_i(\ZZC_{\bullet}))=0$ for any $j \ge 2$ and $i \ge 1$.
		  
		Let $\{\vartheta_1,\ldots,\vartheta_r \}$ be a set of generators of $\nn$ and  $G^{\bullet,\bullet}$ be the corresponding double complex
		 \begin{center}		
		 	\begin{tikzpicture}
		 	\matrix (m) [matrix of math nodes,row sep=1em,column sep=2em,minimum width=2em, text height=1.5ex, text depth=0.25ex]
		 	{
		 		0 & \ZZC_\delta \otimes_R C_{\nn}^{r} & \ZZC_{\delta-1} \otimes_R C_{\nn}^{r}  & \cdots & \ZZC_0 \otimes_R C_{\nn}^{r}  & 0\\
		 		& \vdots & \vdots & & \vdots & \\
		 		0 & \ZZC_\delta \otimes_R C_{\nn}^{1} & \ZZC_{\delta-1} \otimes_R C_{\nn}^{1}  & \cdots & \ZZC_0 \otimes_R C_{\nn}^{1}  & 0\\
		 		0 & \ZZC_\delta \otimes_R C_{\nn}^{0} & \ZZC_{\delta-1} \otimes_R C_{\nn}^{0}  & \cdots & \ZZC_0 \otimes_R C_{\nn}^{0}  & 0\\
		 	};
		 	\path[-stealth]
			(m-1-1) edge (m-1-2)
			(m-1-2) edge (m-1-3)
			(m-1-3) edge (m-1-4)
			(m-1-4) edge (m-1-5)			
			(m-1-5) edge (m-1-6)
			(m-3-1) edge (m-3-2)
			(m-3-2) edge (m-3-3)
			(m-3-3) edge (m-3-4)
			(m-3-4)	 edge (m-3-5)			
			(m-3-5) edge (m-3-6)
			(m-4-1) edge (m-4-2)
			(m-4-2) edge (m-4-3)
			(m-4-3) edge (m-4-4)
			(m-4-4)	 edge (m-4-5)			
			(m-4-5) edge (m-4-6)
			(m-2-2) edge (m-1-2)
			(m-2-3) edge (m-1-3)
			(m-2-5) edge (m-1-5)
			(m-3-2) edge (m-2-2)
			(m-3-3) edge (m-2-3)
			(m-3-5) edge (m-2-5)
			(m-4-2) edge (m-3-2)
			(m-4-3) edge (m-3-3)
			(m-4-5) edge (m-3-5)			
		 	;
		 	\end{tikzpicture}	
		 \end{center}
	 	  given by $\ZZC_\bullet \otimes_R C_{\nn}^\bullet$.
		 The double complex above is written in the second quadrant.
		 Then, the spectral sequence corresponding with the second filtration is given by
		 $$
		 {}^{\text{II}}E_2^{p,-q} = \begin{cases}
		 		\HH_{\nn}^p\left(\Sym(I)\right) \qquad \text{if } q = 0\\
				\HH_{\nn}^p\left(\HH_q(\ZZC_\bullet)\right) \qquad \text{if } p \le 1 \text{ and } q \ge 1\\
				0 \qquad \qquad \qquad \quad \text{otherwise.}
		 \end{cases}
		 $$
		 Thus, it converges with  $E_2^{p,-q}=E_{\infty}^{p,-q}$.
		 In particular, we have $\HH^1(\text{Tot}(G^{\bullet,\bullet}))\cong\HH_{\nn}^1(\Sym(I))$.
		 
		 On the other hand, by computing with the first filtration we get 
		 $$
		 {}^{\text{I}}E_1^{-p,q} = \HH_{\nn}^q(\ZZC_p).
		 $$ 
		 Therefore we obtain the following upper bound
		 $$
		 \rank\Big( {\left[\HH_{\nn}^1(\Sym(I))\right]}_\mathbf{0} \Big) \le \sum_{i=0}^{\delta} \rank\Big( {\left[\HH_{\nn}^{i+1}(\ZZC_{i})\right]}_\mathbf{0} \Big).
		 $$
		 For each $0\le i \le \delta$, since $\ZZC_{i}=\left[Z_i\otimes_R \AAA\right] (i\cdot \mathbf{d},-i)$ then we have that
		\begin{align*}
			\rank\Big({\left[\HH_{\nn}^{i+1}(\ZZC_{i})\right]}_\mathbf{0}\Big) = \rank\Big( {\left[\HH_{\nn}^{i+1}(Z_i)\right]}_{i\cdot \mathbf{d}} \otimes_{\kk} S(-i) \Big) = \dim_{\kk}\big({\left[\HH_{\nn}^{i+1}(Z_i)\right]}_{i\cdot \mathbf{d}}\big).
		\end{align*}
		Finally, the inequality follows from part $(ii)$.
	\end{proof}
\end{theorem}

\subsection{Rational maps defined over multi-projective spaces}\label{subsection_mult_proj_spaces}

Here we specialize further our approach to the case of a multi-graded dominant rational map from a multi-projective space to a projective space.
The main results of this subsection are given in \autoref{thm_bigrad_map_1_n} and \autoref{thm_bigrad_map_2_2}, where we provide effective criteria for the birationality of a bi-graded rational map of the form $\PP^1\times\PP^1 \dashrightarrow \PP^2$ with low bi-degree. 
We set the following notation.

\begin{notation}
	\label{notations_section_5}
	Let $m\ge 1$.
	For each $i = 1,\ldots,m$, let $X_i$ be the projective space $X_i=\PP^{r_i}$ and $A_i$ be its coordinate ring $A_i=\kk[\xx_i]=\kk[x_{i,0},x_{i,1},\ldots,x_{i,r_i}]$.	
	Let  $\mathcal{F} : X = X_1 \times X_2 \times \cdots \times X_m \dashrightarrow \PP^{\delta}$ be a dominant rational map defined by $\delta+1$ multi-homogeneous polynomials $\mathbf{f}=\{f_0, f_1, \ldots, f_\delta\} \subset R:=A_1\otimes_{\kk}A_2\otimes_{\kk}\cdots A_m$ of the same multi-degree $\mathbf{d}=(d_1,\ldots,d_m)$, where $\delta=r_1+r_2+\cdots+r_m$ is the dimension of $X$.
	Let $I \subset R$ be the multi-homogeneous ideal generated by $f_0,f_1,\ldots,f_\delta$.
	Let $S$ be the homogeneous coordinate ring $S = \kk[y_0,y_1,\ldots,y_\delta]$ of $\PP^\delta$.
	Let $\nn$ be the irrelevant multi-homogeneous ideal of $R$, which is given by 
	$
	\nn = \bigoplus_{j_1>0,\ldots,j_m>0} R_{j_1,\ldots,j_m}.
	$
\end{notation}

First we give a description of the local cohomology modules $\HH_{\nn}^j(R)$, with special attention to its multi-graded structure.
We provide a shorter proof than the one obtained in \cite[Section 6.1]{BOTBOL_IMPLICIT_SURFACE}.

Given any subset $\alpha$ of $\{1,\ldots,m\}$, then we define its weight by 
$
\lvert\lvert \alpha \rvert\rvert = \sum_{i \in \alpha} r_i.
$
For $i \in\{1,\ldots,m\},$ let $\mm_i$ be the maximal irrelevant ideal $\mm_i=(\xx_i)=(x_{i,0},x_{i,1},\ldots,x_{i,r_i}).$ We then have that 
$$
\HH_{\mm_i}^j(A_i) \cong \begin{cases}
\frac{1}{\xx_i}\kk[\xx_i^{-1}] \qquad\text{if } j=r_i+1\\
0 \qquad\qquad\quad\;\; \text{otherwise.}
\end{cases}
$$

\begin{proposition}
	\label{local_cohom_multi_grad}
	For any $j \ge 0$ we have that
	$$
	\HH_{\nn}^j(R) 
	\;\cong\; \bigoplus_{\substack{\alpha \subset \{1,\ldots,m\}\\
			\lvert\lvert\alpha \rvert\rvert + 1 = j}
	} 		
	\left( 
	\bigotimes_{i \in \alpha}
	\frac{1}{\xx_i}\kk[\xx_i^{-1}]
	\right) \otimes_{\kk} 		
	\left( 
	\bigotimes_{i \not\in \alpha}
	A_i
	\right).
	$$		
	\begin{proof}
		First we check that $\HH_{\nn}^0(R)=\HH_{\nn}^1(R)=0$.
		It is clear that $\HH_{\nn}^0(R)=0$, and using \autoref{relation_local_sheaf_cohom_zero} we get the exact sequence 
		$$
		0  \rightarrow R \rightarrow \bigoplus_{\mathbf{n} \in \ZZ^m} \HH^0(X, \OO_X(\mathbf{n})) \rightarrow \HH_{\nn}^1(R) \rightarrow 0.
		$$
		From the K\"unneth formula (see \cite[\href{https://stacks.math.columbia.edu/tag/0BEC}{Tag 0BEC}]{stacks-project} for a detailed proof) and \cite[Proposition II.5.13]{HARTSHORNE} we obtain 
		$$
		\HH^0(X, \OO_X(\mathbf{n})) \cong \HH^0\left(X_1,\OO_{X_1}(n_1)\right) \otimes_{\kk}\cdots \otimes_{\kk} \HH^0\left(X_m,\OO_{X_m}(n_m)\right) \cong {\left[A_1\right]}_{n_1} \otimes_{\kk}\cdots\otimes_{\kk}{\left[A_m\right]}_{n_m} \cong R_{\mathbf{n}},
		$$
		so we conclude that $\HH_{\nn}^1(R)=0$.
		
		Let $j\ge 2$. Then, the K\"unneth formula and \autoref{relation_local_sheaf_cohom_positive} yield the following isomorphisms 
		\begin{align*}
		\HH_{\nn}^{j}(R) &\cong \bigoplus_{\mathbf{n} \in \ZZ^m}\HH^{j-1}(X, \OO_X(\mathbf{n}))\\
		&\cong
		\bigoplus_{\mathbf{n} \in \ZZ^m}\left(
		\bigoplus_{j_1+\cdots+j_m=j-1}
		\HH^{j_1}\left(X_1,\OO_{X_1}(n_1)\right) \otimes_{\kk}\cdots \otimes_{\kk} \HH^{j_m}\left(X_m,\OO_{X_m}(n_m)\right)
		\right)
		.
		\end{align*}
		For each $i=1,\ldots,m$ we have that 
		$$
		\bigoplus_{n\in \ZZ}\HH^{j_i}\left(X_i,\OO_{X_i}(n)\right) \cong \begin{cases}
		A_i \qquad\qquad\quad\; \text{ if } j_i=0\\
		\HH_{\mm_i}^{r_i+1}(A_i) \qquad \text{ if } j_i=r_i\\
		0 \qquad\qquad\quad\;\;\; \text{ otherwise.}
		\end{cases}
		$$
		Therefore, we get the formula
		$$
		\HH_{\nn}^{j}(R) \cong \bigoplus_{\substack{j_1+\cdots+j_m=j-1\\j_i = 0 \text{ or } j_i=r_i}}
		\left( 
		\bigotimes_{\{i\mid j_i=r_i\}}
		\HH_{\mm_i}^{r_i+1}(A_i)
		\right) \otimes_{\kk} 		
		\left( 
		\bigotimes_{\{i\mid j_i=0\}}
		A_i
		\right),
		$$
		which is equivalent to the statement of the proposition.
	\end{proof}
\end{proposition}

Now we give a different proof of \autoref{theorem_factorize_rat_map}$(iv)$; in this case we recover the equivalence between the birationality of $\FF$ and the equality $\SSS=S$.
The following result is a  generalization of
\cite[Proposition 1.2]{PAN_RUSSO}.

\begin{proposition}
	\label{birationality_implies_sat}
	Let  $\FF: \PP^{r_1} \times \PP^{r_2} \times \cdots \times \PP^{r_m}\dashrightarrow \PP^\delta$ be a dominant rational map with $r_1+r_2+\ldots+r_m=\delta$.
	Then, the map $\FF$ is birational if and only if for all $n\ge 1$ we have
	$$
	{\left[I^n\right]}_{n\cdot\dd}={\left[{(I^n)}^{\sat}\right]}_{n\cdot\dd}.
	$$		  
	\begin{proof}
		From \autoref{theorem_factorize_rat_map}$(iii)$, the equality above implies the birationality of $\FF$. 
		
		For the other implication, let us assume that $\FF$ is birational.
		Since $\FF$ is dominant, then $S=\kk[y_0,\ldots,y_\delta]$ is isomorphic to the coordinate ring $S \cong \kk[I_{\dd}]=\kk[f_0,\ldots,f_\delta]$ of the image.
		Let $T$ be the multi-Veronese subring $T=\kk[R_{\dd}]$, then after regrading we have a canonical inclusion $S \cong \kk[I_{\dd}] \subset T$ of standard graded $\kk$-algebras. 
		From \autoref{lem_deg_map_field_deg} and the assumption of birationality we get 
		$$
		\left[T:S\right]= \deg(\FF)=1.
		$$
		So we have $\Quot(S)=\Quot(T)$ and the following canonical inclusions
		$$
		S \subset T \subset \Quot(T) = \Quot(S).
		$$
		
		Let $n\ge 1$.
		It is enough to prove that  for any $w \in {\left[{(I^n)}^{\sat}\right]}_{n\cdot\dd} \subset T_n$, we have that $w$ is integral over $S$.
		Indeed, since $S$ is integrally closed, it will imply that $w \in S_{n} \cong {\left[I^n\right]}_{n\cdot\dd}$.
		
		Let $w \in {\left[{(I^n)}^{\sat}\right]}_{n\cdot\dd}$.
		We shall prove the equivalent condition that $S[w]$ is a finitely generated $S$-module (see e.g.~\cite[Proposition 5.1]{ATIYAH_MACDONALD}).				
		From the condition $w \in {\left[{(I^n)}^{\sat}\right]}_{n\cdot\dd}$, we can choose some $r > 0$ such that 
		$$
		T_{rn}\cdot w = R_{rn\cdot\dd}\cdot w \subset {\left[I^n\right]}_{(r+1)n\cdot\dd}.
		$$
		We claim that for any $q \ge r+1$ we have $w^q \in S \cdot R_{rn\cdot\dd}$.
		If we prove this claim, then it will follow that $S[w]$ is a finitely generated $S$-module.
		
		Let $\{F_1,\ldots,F_c\}$ be a minimal generating set of the ideal $I^n$.
		For $q=r+1$, since $w^r \in R_{rn\cdot\dd}$ we can write 
		$$
		w^{r+1} =w^{r}w=h_1F_1 + h_2F_2+\cdots+h_cF_c,
		$$
		where $\deg(h_i)=rn\cdot \dd$ for each $i=1,\ldots,c$.
		For $q=r+2$, since $h_i\in R_{rn\cdot\dd}$ we get
		$$
		w^{r+2}=ww^{r+1}=\sum_{i=1}^{c} \left(h_iw\right)F_i = \sum_{i=1}^c\left(\sum_{j=1}^ch_{ij}F_j\right)F_i=\sum_{i,j}h_{ij}F_iF_j,
		$$
		where  each $h_{ij}$ has degree $\deg(h_{ij})=rn\cdot\dd$.
		Following this inductive process, we have that for each $q \ge r+1$ we can write
		$$
		w^{q} = \sum_{\substack{\beta}} h_{\beta}{\mathbf{F}}^{\beta},
		$$
		where $\deg(h_\beta)=rn\cdot\dd$ for each multi-index $\beta$.
		This gives us the claim that $w^q \in S \cdot R_{rn\cdot \dd}$ for each $q \ge r+1$.
	\end{proof}
\end{proposition}

From \autoref{birationality_implies_sat}
we deduce that for single-graded birational maps with non saturated base ideal,  the module $I^{\sat}/I$ is generated by elements of degree $\ge d+1$.

\begin{corollary}
	Let $\FF:\PP^r \dashrightarrow \PP^r$ be a birational map whose base ideal $I=(f_0,\cdots,f_r)$ is given by  $r+1$ relatively prime polynomials of the same degree $d$.
	Then, we have that 
	$$
	{\left[I^{\sat}/I\right]}_{\le d}=0.
	$$
	\begin{proof}
		From \autoref{birationality_implies_sat} we already have ${\left[I^{\sat}\right]}_d=I_d$.
		If we assume that there exists $0\neq h \in {\left[I^{\sat}\right]}_{d-1}$, then we get the contradiction $I_d=(x_0h,x_1h,\cdots,x_rh)$.
		Therefore, we have ${\left[I^{\sat}/I\right]}_{\le d}=0$.
	\end{proof}
\end{corollary}

For multi-graded birational maps the previous condition must not be necessarily satisfied. 

\begin{example}
	Let $\FF:\PP^1 \times \PP^1 \dashrightarrow \PP^2$ be the birational map given by 
	$$
	(x_{1,0}:x_{1,1}) \times (x_{2,0}:x_{2,1}) \mapsto (x_{1,0}x_{2,0}:x_{1,1}x_{2,0}:x_{1,1}x_{2,1}).
	$$
	Here, the base ideal $I=(x_{1,0}x_{2,0},\,x_{1,1}x_{2,0},\,x_{1,1}x_{2,1})$ is generated by forms of bi-degree $(1,1)$ and $\nn=(x_{1,0},x_{1,1}) \cap (x_{2,0},x_{2,1})$.
	The map $\FF$ is birational, but we have that $I^{\sat}=(I:\nn^{\infty})=(x_{1,1},x_{2,0})$ and so
	$$
	{\left[I^{\sat}/I\right]}_{(1,0)} \neq 0 \qquad \text{and} \qquad {\left[I^{\sat}/I\right]}_{(0,1)} \neq 0.
	$$
\end{example}

From now on, we focus on a dominant rational map of the form 
$
\FF:\PP^1\times\PP^1 \dashrightarrow \PP^2.
$
We shall adapt our previous results to this case and obtain a general upper bound for the degree of $\FF$.
More interestingly, we give a criterion for birationality when the bi-degrees of the $f_i$'s are of the form $\dd=(d_1,d_2)$ and $d_1=1$.
This result extends the work of \cite{EFFECTIVE_BIGRAD}, where a criterion was given for the bi-degrees $(1,1)$ and $(1,2)$.
Also, in the case $\dd=(d_1,d_2)=(2,2)$ we provide a general characterization for the birationality of $\FF$ (see \cite[Theorem 16]{EFFECTIVE_BIGRAD} for a more specific result).

\begin{proposition}
	\label{upper_bound_degree_P1_P1}
	Let $\FF:\PP^1\times\PP^1 \dashrightarrow \PP^2$ be a dominant rational map such that $\dim(\BB)=0$.
	Then, we have the inequality
	$$
	\deg(\FF) \le 1 + (d_1-1)(d_2-1) + \dim_{\kk}\Big({\left[I^{\sat}/I\right]}_{\dd}\Big).
	$$
	\begin{proof}
		From \autoref{degree_in_terms_of_Sym}$(iii)$ we have the inequality
		$$
		\deg(\FF) \le 1+ \dim_{\kk}\Big({\left[\HH_{\nn}^3(Z_2)\right]}_{2\cdot\dd}\Big) + 	\dim_{\kk}\Big({\left[\HH_{\nn}^2(Z_1)\right]}_{\dd}\Big) + \dim_{\kk}\Big({\left[\HH_{\nn}^1(Z_0)\right]}_{\dd}\Big).
		$$
		By \autoref{local_cohom_multi_grad} and the fact that $Z_0\cong R$ and $Z_2\cong R(-3\cdot \dd)$,  we obtain the isomorphisms $\HH_{\nn}^1(Z_0)=0$ and 
		$$
		\HH_{\nn}^3(Z_2)\cong \HH_{\nn}^3(R)(-3\cdot\dd)\cong \left(\frac{1}{\xx_1}\kk[\xx_1^{-1}]\right)(-3d_1) \otimes_{\kk} \left(\frac{1}{\xx_2}\kk[\xx_2^{-1}]\right)(-3d_2).
		$$
		Thus, we get that
		$$
		\dim_{\kk}\Big({\left[\HH_{\nn}^3(Z_2)\right]}_{2\cdot\dd}\Big)=\dim_{\kk}\left({\left[\frac{1}{\xx_1}\kk[\xx_1^{-1}]\right]}_{-d_1} \otimes_{\kk} {\left[\frac{1}{\xx_2}\kk[\xx_2^{-1}]\right]}_{-d_2}\right)=(d_1-1)(d_2-1).
		$$ 
		The exact sequences
		\begin{align*}
		0 \rightarrow Z_1 \rightarrow R^{3}&(-d_1,-d_2) \rightarrow I \rightarrow 0\\
		0 \rightarrow I \rightarrow &R \rightarrow R/I \rightarrow 0
		\end{align*} 
		and \autoref{local_cohom_multi_grad} yield the isomorphisms 
		$$
		\dim_{\kk}\big({\left[\HH_{\nn}^{2}(Z_1)\right]}_{\dd}\big) = \dim_{\kk}\big({\left[\HH_{\nn}^1(I)\right]}_{\dd}\big)=\dim_{\kk}\big({\left[\HH_{\nn}^0(R/I)\right]}_{\dd}\big)=\dim_{\kk}\Big({\left[I^{\text{sat}}/I\right]}_{\dd}\Big).
		$$
		Therefore, by combining these computations, we get the claimed  upper bound.
	\end{proof}
\end{proposition}

\begin{theorem}
	\label{thm_bigrad_map_1_n}
	Let $\FF:\PP^1\times\PP^1 \dashrightarrow \PP^2$ be a dominant rational map such that $\dim(\BB)=0$ and $\dd=(1,d_2)$.
	Then, $\FF$ is birational if and only if $I_{\dd}={\left[I^{\sat}\right]}_{\dd}$.
	\begin{proof}
		We get one implication from \autoref{birationality_implies_sat} and the other by specializing the data in the inequality of \autoref{upper_bound_degree_P1_P1}.
	\end{proof}
\end{theorem}

To illustrate this theorem, let $\FF$ be as above and assume moreover that there exists a nonzero syzygy of $I$ of bi-degree $(0,1)$. As in \cite[Remark 10]{EFFECTIVE_BIGRAD}, we deduce that $x_{2,0}(\sum_{i=0}^2\alpha_if_i)-x_{2,1}(\sum_{i=0}^2\beta_if_i)=0$
for some $\alpha_i$'s and $\beta_i$'s in $\kk$ and hence we deduce that there exist three polynomials $p,q,r$ of bi-degree $(1,d_2-1)$ such that $I=(x_{2,0}p,\, x_{2,1}p,\,x_{2,0}q+x_{2,1}r)$. Therefore, the ideal $I$ admits a Hilbert-Burch presentation of the form 
$$ F_\bullet: \quad0 \rightarrow R(-1,-1-d_2)\oplus R(-2,-2d_2+1) \rightarrow R(-1,-d_2)^3 \rightarrow R.$$
Studying the two spectral sequences coming from the double complex $F_\bullet \otimes_{R} C_{\nn}^\bullet$, together with  \autoref{local_cohom_multi_grad}, it is then easy to see that 
${\left[I^{\sat}/I\right]}_{\dd} \cong {\left[\HH^0_{\nn}(R/I)\right]}_{\dd}
=0$. 
Thus, \autoref{thm_bigrad_map_1_n} implies that $\FF$ is birational, a fact that can be deduced more directly and that is the main ingredient to ensure birationality in  \cite{SEDERBERG20161}. But \autoref{thm_bigrad_map_1_n} provides actually a finer result. 
Indeed, suppose that the ideal $I$ admits the following more general Hilbert-Burch presentation
 $$ 0 \rightarrow R(-1,-\mu-d_2)\oplus R(-2,-2d_2+\mu) \rightarrow R(-1,-d_2)^3 \rightarrow R$$
where $\mu$ is a positive integer. Then, a similar computation shows that 
${\left[\HH^0_{\nn}(R/I)\right]}_{\dd}\cong
 {\left[\HH^2_{\nn}(R)\right]}_{(0,-\mu)}$ and from here we deduce that $\FF$ cannot be a birational map if $\mu>1$.

\begin{lemma}
	\label{lem_deg_six}
	Let $\FF:\PP^1\times\PP^1 \dashrightarrow \PP^2$ be a dominant rational map such that $\dim(\BB)=0$ and $\dd=(2,2)$.
	Then, 
	$I_{\dd}={\left[I^{\sat}\right]}_{\dd}$ if and only if $\deg(\BB)=6$.
	\begin{proof}
		From \autoref{relation_local_sheaf_cohom_zero} we have the short exact sequence 
		$$
		0 \rightarrow {\left[\HH_{\nn}^0(R/I)\right]}_{\dd} \rightarrow {\left[R/I\right]}_{\dd} \rightarrow \HH^0\left(X, \left(\OO_X/I^\sim\right) (\dd)\right) \rightarrow {\left[\HH_{\nn}^1(R/I)\right]}_{\dd} \rightarrow 0.
		$$
		Using \cite[Lemma 5]{EFFECTIVE_BIGRAD} we deduce that ${\left[\HH_{\nn}^1(R/I)\right]}_{\dd}=0$.
		Therefore, we obtain
		\begin{align*}
		\deg(\BB) &= \dim_{\kk}\left(\HH^0\left(X, \left(\OO_X/I^\sim\right) (\dd)\right)\right)\\
		&=\dim_{\kk}\left({\left[R/I\right]}_{\dd}\right)- \dim_{\kk}\left({\left[\HH_{\nn}^0(R/I)\right]}_{\dd}\right) = 6 - \dim_{\kk}\left({\left[\HH_{\nn}^0(R/I)\right]}_{\dd}\right) 		
		\end{align*}
		from the exact sequence above,
		and so the claimed result follows.
	\end{proof}
\end{lemma}

\begin{theorem}
	\label{thm_bigrad_map_2_2}
	Let $\FF:\PP^1\times\PP^1 \dashrightarrow \PP^2$ be a dominant rational map.
	Suppose that $\dim(\BB)=0$ and $\dd=(2,2)$.
	Then, $\FF$ is birational if and only if the following conditions are satisfied:
	\begin{enumerate}[(i)]
		\item $I_{\dd}={\left[I^{\sat}\right]}_{\dd}$.
		\item $I$ is not locally a complete intersection at its minimal primes.
	\end{enumerate}
	\begin{proof}
		The degree formula of \autoref{thm_degree_formula_base_points} applied in our setting gives
		$$
		\deg(\FF) = 8 - e(\BB).
		$$
		Hence, we deduce that $e(\BB)\le 7$ and that $\FF$ is birational if and only if $e(\BB)=7$.
		We know that 
		$
		\deg(\BB) \le e(\BB),
		$		
		and that $\deg(\BB) = e(\BB)$ if and only if $I$ is locally a complete intersection at its minimal primes.		
		Moreover, we have already seen that the  condition $I_{\dd}={\left[I^{\sat}\right]}_{\dd}$ is necessary for the birationality of $\FF$ (\autoref{birationality_implies_sat}) and that it is equivalent to $\deg(\BB)=6$ (\autoref{lem_deg_six}).
		Therefore, assuming $I_{\dd}={\left[I^{\sat}\right]}_{\dd}$, we have that $I$ is not locally a complete intersection at its minimal primes
		if and only if
		$$
		6=\deg(\BB) < e(\BB)=7,
		$$ 
		and the later one is equivalent to the birationality of $\FF$.
	\end{proof}
\end{theorem}

\subsection{An explicit upper bound for the degree of a rational map defined over a projective space}
\label{subsection_single_grad_rat_maps}

In this subsection we consider the more specific case of single-graded  dominant rational maps. The main result here is  \autoref{graded_case_degree_in_terms_of_Sym} where the upper bound for the degree of a rational map given in \autoref{degree_in_terms_of_Sym}$(iii)$, is expressed solely in terms of the Hilbert functions of $R/I$ and $I^{\sat}/I$, instead of some local cohomology modules of Koszul cycles. We also show that this bound is sharp in some cases. We set the following notation.

\begin{notation}
	\label{nota_single_graded}
	Let $R$ be the standard graded polynomial ring $R=\kk[x_0,x_1,\ldots,x_r]$, and $\mm$ be the maximal irrelevant ideal $\mm=(x_0,\ldots,x_r)$.
	Let  $\mathcal{F} : \PP^r \dashrightarrow  \PP^r$ be a dominant rational map defined by $r+1$ homogeneous polynomials $\mathbf{f}=\{f_0, f_1, \ldots, f_r\} \subset R$ of the same degree $d$.
	Let $I \subset R$ be the homogeneous ideal generated by $f_0,f_1,\ldots,f_r$.
	Let $S$ be the standard graded polynomial ring $S = \kk[y_0,y_1,\ldots,y_r]$.
	Let $\AAA$ be the bigraded polynomial ring $\AAA = R \otimes_{\kk} S$, where $\bideg(x_i)=(1,0)$ and $\bideg(y_j)=(0,1)$. For any graded $R$-module $M$, we set $M^\vee={}^*\Hom_{\kk}(M,\kk)$ to be the graded Matlis dual of $M$ (see e.g.~\cite[Section 3.6]{BRUNS_HERZOG}).
\end{notation}

 The following lemma is equivalent to \cite[Lemma 1]{LAURENT_CHARDIN_IMPLICIT} in our setting; we include a proof for the sake of completeness and the convenience of the reader. 

\begin{lemma}
	\label{lem_local_cohom_Kosz_cycles}
	Let $Z_i$ and $H_i$ be the cycles and homology modules of the Koszul complex $K(\mathbf{f}; R)$, respectively.
	Assume that $\dim(R/I) \le 1$ and let $\xi = (r+1)(d-1)$.
	Then, 
	\begin{enumerate}[(i)]
		\item $Z_{r+1}=0$, $Z_{r}\cong R\left(-(r+1)d\right)$, $Z_0 = R$, $H_i=0$ for $i>1$,
		$H_1=0$ if and only if $\dim(R/I)=0$.
		If $\dim(R/I)=1$, then $H_1\cong \omega_{R/I}(-\xi)$.
		\item If $r \ge 2$ and $1 \le p < r$, then
		$$
		\HH_{\mm}^q(Z_p) \cong \begin{cases}
		\HH_{\mm}^{q-2}(R/I) 				 \qquad \;\;\text{if } p  = 1 \text{ and } q \le r \\
		H_{q-p-1}^\vee(-\xi) \qquad \text{if } 2 \le p < r \text{ and } q \le r\\
		Z_{r-p}^\vee(-\xi) \qquad \;\;\;\;\text{if } q = r+1.\\				
		\end{cases}
		$$
	\end{enumerate}
	\begin{proof}
		$(i)$ This part follows from well-known properties of the Koszul complex (see e.g.~\cite[Section 1.6]{BRUNS_HERZOG}).
		
		$(ii)$ We only need to compute the local cohomology modules of $Z_p$ for $1 \le p < r$.
		
		Let $2 \le \ell < r$. We denote by $K_{\bullet}^{>\ell}$ the truncated Koszul complex
		$$
		K_{\bullet}^{>\ell}: \quad 0 \rightarrow K_{r+1} \rightarrow K_r \rightarrow \cdots \rightarrow  K_{\ell+1} \rightarrow Z_{\ell} \rightarrow 0,
		$$
		which is exact from the condition $\dim(R/I) \le 1$.
		Let $F^{\bullet,\bullet}$ be the double complex given by $F^{\bullet,\bullet} = K_{\bullet}^{>\ell} \otimes_{R} C_{\mm}^{\bullet}$.
		The exactness of $				K_{\bullet}^{>\ell}$ implies that $\HH^{\bullet}\left(\text{Tot}(F^{\bullet,\bullet})\right)=0$.
		Hence computing with the first filtration we get the spectral sequence ${}^{\text{I}}E_{1}^{-p,q}=\HH_{\mm}^q(K_p^{>\ell}) \Rightarrow 0$, which at the first page is given by
		\begin{center}		
			\begin{tikzpicture}
			\matrix (m) [matrix of math nodes,row sep=1em,column sep=2em,minimum width=2em, text height=1.5ex, text depth=0.25ex]
			{
				\HH_{\mm}^{r+1}(K_{r+1}) &  \HH_{\mm}^{r+1}(K_{r}) & \cdots & \HH_{\mm}^{r+1}(K_{\ell+1}) & \HH_{\mm}^{r+1}(Z_\ell)\\
				0 &  0 & \cdots & 0  & \HH_{\mm}^{r}(Z_\ell)\\
				\vdots & \vdots &  & \vdots & \vdots & \\
				0 & 0 & \cdots  & 0  & \HH_{\mm}^0(Z_\ell).\\
			};
			\path[-stealth]
			(m-1-1) edge (m-1-2)
			(m-1-2) edge (m-1-3)
			(m-1-3) edge (m-1-4)
			(m-1-4) edge (m-1-5)
			;
			\end{tikzpicture}	
		\end{center}
		From the graded local duality theorem (see e.g.~\cite[Theorem 3.6.19]{BRUNS_HERZOG}) and the self-duality  of the Koszul complex, we have the following isomorphisms of complexes
		$$
		\HH_{\mm}^{r+1}(K_{\bullet}) \cong {\big( \Hom_R\left(K_{\bullet},R(-r-1)\right) \big)}^{\vee} \cong {\big(K_\bullet[r+1]\left((r+1)d-r-1\right)\big)}^\vee  \cong {\big(K_{\bullet}[r+1]\big)}^\vee(-\xi),			 $$
		where $[r+1]$ denotes homological shift degree.
		So the top row of the diagram above is given by the complex
		$$
		K_0^\vee(-\xi) \rightarrow K_1^\vee(-\xi) \rightarrow \cdots \rightarrow 	K_{r+1-(\ell+1)}^\vee(-\xi) \rightarrow \HH_{\mm}^{r+1}(Z_\ell).
		$$	
		For each $q \le r,$ 
		when we compute cohomology in the page $r+3-q,$
		we get the exact sequence 
		$$
		0 \rightarrow {}^{\text{I}}E_{r+3-q}^{-(\ell+r+2-q),r+1} \rightarrow H_{q-\ell-1}^\vee(-\xi) \rightarrow \HH_{\mm}^q(Z_\ell) \rightarrow {}^{\text{I}}E_{r+3-q}^{-\ell,q} \rightarrow 0.
		$$
		Since ${}^{\text{I}}E_{r+3-q}^{-(\ell+r+2-q),r+1}={}^{\text{I}}E_{\infty}^{-(\ell+r+2-q),r+1}=0$ and ${}^{\text{I}}E_{r+3-q}^{-\ell,q}={}^{\text{I}}E_{\infty}^{-\ell,q}=0$, then we get the isomorphism $\HH_{\mm}^q(Z_\ell) \cong H_{q-\ell-1}^\vee(-\xi)$ when $q \le r$.
		
		In the case of $q=r+1$, we have the exact sequence 
		$$
		K_{r+1-(\ell+2)}^\vee(-\xi)  \rightarrow 	K_{r+1-(\ell+1)}^\vee(-\xi) \rightarrow \HH_{\mm}^{r+1}(Z_\ell) \rightarrow 0,
		$$
		that induces the isomorphism $\HH_{\mm}^{r+1}(Z_{\ell})\cong Z_{r-\ell}^\vee(-\xi)$.
		
		When $\ell=1$, we consider the truncated Koszul complex 
		$$
		K_{\bullet}^{>1}: \quad 0 \rightarrow K_{r+1} \rightarrow K_r \rightarrow \cdots \rightarrow  K_{2} \rightarrow Z_{1} \rightarrow 0,
		$$
		that is not exact only at the module $Z_1$.
		The double complex $G^{\bullet,\bullet}=K_{\bullet}^{>1} \otimes_{R} C_{\mm}^{\bullet}$ now yields the spectral sequence 
		$$
		{}^{\text{I}}E_{1}^{-p,q}=\HH_{\mm}^q(K_p^{>\ell})  \Longrightarrow \HH^{-p+q}\left(G^{\bullet,\bullet}\right) 
		= \begin{cases}
		\HH_{\mm}^1(H_1) \quad \text{ if} -p+q=0\\
		0 			\qquad\qquad\; \text{otherwise.}
		\end{cases}
		$$ 
		Thus again we  have the exact sequence
		$$
		K_{r-2}^\vee(-\xi)  \rightarrow 	K_{r-1}^\vee(-\xi) \rightarrow \HH_{\mm}^{r+1}(Z_1) \rightarrow 0,
		$$
		and this gives us the isomorphism $\HH_{\mm}^{r+1}(Z_1)\cong Z_{r-1}^\vee(-\xi)$.
		
		Finally, using the following two short exact sequences
		\begin{align*}
		0 \rightarrow Z_1 \rightarrow K_1 \rightarrow I \rightarrow 0\\
		0 \rightarrow I \rightarrow R \rightarrow R/I \rightarrow 0
		\end{align*}
		we can obtain the isomorphisms
		$
		\HH_{\mm}^{q}(Z_1) \cong \HH_{\mm}^{q-1}(I) \cong \HH_{\mm}^{q-2}(R/I)
		$ for $q \le r$.
	\end{proof}	
\end{lemma}

Since the linear type condition has almost no geometrical meaning, we briefly restate the equality of \autoref{degree_in_terms_of_Sym}$(ii)$ in the locally complete intersection case.

\begin{lemma}
	Let $\FF: \PP^r \dashrightarrow \PP^r$ be a dominant rational map with a dimension $1$ base ideal $I$. 
	If $I$ is locally a complete intersection at its minimal primes then 
	$$
	\deg(\FF) = \rank\Big( {\left[\HH_{\mm}^1\big(\Sym(I)\big)\right]}_0  \Big) + 1.
	$$
	\begin{proof}
		From either \cite[Section 5]{HSV_TRENTO_SCHOOL} or \cite[Proposition 3.7]{Simis_Vasc_Syz_Conormal_Mod} we get that $I$ is of linear type. 
		Thus, the assertion follows from \autoref{degree_in_terms_of_Sym}$(ii)$.
	\end{proof}
\end{lemma}

The next theorem translates \autoref{degree_in_terms_of_Sym}$(iii)$ in terms of the Hilbert functions of $R/I$ and $I^{\sat}/I$.

\begin{theorem}
	\label{graded_case_degree_in_terms_of_Sym}
	Let $\FF:\PP^r \dashrightarrow \PP^r$ be a dominant rational map with base ideal $I$. 
	If $\dim(R/I) \le 1$, then we have the following upper bound
	$$
	\deg(\FF) \le 1+ \binom{d-1}{r}+ \dim_{\kk}\left({\left[I^{\sat}/I\right]}_d\right) + \sum_{i=2}^{r-1} \dim_{\kk}\left({\left[R/I\right]}_{(r+1-i)d-r-1}\right).
	$$
	\begin{proof}
		Since $Z_0=R$ and $Z_r\cong R(-(r+1)d)$, we have $\HH_{\mm}^1(Z_0)=0$ and 
		$$
		\dim_{\kk}\left({\left[\HH_{\mm}^{r+1}(Z_r)\right]}_{rd}\right) = \dim_{\kk}\left({\left[\frac{1}{\xx}\kk[\xx^{-1}]\right]}_{-d}\right) = \binom{(d-r-1)+r}{r} = \binom{d-1}{r}.
		$$
		From \autoref{lem_local_cohom_Kosz_cycles} we obtain that 
		$$
		\dim_{\kk}\left({\left[\HH_{\mm}^{2}(Z_1)\right]}_d\right) = \dim_{\kk}\left({\left[\HH_{\mm}^0\left(
			R/I\right)\right]}_d\right) = \dim_{\kk}\left({\left[I^{\sat}/I\right]}_d\right)
		$$
		and 
		$$
		\dim_{\kk}\left({\left[\HH_{\mm}^{i+1}(Z_i)\right]}_{id}\right) = \dim_{\kk}\left({\left[H_0^\vee\right]}_{(i-r-1)d+r+1}\right) = \dim_{\kk}\left({\left[R/I\right]}_{(r+1-i)d-r-1}\right)	
		$$
		for any $2\le i \le r-1$.	
		Finally, by substituting these computations in \autoref{degree_in_terms_of_Sym}$(iii)$, we obtain the required upper bound.
	\end{proof}
\end{theorem}

To end this subsection, we show that the above upper bound becomes sharp for dominant plane rational maps when the base ideal is of linear type and is defined by polynomials degree $d\le 3$.

\begin{proposition}
	\label{P^2_case}
	Let $\FF:\PP^2 \dashrightarrow \PP^2$ be a dominant rational map with a dimension $1$ base ideal $I$.
	Then, the following statements hold:
	\begin{enumerate}[(i)]
		\item $\deg(\FF) \le \frac{(d-1)(d-2)}{2} + \dim_{\kk}\Big( {\left[I^{\text{sat}}/I\right]}_d \Big) + 1.$
		\item If $I$ is of linear type and is generated in degree $d\le 3$, then 
		$$
		\deg(\FF) = \frac{(d-1)(d-2)}{2} + \dim_{\kk}\Big( {\left[I^{\text{sat}}/I\right]}_d \Big) + 1.
		$$
	\end{enumerate}
	\begin{proof}
		$(i)$ 
		It follows from \autoref{graded_case_degree_in_terms_of_Sym}.
		
		$(ii)$ From \autoref{degree_in_terms_of_Sym}$(ii)$, the linear type assumption implies $\deg(\FF) = \rank\Big( {\left[\HH_{\mm}^1\big(\Sym(I)\big)\right]}_0  \Big) + 1$.
		The spectral sequence $ {}^{\text{I}}E_1^{-p,q} = \HH_{\mm}^q(\ZZC_p)$ of the proof of \autoref{degree_in_terms_of_Sym}$(iii)$ is given by:
		\begin{center}		
			\begin{tikzpicture}
			\matrix (m) [matrix of math nodes,row sep=1em,column sep=3em,minimum width=2em, text height=1.5ex, text depth=0.25ex]
			{
				\HH_{\mm}^3(\ZZC_{2}) & \HH_{\mm}^3(\ZZC_{1}) & \HH_{\mm}^3(\ZZC_{0}) \\
				0 & \HH_{\mm}^2(\ZZC_{1}) & 0 \\
				0 & 0 & 0 \\
				0 & 0 & 0 \\
			};
			\path[-stealth]		
			(m-1-1) edge (m-1-2)	
			(m-1-2) edge (m-1-3)
			;
			\end{tikzpicture}	
		\end{center}
		Therefore, if we prove that ${\left[\HH_{\mm}^3(\ZZC_{1})\right]}_0=0$ then the convergence of this spectral sequence implies the required equality.
		Since $\ZZC_1=\big[Z_1 \otimes_R \AAA\big](d,-1)$, then it is enough to check that ${\left[\HH_{\mm}^3(Z_1)\right]}_d=0$.
		The short exact sequence
		$$
		0 \rightarrow Z_1 \rightarrow R^3(-d) \rightarrow I \rightarrow 0
		$$
		yields the following exact sequence
		$$
		0 \rightarrow \HH_{\mm}^2(I) \rightarrow \HH_{\mm}^3(Z_1) \rightarrow \HH_{\mm}^3(R^3)(-d),
		$$	
		and so we have ${\left[\HH_{\mm}^3(Z_1)\right]}_d \cong {\left[\HH_{\mm}^2(I)\right]}_d \cong {\left[\HH_{\mm}^1(R/I)\right]}_d$.
		Finally, from 
		\cite[Theorem 1.2$(ii)$]{Hassanzadeh_Simis_Cremona_Sat} we have that $\text{end}(\HH_{\mm}^1(R/I))\le 2d-4$, and so under the assumption $d \le 3$ we have ${\left[\HH_{\mm}^1(R/I)\right]}_d=0$. 
	\end{proof}
\end{proposition}

\section{Multi-graded Jacobian dual criterion of birationality}\label{sec:jacdualcriterion}

A rational map is birational if and only if its degree is equal to one, so the results we have previously developed provide birationality criteria. Nevertheless, because of its theoretical and practical importance, some more specific techniques have been developed to decide birationality, mostly for single-graded rational maps. In particular, it has been shown that birationality is controlled by a single numerical invariant that corresponds to the rank of a certain matrix called the \textit{Jacobian dual matrix} (see \cite{SIM_ULRICH_VASC_JACOBIAN}, \cite{Simis_cremona}, \cite[\S 2.3 and \S 2.4]{AB_INITIO} and \cite[Section 2.2]{EFFECTIVE_BIGRAD}). 
In this section, we extend this theory to the multi-graded setting. In \autoref{subsection_jacobian_criterion}, the multi-graded version of the Jacobian dual matrix is introduced and a general birationality criterion is proved (\autoref{jacob_main_crit}). 
As an illustration, a very simple birationality criterion is deduced for certain monomial multi-graded maps (\autoref{cor:monomialmaps}). Then, in \autoref{subsection_conseq_jacobian}, we investigate how birationality can be detected by using only the syzygies of the base ideal $I$ of a rational map, instead of the whole collection of equations of the Rees algebra of $I$ (\autoref{prop_linear_syz_imply_birat}), which are required for the Jacobian dual matrix. 
Under the assumption that $I$ is of linear type, we also obtain a syzygy-based birationality criterion (\autoref{thm:syzygy-jacobiandual}). 

\medskip

In this section we use the same notations and conventions of \autoref{subsection_prelim}.
If the dominant rational map $\FF : X = X_1 \times X_2 \times \cdots \times X_m \dashrightarrow Y$ has an inverse, then it is denoted by  
$$
\mathcal{G}: Y  \dashrightarrow (X_1, X_2, \ldots, X_m).
$$
Each rational map $Y \dashrightarrow X_i \subset \PP^{r_i}$ is defined by $r_i+1$ homogeneous polynomials $\mathbf{g}_i=\{g_{i,0}, g_{i,1},\ldots, g_{i,r_i}\} \subset S$ of the same degree. 
For each $i=1,\ldots,m$, we set $J_i$ to be the homogeneous ideal generated by $\mathbf{g}_i$.

\subsection{Jacobian dual matrices and the main criterion}
\label{subsection_jacobian_criterion}

We begin this section with the following preliminary lemma which is based on \cite[Lemma 1]{EFFECTIVE_BIGRAD}, \cite[Proposition 2.1]{Simis_cremona} and \cite[Theorem 2.18]{AB_INITIO}.

\begin{lemma}
	\label{lem_birat_implies_isom_Rees}
	Assume that $\mathcal{F}$ is a birational map with inverse $\mathcal{G}$.
	Let $I = (\mathbf{f})$ and $J_1 = (\mathbf{g}_1), \ldots, J_m = (\mathbf{g}_m)$. 
	Then, the identity map of $\kk[\mathbf{x},\mathbf{y}]$ induces a $\kk$-algebra isomorphism between the Rees algebra $\Rees_R(I)$ and the multi-graded Rees algebra $\Rees_S(J_1\oplus J_2 \oplus \cdots \oplus J_m)$.
	\begin{proof}
		First we note that both algebras can be identified as a quotient of $R \otimes_{\kk} S \cong \frac{\kk[\mathbf{x}, \mathbf{y}]}{\mathfrak{a}_1, \ldots, \mathfrak{a}_m, \mathfrak{b}}$.
		The algebra $\Rees_R(I)$ has a presentation given by 
		\begin{eqnarray*}
					\frac{\kk[\mathbf{x}]}{\mathfrak{a}_1,\ldots,\mathfrak{a}_m}[\mathbf{y}]   & \twoheadrightarrow & \Rees_R(I) = R[\mathbf{f}t]\\
					y_i  & \mapsto & f_it.
		\end{eqnarray*}
		Let $\overline{\REQ} = (\REQ, \mathfrak{a}_1,\ldots,\mathfrak{a}_m)/(\mathfrak{a}_1,\ldots,\mathfrak{a}_m)$ denote the kernel of this map. 
		Since  $Y$ can be identified with $\text{Proj}(\kk[\mathbf{f}])$ and the two algebras $\kk[\mathbf{f}]$ and $\kk[\mathbf{f}t]$ are isomorphic, then we get $\mathfrak{b} = \Ker(\kk[\mathbf{y}] \twoheadrightarrow \kk[\mathbf{f}]) = \Ker(\kk[\mathbf{y}] \twoheadrightarrow \kk[\mathbf{f}t]) \subset \REQ$, as required.
		
		Similarly, the algebra $\Rees_S(J_1\oplus \cdots \oplus J_m)$ has a presentation 
		\begin{eqnarray*}
		\frac{\kk[\mathbf{y}]}{\mathfrak{b}}[\mathbf{x}]&  \twoheadrightarrow & \Rees_S(J_1\oplus \cdots \oplus J_m) = S[\mathbf{g}_1t_1, \ldots, \mathbf{g}_mt_m]\\
		x_{i,j} & \mapsto & g_{i,j}t_i.
		\end{eqnarray*}
		We denote by $\overline{\mathcal{J}} = (\mathcal{J}, \mathfrak{b})/\mathfrak{b}$ the kernel of this map. 
		For each $i=1,\ldots,m$, we can identify $X_i$ with $\text{Proj}(\kk[\mathbf{g}_i])$ and as before we get $\mathfrak{a}_i = \Ker(\kk[\mathbf{x}_i] \twoheadrightarrow \kk[\mathbf{g}_i]) = \Ker(\kk[\mathbf{x}_i] \twoheadrightarrow \kk[\mathbf{g}_it_i]) \subset \mathcal{J}$.
		
		Since now we can regard $\Rees_R(I)$ and $\Rees_S(J_1\oplus\cdots\oplus J_m)$ as quotients of $\frac{\kk[\mathbf{x}, \mathbf{y}]}{\mathfrak{a}_1, \ldots, \mathfrak{a}_m, \mathfrak{b}}$, then it is enough to prove that $\mathcal{J} \subset (\REQ, \mathfrak{a}_1,\ldots,\mathfrak{a}_m)$ and that $\REQ \subset (\mathcal{J}, \mathfrak{b})$. 
		
		Let $F(\mathbf{y}, \mathbf{x}_1,\ldots,\mathbf{x}_m) \in \mathcal{J}$ be multi-homogeneous, then we have $$
		F(\mathbf{y},\mathbf{g}_1t_1, \ldots, \mathbf{g}_mt_m) = 0 \in S[\mathbf{g}_1t_1, \ldots, \mathbf{g}_mt_m]
		$$
		and using the multi-homogeneity of $F$ we get $F(\mathbf{y}, \mathbf{g}_1,\ldots,\mathbf{g}_m)=0 \in S$.
		From the canonical injection $S \cong \kk[\mathbf{f}] \hookrightarrow R$ we make the substitution $y_i \mapsto f_i$, and we obtain
		$$
		F(\mathbf{f}, \mathbf{g}_1(\mathbf{f}), \ldots, \mathbf{g}_m(\mathbf{f})) = 0 \in R.		 	
		$$
		By the assumption of $\mathcal{F}$ being birational, there exist nonzero multi-homogeneous forms $D_1, \ldots,D_m$ in $R$, possibly of different multi-degrees, such that 
		$$
		\mathbf{g}_1(\mathbf{f}) = D_1 \mathbf{x}_1,\; \mathbf{g}_2(\mathbf{f}) = D_2 \mathbf{x}_2,\; \ldots, \;\mathbf{g}_m(\mathbf{f}) = D_m \mathbf{x}_m.
		$$
		Again, from the multi-homogeneity of $F$ we get 
		$$
		F(\mathbf{f}, \mathbf{g}_1(\mathbf{f}), \ldots, \mathbf{g}_m(\mathbf{f})) = D_1^{\alpha_1}\cdots D_m^{\alpha_m} F(\mathbf{f}, \mathbf{x}_1, \ldots, \mathbf{x}_m) = 0 \in R,
		$$
		and so $F(\mathbf{f}, \mathbf{x}_1, \ldots, \mathbf{x}_m) = 0$ because $R$ is an integral domain. 
		From the identification $\kk[\mathbf{f}]\cong\kk[\mathbf{f}t]$ we get 
		$$
		F(\mathbf{f}t,\mathbf{x}_1,\ldots,\mathbf{x}_m) = 0 \in R[\mathbf{f}t],
		$$
		then by definition we get $F \in (\REQ, \mathfrak{a}_1,\ldots,\mathfrak{a}_m)$. Therefore, $\mathcal{J} \subset (\REQ, \mathfrak{a}_1,\ldots,\mathfrak{a}_m)$.
		
		We can prove the other containment with similar arguments. 
	\end{proof}		
\end{lemma}

Let $(\mathfrak{a}_1,\ldots,\mathfrak{a}_m,\REQ) \subset \kk[\mathbf{x},\mathbf{y}]$ be the defining equations of the Rees algebra $\Rees_R(I)$. We shall adopt the following notation.

\begin{notation}
	For each $1 \le i \le m$, let $\{h_{i,1},\ldots,h_{i,k_i}\}$ be a minimal set of generators of the multi-graded part of $(\mathfrak{a}_1,\ldots,\mathfrak{a}_m, \REQ)$ of multi-degree 
	$$
	(0,\ldots, \underbrace{1}_{i\text{-th}}, \ldots,0, *),
	$$
	where $*$ denotes arbitrary degree in $\mathbf{y}$.
	We denote by $\psi_i$ the Jacobian matrix of the collection of polynomials $\{h_{i,1},\ldots,h_{i,k_i}\}$ with respect to $\mathbf{x}_i$, that is
	$$
	\psi_i = \left( 
	\begin{array}{cccc}
	\frac{h_{i,1}}{\partial x_{i, 0}} & \frac{h_{i,1}}{\partial x_{i, 1}} & \cdots & \frac{h_{i,1}}{\partial x_{i, r_i}}	\\
	\frac{h_{i,2}}{\partial x_{i, 0}} & \frac{h_{i,2}}{\partial x_{i, 1}} & \cdots & \frac{h_{i,2}}{\partial x_{i, r_i}}	\\
	\vdots & \vdots & & \vdots \\		
	\frac{h_{i,k_i}}{\partial x_{i, 0}} & \frac{h_{i,k_i}}{\partial x_{i, 1}} & \cdots & \frac{h_{i,k_i}}{\partial x_{i, r_i}}
	\end{array} 
	\right).
	$$
	Following \cite{Simis_cremona,AB_INITIO,EFFECTIVE_BIGRAD}, the matrix $\psi_i$ will be called the $\mathbf{x}_i$-partial Jacobian dual matrix. We note that its entries are polynomials in $\kk[\yy]$.
	Finally, the matrix obtained by concatenating all the $\psi_i$'s in the main diagonal
	$$
	\psi = \left( 
	\begin{array}{cccc}
		\psi_1  & 0 & \cdots & 0	\\
		0 & \psi_2  & \cdots & 0	\\
		\vdots & \vdots & & \vdots \\		
		0 & 0 & \cdots & \psi_m
	\end{array} 
	\right)
	$$
	will be called the full Jacobian dual matrix.
\end{notation}

The next proposition is based on \cite[Proposition 2.15]{AB_INITIO}.
It shows that the ranks of the Jacobian dual matrices are sensitive to the dimensions of the source and the target.

\begin{proposition}
	\label{ineq_dim_jdrank}
	Let $\mathcal{F}: X_1 \times \cdots \times X_m \dashrightarrow Y$ be a dominant rational map.
	Then, we have the following inequalities:
	\begin{align}
	\label{general_ineq_sum}
	&\dim(X_1) + \cdots + \dim(X_m) - \dim(Y) \le \sum_{i=1}^{m} r_i - \sum_{i=1}^{m} \rank_S(\psi_i \otimes_{\kk[\mathbf{y}]} S).\\	
	\label{ineq_rank_phi_i}	
	&\rank_S(\psi_i \otimes_{\kk[\mathbf{y}]} S) \le r_i \text{ for each } i=1,\ldots,m.
	\end{align}
	\begin{proof}
		We begin with the first inequality. 	
		For each $1 \le i \le m$. 
		Let $E_i$ be the $S$-module $E_i = \Coker_S(\psi_i^t \otimes_{\kk[\mathbf{y}]} S)$ with presentation 
		$$
		S^{k_i} \xrightarrow{\psi_i^t \otimes_{\kk[\mathbf{y}]} S} S^{r_i+1} \rightarrow E_i \rightarrow 0. 
		$$  
		The direct sum $E = E_1 \oplus E_2 \oplus \cdots \oplus E_m$ is an $S$-module with presentation 
		$$
		S^{k_1}\oplus S^{k_2}\oplus\cdots \oplus S^{k_m} \xrightarrow{\psi^t \otimes_{\kk[\mathbf{y}]} S} S^{r_1+1}\oplus S^{r_2+1}\oplus\cdots \oplus S^{r_m+1} \rightarrow E \rightarrow 0.
		$$
		By the definition of the Jacobian dual matrices we have that $I_1(\mathbf{x} \cdot \psi^t) \subset \REQ$, and we saw in the proof of \autoref{lem_birat_implies_isom_Rees} that $\mathfrak{b} \subset \REQ$. Hence, we get a canonical surjective homomorphism $\Sym_S(E) \twoheadrightarrow \Rees_R(I)$ of $S$-algebras given by 
		\begin{align*}
		\alpha: \Sym_S(E) &\cong S[\mathbf{x}]/I_1(\mathbf{x} \cdot \big(\psi^t \otimes_{\kk[\mathbf{y}]} S)\big) \cong \kk[\mathbf{y}][\mathbf{x}]/\big(\mathfrak{b}, I_1(\mathbf{x} \cdot \psi^t)\big) \twoheadrightarrow R[\mathbf{y}]/\big(\mathfrak{b},I_1(\mathbf{x} \cdot \psi^t)\big) \\
		&\twoheadrightarrow \Rees_R(I).
		\end{align*}
		
		Following \cite{Rees_alg_modules}, we have that $\Rees_S(E) = \Sym_S(E) / \mathcal{T}$ where $\mathcal{T}$ represents the $S$-torsion submodule of $\Sym_S(E)$.
		Let $G \in \mathcal{T}$, there exists some $s \in S \setminus 0$ such that $s\cdot G = 0 \in \Sym_S(E)$. By using the isomorphisms $S \cong \kk[\mathbf{f}]  \cong \kk[\mathbf{f}t] \subset \Rees_R(I),$  we can see that 
		$$
		0=\alpha(s\cdot G)=\alpha(s)\alpha(G)=s(\mathbf{f}t)\alpha(G) \in \Rees_R(I)
		$$
		where $s(\mathbf{f}t) \neq 0$. 
		Since $\Rees_R(I)$ is an integral domain then it follows that $\alpha(G)=0$, and so we have a canonical surjective homomorphism 
		\begin{equation}
		\label{map_Rees_E_Rees_I}
		\Rees_S(E) \twoheadrightarrow \Rees_R(I)
		\end{equation}			
		of $S$-algebras. 
		
		Finally, from \cite[Proposition 2.2]{Rees_alg_modules} we get 
		\begin{align*}
		\dim\big(\Rees_R(I)\big) &\le \dim\big(\Rees_S(E)\big) \\
		\dim(R) + 1 &\le \dim(S) + m + \sum_{i=1}^{m} r_i - \sum_{i=1}^{m} \rank_S(\psi_i \otimes_{\kk[\mathbf{y}]} S), 					
		\end{align*}
		and using the equality $\dim(R) = \dim(A_1) + \cdots + \dim(A_m),$ we substitute 
		\begin{align*}
		\dim(X_1)+\cdots +\dim(X_m) + m + 1 &\le \dim(Y) + 1 + m + \sum_{i=1}^{m} r_i - \sum_{i=1}^{m} \rank_S(\psi_i \otimes_{\kk[\mathbf{y}]} S) \\
		\dim(X_1)+\cdots +\dim(X_m) - \dim(Y) &\le \sum_{i=1}^{m} r_i - \sum_{i=1}^{m} \rank_S(\psi_i \otimes_{\kk[\mathbf{y}]} S).					
		\end{align*}
		
		Now, we turn to the proof of the second claimed inequality. We follow one of the steps in the proof of \cite[Proposition 3]{EFFECTIVE_BIGRAD}. Fix $i=1,\ldots,m$.
		We have that $A_1 \otimes_{\kk} \cdots \otimes_{\kk} A_{i-1} \otimes_{\kk} A_{i+1} \otimes_{\kk} \cdots \otimes_{\kk} A_m$ is an integral domain, then let us denote by $\mathbb{L}$ its quotient field. 
		Let $X_i^{\prime}=\text{Proj}\big(\mathbb{L}[\mathbf{x}_i] / (\mathfrak{a}_i)\big)$, we define a rational map 
		$$
		\mathcal{F}^{\prime}: X_i^{\prime} \dashrightarrow Y'=\text{Proj}(S^{\prime}) \subset \PP_{\mathbb{L}}^s
		$$
		given by the classes of $f_0,\ldots,f_s$  inside $\mathbb{L}[\mathbf{x}_i] / (\mathfrak{a}_i)$, and  we denote $S^{\prime}:= \mathbb{L}[\mathbf{f}]$.
		Using the field inclusion $\kk \hookrightarrow \mathbb{L}$ we can check that any polynomial in the defining equations of the Rees algebra $\Rees_R(I)$ is also contained in the defining equations of the Rees algebra $\Rees_{\mathbb{L}[\mathbf{x}_i] / (\mathfrak{a}_i)}\big(I\big)$. 
		In particular, we have that the row space of $\psi_i \otimes_{\kk[\mathbf{y}]} S$ is contained in the row space of $\psi^{\prime} \otimes_{\mathbb{L}[\mathbf{y}]} S^{\prime}$, where $\psi^{\prime}$ denotes the Jacobian dual matrix of $\mathcal{F}^{\prime}$.
		Hence, $\rank_S(\psi_i \otimes_{\kk[\mathbf{y}]} S) \le \rank_{S^{\prime}}(\psi^{\prime} \otimes_{\mathbb{L}[\mathbf{y}]} S^{\prime}) \le r_i$, and the last inequality follows from \cite[Corollary 2.16]{AB_INITIO} or \autoref{general_ineq_sum}.
	\end{proof}

\end{proposition}

The following birationality criterion is the main result of this section; it is the multi-graded version of \cite[Theorem 2.18]{AB_INITIO} and \cite[Theorem 2]{EFFECTIVE_BIGRAD}.

\begin{theorem}
	\label{jacob_main_crit}
	Let $\mathcal{F}:X_1\times \cdots \times X_m \dashrightarrow Y$ be a dominant rational map.
	Then, the following three conditions are equivalent:
	\begin{enumerate}[(i)]
		\item $\FF$ is birational.
		\item $\rank_S(\psi_i \otimes_{\kk[\mathbf{y}]} S) = r_i	$\,
		for each $i=1,\ldots,m$.
		\item $\rank_S(\psi \otimes_{\kk[\mathbf{y}]} S) = r_1+r_2+\cdots+r_m$. 
	\end{enumerate}

	In addition, if $\mathcal{F}$ is birational then its inverse is of the form $\mathcal{G} : Y \dashrightarrow X_1\times \cdots \times X_m$, where each map $Y \dashrightarrow X_i$ is given by the signed ordered maximal minors of an $r_i \times (r_i+1)$ submatrix of $\psi_i$ of rank $r_i$.
	\begin{proof}
		$(i) \Rightarrow (ii)$.  Let us suppose that $\mathcal{F}$ is birational. 
		From \autoref{lem_birat_implies_isom_Rees} we get an isomorphism $\Rees_R(I) \cong \Rees_S(J_1\oplus\cdots\oplus J_m)$ induced by the identity of $\kk[\mathbf{x}, \mathbf{y}]$.
		So we obtain the equality $(\REQ,\mathfrak{a}_1,\ldots,\mathfrak{a}_m)=(\mathcal{J},\mathfrak{b})$ that in particular gives us
		\begin{equation}
		\label{equality_graded_parts}
		{\left[(\REQ,\mathfrak{a}_1,\ldots,\mathfrak{a}_m)\right]}_{(\underbrace{0,\ldots,1,\ldots,0}_{\mathbf{x}},\underbrace{*}_{\mathbf{y}})} =			{\left[(\mathcal{J},\mathfrak{b})\right]}_{(\underbrace{*}_{\mathbf{y}},\underbrace{0,\ldots,1,\ldots,0}_{\mathbf{x}})}
		\end{equation}
		for each $i=1,\ldots,m$.
		By reducing modulo $\mathfrak{b}$, the right hand side of \autoref{equality_graded_parts} yields a presentation 
		$$
		0 \rightarrow {\left[(\mathcal{J},\mathfrak{b})/\mathfrak{b}\right]}_{(*,0,\ldots,1,\ldots,0)} \rightarrow S[\mathbf{x}_i] \rightarrow \Sym_S(\mathbf{g}_i) \rightarrow 0
		$$
		of the symmetric algebra $\Sym_S\left(\mathbf{g_i}\right)$ of $(\mathbf{g_i})$.
		On the other hand, from the definition of Jacobian dual matrices we have
		$$
		{\left[(\REQ,\mathfrak{a}_1,\ldots,\mathfrak{a}_m)/\mathfrak{b}\right]}_{(0,\ldots,1,\ldots,0,*)} \,=\, I_1\big(\mathbf{x}_i \cdot (\psi_i^t \otimes_{\kk[\mathbf{y}]} S)\big).
		$$
		Let $\Syz_S(\mathbf{g}_i)$ be the matrix of syzygies of $(\mathbf{g}_i)$. 
		By the two previous reductions of \autoref{equality_graded_parts}, we obtain 
		\begin{equation}
		\label{containt_Sym_E_J}
		I_1\big(\mathbf{x}_i \cdot (\psi_i^t \otimes_{\kk[\mathbf{y}]} S)\big) = I_1(\mathbf{x}_i \cdot \Syz_S(\mathbf{g}_i)).
		\end{equation}
		Since both matrices $\psi_i^t \otimes_{\kk[\mathbf{y}]} S$ and $\Syz_S(\mathbf{g}_i)$ have entries in $S,$ the column space of $\Syz_S(\mathbf{g}_i)$ is equal to the one of $\psi_i^t \otimes_{\kk[\mathbf{y}]} S$. 
		Finally, the fact that $\rank_S(\Syz_S(\mathbf{g}_i))=r_i$ implies that $\rank_S(\psi_i^t \otimes_{\kk[\mathbf{y}]} S) = r_i$. 
		
		$(ii) \Rightarrow (i)$. We assume that $\rank_S(\psi_i \otimes_{\kk[\mathbf{y}]} S) = r_i$ for each $i=1,\ldots,m$. 
		Fix $i = 1,\ldots,m$. 
		Let $M_i$ be a $r_i\times (r_i+1)$ submatrix of $\psi_i$ such that $\rank_S\left(M_i \otimes_{\kk[\yy]}S\right)=r_i$.
		Denote by $\Delta_0(\mathbf{y}), \Delta_1(\mathbf{y}),\cdots,\Delta_{r_i}(\mathbf{y})$ the ordered signed minors of $M_i^t.$ The Hilbert-Koszul lemma (\cite[Propositon 2.1]{AB_INITIO}) implies that the vector $e_a\Delta_b(\mathbf{y}) - e_b\Delta_a(\mathbf{y})$ belongs to the column space of $M_i^t$, and so also to the one of $\psi_i^t$.
		Since by definition $I_1(\mathbf{x}_i \cdot \psi_i^t)=			{\left[\left(\REQ,\mathfrak{a}_1,\ldots,\mathfrak{a}_m\right)\right]}_{(0,\ldots,1,\ldots,0,*)}$, we get $x_{i,a}\Delta_b(\mathbf{y}) - x_{i,b}\Delta_a(\mathbf{y}) \in {\left[\left(\REQ,\mathfrak{a}_1,\ldots,\mathfrak{a}_m\right)\right]}_{(0,\ldots,1,\ldots,0,*)}$.
		
		Making a substitution via the canonical homomorphism $\kk[\xx,\yy] \rightarrow \Rees_R(I)$, we automatically get 
		$$
		x_{i,a}\Delta_b(\mathbf{f}) - x_{i,b}\Delta_a(\mathbf{f})=0 \in R, \quad \text{for every pair }  a,b.
		$$
		From the inclusion $S \cong \kk[\mathbf{f}] \cong \kk[\mathbf{f}t] \subset \Rees_R(I)$ and the rank assumption, we have that the tuple 
		$$
		\left(\Delta_0(\mathbf{f}),\ldots,\Delta_{r_i}(\mathbf{f})\right)
		$$ 
		does not vanish on $R$.
		Let $\mathcal{G} : Y \dashrightarrow X_1 \times \cdots \times X_m$ where each map $Y \rightarrow X_i$ is given by the tuple $\left(\Delta_0(\mathbf{y}),\ldots,\Delta_{r_i}(\mathbf{y})\right)\otimes_{\kk[\mathbf{y}]} S.$ We have proven that $\mathcal{G}$ is the inverse of $\mathcal{F}$.
		
		$(ii) \Leftrightarrow (iii)$.
		This part follows from the inequalities of \autoref{ineq_rank_phi_i} and the fact that $\rank_S(\psi \otimes_{\kk[\mathbf{y}]} S) = \sum_{i=1}^{m} \rank_S(\psi_i \otimes_{\kk[\mathbf{y}]} S)$.
	\end{proof}
\end{theorem}

To illustrate this theorem, we provide two corollaries. The first one is a rigorous translation of birationality in terms of an isomorphism between the corresponding  Rees algebras; this result is the multi-graded version of \cite[Proposition 2.1]{Simis_cremona}. The second is a specific birationality criterion dedicated to some particular monomial maps. 

\begin{corollary}
	The rational map $\mathcal{F}:X=X_1\times \cdots \times X_m \dashrightarrow Y$ is birational with inverse $\mathcal{G}$ if and only if  $\FF$ is dominant, the image of $\mathcal{G}$ is $X$, and the identity map of $\kk[\mathbf{x},\mathbf{y}]$ induces a $\kk$-algebra isomorphism between the Rees algebra $\Rees_R(I)$ and the multi-graded Rees algebra $\Rees_S(J_1\oplus J_2 \oplus \cdots \oplus J_m)$.
	\begin{proof}		
		One implication was proved in \autoref{lem_birat_implies_isom_Rees}. 
		Let us assume that $\FF$ and $\mathcal{G}$ are dominant and the identity map of $\kk[\mathbf{x},\mathbf{y}]$ induces an isomorphism between $\Rees_R(I)$ and $\Rees_S(J_1\oplus\cdots\oplus J_m)$.
		
		As in \autoref{ineq_dim_jdrank}, let $E = \Coker_S(\psi^t \otimes_{\kk[\mathbf{y}]} S)$. 
		Identity \autoref{containt_Sym_E_J} gives us a canonical isomorphism of $S$-algebras
		$$
		 \Sym_S(E)  \cong \Sym_S(J_1\oplus \cdots \oplus J_m).
		$$
		From the assumption $\Rees_S(J_1\oplus\cdots\oplus J_m) \cong \Rees_R(I)$, we get the following isomorphisms
		$$
		\Rees_S(E) \cong \Rees_S(J_1\oplus\cdots\oplus J_m) \cong \Rees_R(I),
		$$
		which are induced by the identity map on $\kk[\mathbf{x},\mathbf{y}]$.
		
		Performing the same computation of \autoref{ineq_dim_jdrank}, now we  get 	
	\begin{align*}
		\dim(S) + \sum_{i=1}^m(r_i+1) - \rank_S(\psi \otimes_{\kk[\mathbf{y}]} S) &= \dim(R) +1\\
		\dim(Y) + m +1 + \sum_{i=1}^m r_i -  \rank_S(\psi \otimes_{\kk[\mathbf{y}]} S) &= \dim(X_1)+\cdots+\dim(X_m)+m+1.
	\end{align*}
	Since $\FF$ and $\mathcal{G}$ are dominant,  we have $\dim(Y)=\dim(X_1)+\cdots+\dim(X_m)$.
	So it follows that 
	$
	\rank_S(\psi \otimes_{\kk[\mathbf{y}]} S) = \sum_{i=1}^{m}r_i.
	$ 
		
		Therefore, from \autoref{jacob_main_crit} we have that $\mathcal{F}$ is birational. Let us denote by $\mathcal{G}^{\prime}$ its inverse.
		Let $J_1^{\prime}, \ldots, J_m^{\prime}$ be the base ideals of $\mathcal{G}^{\prime}$.
		Applying  \autoref{lem_birat_implies_isom_Rees}, we have that the identity map of $\kk[\mathbf{x},\mathbf{y}]$ induces the following isomorphisms
		$$
		\Rees_S(J_1^{\prime} \oplus \cdots \oplus J_m^{\prime}) \cong \Rees_R(I) \cong \Rees_S(J_1 \oplus \cdots \oplus J_m).
		$$
		In particular, we have an isomorphism between the symmetric algebras of $J_1\oplus\cdots\oplus J_m$ and $J_1^{\prime}\oplus\cdots\oplus J_m^{\prime}$ over $S$, which implies an equality between their syzygies.
		Therefore, the tuples defining $\mathcal{G}$ and $\mathcal{G}^{\prime}$ are proportional and so they define the same rational map.
	\end{proof}
\end{corollary}

Now, we focus on the case of a monomial multi-graded rational map $\FF:{\left(\PP^1\right)}^s \dashrightarrow \PP^s$. We assume that $I=\left(\xx^{\alpha_0},\xx^{\alpha_1},\ldots,\xx^{\alpha_s}\right)$, where each $\alpha_i=\left(\alpha_{i,1},\ldots,\alpha_{i,2s}\right)$ is a vector of $2s$ entries, and $\xx^{\alpha_i}$ denotes the monomial
$$
\xx^{\alpha_i} = x_{1,0}^{\alpha_{i,1}}x_{1,1}^{\alpha_{i,2}}\;\cdots\; x_{s,0}^{\alpha_{i,2s-1}}x_{s,1}^{\alpha_{i,2s}}.
$$

In this setting, the presentation \autoref{presentation_Rees} of $\Rees(I)$ can be encoded by the following matrix: 
\begin{equation}
\label{matrix_def_Rees}
M = \left( 
\begin{array}{cccccccc}
e_1 & e_2 & \ldots & e_{2s}  & \alpha_{0,1} & \alpha_{1,1} & \ldots & \alpha_{s,1} \\
&&&&\vdots & \vdots & \vdots & \vdots \\
&&&&\alpha_{0,2s} & \alpha_{1,2s} & \ldots & \alpha_{s,2s} \\
&&&&1 & 1 & \ldots & 1
\end{array}
\right),
\end{equation}
where $e_1, e_2, \ldots e_{2s}$ are the first $2s$ unit vectors in $\mathbb{Z}^{2s+1}$.
For any integer vector $\beta \in \ZZ^{3s+1}$, we denote by ${\xx\yy}^{\beta}$ the following monomial
$$
{\xx\yy}^{\beta}=x_{1,0}^{\beta_{1}}x_{1,1}^{\beta_{2}}\;\cdots\; x_{s,0}^{\beta_{2s-1}}x_{s,1}^{\beta_{2s}}\;y_0^{\beta_{2s+1}}y_1^{\beta_{2s+2}}\;\cdots\;y_s^{\beta_{3s+1}}.
$$
The ideal of defining equations of $\Rees(I)$ is a toric ideal (see \cite[Chapter 4]{STURMFELS_GROBNER}). It is generated by the following binomials 
\begin{equation}
\label{descript_toric_case}
\REQ = \left({\xx\yy}^{\beta^+}-{\xx\yy}^{\beta^-} \mid M\beta=0, \;\; \beta=\beta^+-\beta^-,\;\; \beta^+,\beta^-\ge 0\right).
\end{equation}

The following corollary contains a very effective way of testing the birationality of $\FF$, which can be done for instance by using Hermite normal form algorithms.

\begin{corollary}\label{cor:monomialmaps}	
	Let $\FF:{\left(\PP^1\right)}^s \dashrightarrow \PP^s$ be a monomial dominant multi-graded rational map. 
	Let $A$ be the submatrix of $M$ in \autoref{matrix_def_Rees} given by the last $s+1$ columns.
	Then, $\FF$ is birational if and only if the following conditions 	are satisfied for each $i=1,\ldots,s$:
	\begin{equation}
	\label{sys_eqs_toric_case}
	\big\{\gamma \in \ZZ^{s+1} \;  \mid  \; A \gamma =  e_{2i-1} - e_{2i} \big\} \neq \emptyset.
	\end{equation}	
	\begin{proof}
		From \autoref{jacob_main_crit} we only need to check that 
		$$
		\REQ_{(0,\ldots,\underbrace{1}_{i\text{-th}},\ldots,0,*)}\neq 0 \quad\text{ for each }\quad i=1,\ldots,m.
		$$ 
		By the description of \autoref{descript_toric_case}, this inequality is equivalent to the solution of the  systems of equations given in \autoref{sys_eqs_toric_case}.
	\end{proof}
\end{corollary}

\subsection{Linear syzygies and some consequences}
\label{subsection_conseq_jacobian}

The birationality criterion provided in \autoref{jacob_main_crit} requires the computation of the equations of the Rees algebra of the base ideal of a rational map. In this subsection, we investigate how the syzygies of the base ideal can be used instead in order to deduce, or to characterize, the birationality of a multi-graded rational map.

\begin{notation}
	Let $\varphi$ be the matrix of syzygies of $I$ whose entries are multi-homogeneous polynomials. We denote by $\varphi_1$ the submatrix of $\varphi$ whose columns are the columns of $\varphi$ corresponding to syzygies of $I$ of multi-degree $(1,\ldots,0),$ $(0,1,0,\ldots)$, $\ldots$, or $(0,\ldots,1)$.
	The matrix $\varphi_1$ is called the linear part of the matrix $\varphi$.
\end{notation}

The following proposition is based on \cite[Theorem 3.2]{AB_INITIO} and \cite[Proposition 3]{EFFECTIVE_BIGRAD}.

\begin{proposition}
	\label{prop_linear_syz_imply_birat}
	Let $\mathcal{F}: \PP^{r_1} \times \cdots \times \PP^{r_m} \dashrightarrow \PP^{r_1+\cdots+r_m}$ be a dominant rational map. 
	If $\rank(\varphi_1)=r_1+\cdots+r_m$, then $\mathcal{F}$ is birational.
	\begin{proof}
		We choose a matrix $\rho$  with entries in $S$ such that $\yy \cdot \varphi_1 = \xx \cdot \rho$.
		Let $E = \Coker(\rho)$, then the previous equality gives us the isomorphism $\Sym_R\left(\Coker(\varphi_1)\right)\cong \Sym_S\left(E\right)$.
		We present the Rees algebras $\Rees_R(\Coker(\varphi_1))$ and $\Rees_S(E)$ by 
		$$
		\Rees_R(\Coker(\varphi_1)) = \frac{\kk[\mathbf{x},\mathbf{y}]}{(I_1(\mathbf{y}\cdot \varphi_1), \mathcal{T}_1)} \quad\text{ and }\quad 			\Rees_S(E) = \frac{\kk[\mathbf{x},\mathbf{y}]}{(I_1(\mathbf{x}\cdot \rho), \mathcal{T}_2)},
		$$
		where $\mathcal{T}_1$ represents the $R$-torsion of $\Sym_R(\Coker(\varphi_1))$ and $\mathcal{T}_2$ is the $S$-torsion of $\Sym_S(E)$, both lifted to $\kk[\mathbf{x},\mathbf{y}]$.
		Since $S$ is an integral domain and $E$ has rank, then $\Rees_S(E)$ is an integral domain and so $(I_1(\mathbf{x}\cdot \rho), \mathcal{T}_2)$ is a prime ideal.
		
		Let $G(\mathbf{x}, \mathbf{y}) \in \T_1.$ There exists $F(\mathbf{x}) \in \kk[\mathbf{x}] \setminus 0$ such that $F(\mathbf{x})G(\mathbf{x}, \mathbf{y}) \in I_1(\mathbf{y} \cdot \varphi_1) \subset ( I_1(\mathbf{x}\cdot \rho), \mathcal{T}_2)$.
		We assume $G(\mathbf{x}, \mathbf{y}) \not\in (I_1(\mathbf{x}\cdot \rho), \mathcal{T}_2)$, then it follows that $F(\mathbf{x}) \in \T_2$ due to the fact that $(I_1(\xx\cdot \rho), \mathcal{T}_2)$ is prime and the ideal $ I_1(\mathbf{x}\cdot \rho)$ is generated by multi-homogeneous polynomials with positive degree on $\mathbf{y}$.
		Thus, there exists a polynomial $H(\mathbf{y}) \in \kk[\mathbf{y}] \setminus
		0$ such that $H(\mathbf{y})F(\mathbf{x}) \in I_1(\mathbf{x}\cdot \rho)$.
		Since $I_1(\xx\cdot\rho)$ is generated by syzygies of $I$, when we substitute $y_j\mapsto f_j$ then we get $H(\mathbf{f})F(\mathbf{x})=0$.
		From the fact that 		$H(\mathbf{f})\neq 0$ (note that here we have $S\cong\kk[\yy]$), it follows the contradiction $F(\mathbf{x})=0$.
		
		Therefore, we have a surjective $R$-algebra map $\Rees_R(\Coker(\varphi_1)) \twoheadrightarrow \Rees_S(E),$ and so we get the inequality
		\begin{align*}
		\dim(\Rees_S(E)) &\le \dim(\Rees_R(\Coker(\varphi_1))) \\
		\dim(S)  + \sum_{i=1}^m(r_i+1) - \rank(\rho) &\le \dim(R) + 1 + \sum_{i=1}^{m}r_i  - \rank(\varphi_1). 
		\end{align*}
		Substituting $\rank(\varphi_1) = \sum_{i=1}^{m}r_i$, $\dim(S) = 1+\sum_{i=1}^mr_i$ and $\dim(R)=\sum_{i=1}^{m}(r_i+1),$  we get 
		$$
		\sum_{i=1}^mr_i \le \rank(\rho).
		$$
		The inclusion  $I_1\left(\xx\cdot\rho\right) \subset I_1\left(\xx\cdot\psi^t\right)$
		gives us the inequality $\rank(\rho) \le \rank(\psi^t)$.
		Combining this with \autoref{ineq_dim_jdrank} we necessarily get 	$\rank_S(\psi \otimes_{\kk[\mathbf{y}]} S) = \sum_{i=1}^{m}r_i$.
		Therefore, the result follows from \autoref{jacob_main_crit}.
	\end{proof}
\end{proposition}

The above proposition gives a sufficient syzyzy-based property to ensure birationality. In the next result we prove that it becomes also a necessary condition under the assumption that the base ideal is of linear type. This effective birationality criterion is the multi-graded version of \cite[Proposition 3.4]{AB_INITIO}.

\begin{theorem}\label{thm:syzygy-jacobiandual}
	Let $\mathcal{F}: \PP^{r_1} \times \cdots \times \PP^{r_m} \dashrightarrow
	\PP^{r_1+\cdots+r_m}$ be a rational map whose base ideal $I=(\mathbf{f})$ is of linear type.
	Then, the following conditions are equivalent:
	\begin{enumerate}[(i)]
		\item $\mathcal{F}$ is birational.
		\item $\rank(\varphi_1)=r_1+\cdots+r_m$. 
	\end{enumerate}
\end{theorem}

To prove this theorem, we will need the following preliminary lemma on the torsion of symmetric algebras in the multi-graded setting. It is essentially an adaptation of \cite{MICALI_REES} to the multi-graded case. As we are following the general setup of \cite{Rees_alg_modules}, $\Rees_R(I_1 \oplus \cdots \oplus I_n)$ means  $\Sym_R(I_1\oplus \cdots \oplus I_n)$ modulo its $R$-torsion. 

\begin{lemma}
	\label{lem_torsion_multi_sym_alg}
	Let $R$ be a Noetherian commutative ring and $I_1,\ldots,I_n$ be ideals having rank one. 
	Then, we have the following relation between (multi-graded) symmetric and Rees algebras
	$$
	\Rees_R(I_1\oplus \cdots \oplus I_n) = \frac{\Sym_R(I_1\oplus \cdots \oplus I_n)}{\HH^0_{I_1\cdots I_n}\big( \Sym_R(I_1\oplus \cdots \oplus I_n) \big)}.
	$$
	In particular, if $R$ is local with maximal ideal $\mm$ and each $I_i$ is $\mm$-primary then we have
	$$
	\Rees_R(I_1\oplus \cdots \oplus I_n) = \frac{\Sym_R(I_1\oplus \cdots \oplus I_n)}{\HL^0\big( \Sym_R(I_1\oplus \cdots \oplus I_n) \big)}.
	$$
	\begin{proof}
		As part of the proof of this lemma we shall obtain that $\Rees_R(I_1 \oplus \cdots \oplus I_n)$  coincides with the usual multi-graded Rees algebra
		$$
		\Rees_R(I_1,\ldots,I_n) = \bigoplus_{i_1,\ldots,i_n} I_1^{i_1}\cdots I_n^{i_n}t_1^{i_1}\cdots t_n^{i_n}.
		$$
		
		By the assumption that each ideal $I_i$ has rank one then we have $\grade(I_i) \ge 1$, and from the Unmixedness Theorem (see e.g.~\cite[Exercise 1.2.21]{BRUNS_HERZOG},  \cite[Theorem 125]{KAP}) we can assume that $I_i=(\mathbf{f}_i)$ where $\mathbf{f}_i=(f_{i,1},\ldots,f_{i,m_i})$ and each $f_{i,j}$ is an $R$-regular element.
		
		Let $\AAA$ be the polynomial ring $\AAA=R[\mathbf{T}_1,\ldots, \mathbf{T}_n]$ where $\mathbf{T}_i=\{T_{i,1},\ldots,T_{i,m_i}\}$ for each $i=1,\ldots,n$.
		The symmetric algebra $\Sym_R(I_1\oplus \cdots \oplus I_n)$ can easily be presented by
		\begin{equation*}
		0 \rightarrow \REQ_1 \rightarrow \AAA \rightarrow \Sym_R(I_1 \oplus \cdots \oplus I_n) \rightarrow 0,
		\end{equation*}
		where $\REQ_1 = \big( I_1(\mathbf{T}_1 \cdot \Syz(\mathbf{f}_i)), \ldots, I_1(\mathbf{T}_n \cdot \Syz(\mathbf{f}_n)) \big)$.
		On the other hand, the Rees algebra can be presented by 
		\begin{align*}
		0 \rightarrow \REQ \rightarrow \AAA &\rightarrow \Rees_R(I_1, \ldots, I_n) \rightarrow 0\\
		\mathbf{T}_i &\mapsto \mathbf{f}_it_i,
		\end{align*}
		where $\REQ$ is the ideal generated by the multi-homogeneous polynomials $F(\mathbf{T}_1, \ldots, \mathbf{T}_n) \in \AAA$ such that $F(\mathbf{f}_1, \ldots, \mathbf{f}_n) = 0$.
		Therefore, we want to analyze the canonical exact sequence 
		$$
		0 \rightarrow (\REQ/\REQ_1) \rightarrow \Sym_R(I_1 \oplus \cdots \oplus I_n) \xrightarrow{\alpha} \Rees_R(I_1, \ldots, I_n) \rightarrow 0.
		$$
		
		It is clear that the $R$-torsion submodule of $\Sym_R(I_1 \oplus \cdots \oplus I_n)$ is contained in $\Ker(\alpha)$, and in particular, by the assumption on the ideals $I_i$, the elements of $\Sym_R(I_1 \oplus \cdots \oplus I_n)$ annihilated by some power ${(I_1\cdots I_n)}^l$ are also contained in $\Ker(\alpha)$.
		If we prove that any element in $\Ker(\alpha)$ is contained in the $R$-torsion submodule of $\Sym_R(I_1\oplus \cdots \oplus I_n)$ and  is annihilated by some power ${(I_1\cdots I_n)}^l$, then we are done because we get the following equality and isomorphism
		$$
		\Rees_R(I_1 \oplus \cdots \oplus I_n) = \frac{\Sym_R(I_1\oplus \cdots \oplus I_n)}{\HH^0_{I_1\cdots I_n}\big( \Sym_R(I_1\oplus \cdots \oplus I_n) \big)} \cong \Rees_R(I_1, \ldots,I_n).
		$$
		
		By the assumption that all the $f_{i,j}$ are $R$-regular, the proof of the two previous assertions will follow from the next claim.

		\noindent \textbf{Claim.} Let $F \in \REQ.$  Then, for any element of the form $f_{1,j_1}f_{2,j_2}\cdots f_{n,j_n}$ (i.e. a generator of $I_1\cdots I_n$), there exists some integer $l> 0$ such that $(f_{1,j_1}f_{2,j_2}\cdots f_{n,j_n})^l F \in \REQ_1$.
		
		\noindent
	    {\textit{Proof of the claim.}}
		Fix any generators $f_{1,j_1} \in I_1,  f_{2,j_2} \in I_2, \ldots, f_{n,j_n} \in I_n$.
		Let $F \in \REQ$ be multi-homogeneous  of multi-degree $(d_1,d_2,\ldots,d_n)$ we shall proceed by induction on $d=d_1+\cdots +d_n$. 
		In the inductive step,  it is enough to prove that there exists integers $\alpha_1\ge 0,\ldots,\alpha_n\ge 0$ such that
		$$f_{1,j_1}^{\alpha_1}f_{2,j_2}^{\alpha_2}\cdots f_{n,j_n}^{\alpha_n} F \in \REQ_1.
		$$
		If $d = 1$ then $F$ clearly satisfies the previous condition.  
		So, we assume that $d>1$ and by simply ordering the variables $\mathbf{T}_i$ we may suppose that $d_1\ge 1$. 
		We can write $F$ in the following way
		$$
		F = \sum_{k=1}^{m_1} T_{1,k}H_{k}\Big( T_{1,k},\ldots,T_{1,m_1},\mathbf{T}_2, \ldots, \mathbf{T}_n\Big)
		$$
		Then we define the following polynomial 
		$$
		G = \sum_{k=1}^{m_1} T_{1,k}H_{k}\Big( f_{1,k},\ldots,f_{1,m_1},\mathbf{f}_2, \ldots, \mathbf{f}_n\Big)
		$$
		which belong $\REQ_1$.  We compute the polynomial 
		\begin{align*}
		f_{1,j_1}^{d_1-1}			f_{2,j_2}^{d_2}\cdots
		f_{n,j_n}^{d_n} F - 			T_{1,j_1}^{d_1-1}			T_{2,j_2}^{d_2}&\cdots
		T_{n,j_n}^{d_n} G = \\\sum_{k=1}^{m_1} T_{1,k}\Bigg( f_{1,j_1}^{d_1-1}			f_{2,j_2}^{d_2}&\cdots
		f_{n,j_n}^{d_n}H_{k}\Big( T_{1,k},\ldots,T_{1,m_1},\mathbf{T}_2, \ldots, \mathbf{T}_n\Big) -\\ &T_{1,j_1}^{d_1-1}			T_{2,j_2}^{d_2}\cdots
		T_{n,j_n}^{d_n} 
		H_{k}\Big( f_{1,k},\ldots,f_{1,m_1},\mathbf{f}_2, \ldots, \mathbf{f}_n\Big)
		\Bigg),
		\end{align*}
		where each  polynomial 
		$$
		f_{1,j_1}^{d_1-1}			f_{2,j_2}^{d_2}\cdots
		f_{n,j_n}^{d_n}H_{k}\Big( T_{1,k},\ldots,T_{1,m_1},\mathbf{T}_2, \ldots, \mathbf{T}_n\Big) - T_{1,j_1}^{d_1-1}			T_{2,j_2}^{d_2}\cdots
		T_{n,j_n}^{d_n} 
		H_{k}\Big( f_{1,k},\ldots,f_{1,m_1},\mathbf{f}_2, \ldots, \mathbf{f}_n\Big)
		$$
		belongs to $\REQ$ and has total degree smaller than $d$.
		Therefore, the proof of the claim follows by induction. 
	\end{proof}
\end{lemma}

	\begin{proof}[Proof of \autoref{thm:syzygy-jacobiandual}]
		$(ii) \Rightarrow (i)$ Since $I$ is of linear type then the polynomials of $\mathbf{f}$ are algebraically independent. Therefore, this implication follows from \autoref{prop_linear_syz_imply_birat}.
		
		$(i) \Rightarrow (ii)$ 
		From the assumption of $\mathcal{F}$ being birational,  let $\mathbf{g}_1, \ldots, \mathbf{g}_m$ be a set of representatives of the inverse map $\mathcal{G}: \PP^{r_1+\cdots+r_m} \dashrightarrow \PP^{r_1} \times \cdots \times \PP^{r_m}$.
		
		Since $I$ is of linear type, we have $\REQ = I_1(\mathbf{y} \cdot \varphi)$ and so we obtain the following equality 
		\begin{equation}
		\label{eq_linear_syz_Jac_dual_mat}
		I_1\left(\mathbf{y} \cdot \varphi_1\right) = I_1\left(\mathbf{x} \cdot \psi^t\right).
		\end{equation}
		Due to the isomorphism obtained in \autoref{lem_birat_implies_isom_Rees}, the module $(\mathbf{g}_1)\oplus \cdots \oplus (\mathbf{g}_m)$ has the following presentation 
		$$
		S^p \xrightarrow{\psi^t} S^{r_1+\cdots+r_m+m} \rightarrow (\mathbf{g}_1)\oplus \cdots \oplus (\mathbf{g}_m) \rightarrow 0.
		$$
		We also also consider the module $E = \Coker(\varphi_1)$ with presentation 
		$$
		R^p \xrightarrow{\varphi_1} R^{r_1+\cdots r_m + 1} \rightarrow E \rightarrow 0.
		$$
		
		From the equality \autoref{eq_linear_syz_Jac_dual_mat} we get an isomorphism $\Sym_S\big((\mathbf{g}_1)\oplus \cdots \oplus (\mathbf{g}_m)\big) \cong \Sym_R(E)$ induced by the identity map of $\kk[\mathbf{x}, \mathbf{y}]$.
		Then, we have the following 
		$$
		\Sym_S\big((\mathbf{g}_1)\oplus \cdots \oplus (\mathbf{g}_m)\big) \cong \Sym_R(E) \twoheadrightarrow \Rees_R(I).
		$$
		
		Let $\mathcal{T}$ be the $S$-torsion of $\Sym_S\big((\mathbf{g}_1)\oplus \cdots \oplus (\mathbf{g}_m)\big)$ and $\lambda$ be the isomorphism 
		$$
		\lambda : \Sym_S\big((\mathbf{g}_1)\oplus \cdots \oplus (\mathbf{g}_m)\big) \xrightarrow{\cong} \Sym_R(E)
		$$
		If we prove that $\lambda(\T)$ is contained in the $R$-torsion of $\Sym_R(E)$,
		we will get the following epimorphism
		$$
		\Rees_S\big((\mathbf{g}_1)\oplus \cdots \oplus (\mathbf{g}_m)\big) \twoheadrightarrow \Rees_R(E) \twoheadrightarrow \Rees_R(I).
		$$
		Therefore, from \autoref{lem_birat_implies_isom_Rees} we  get $\Rees_S((\mathbf{g}_1)\oplus \cdots \oplus (\mathbf{g}_m)) \cong \Rees_R(E) \cong \Rees_R(I)$ and so $\text{rank}(E)=1$ which implies the statement. 
		
		Thus we shall focus on the claim below:\\
		\noindent \textbf{Claim:} $\lambda(\T)$ is contained in the $R$-torsion of $\Sym_R(E)$.
		
		\noindent
		{\textit{Proof of the claim.}}	 			
		First, by applying \autoref{lem_torsion_multi_sym_alg} we get that there exists some $l$ such that $\big((\mathbf{g}_1) \cdots (\mathbf{g}_m)\big)^l\T = 0$.
		Here, we are considering $(\mathbf{g}_1) \cdots (\mathbf{g}_m) \subset S \subset \Sym_S\big((\mathbf{g}_1)\oplus \cdots \oplus (\mathbf{g}_m)\big)$, thus $\big((\mathbf{g}_1) \cdots (\mathbf{g}_m)\big)^l$ lifts to $\kk[\mathbf{x},\mathbf{y}]$ exactly as $\big((\mathbf{g}_1) \cdots (\mathbf{g}_m)\big)^l$. We map into $\Sym_R(E)$ using the canonical map
		\begin{align*}
		\kk[\mathbf{x},\mathbf{y}] \rightarrow \Sym_R(E)\\
		\mathbf{x}_i \mapsto \mathbf{x}_i, \qquad y_i \mapsto e_i,
		\end{align*}
		where $e_i$ are the homogeneous generators of $E$ given by its presentation.
		Summarizing, we have that 
		$$
		\lambda\Big(\big((\mathbf{g}_1) \cdots (\mathbf{g}_m)\big)^l\T \Big) = \big((\mathbf{g}_1(\mathbf{e})) \cdots (\mathbf{g}_m(\mathbf{e}))\big)^l \lambda(\T).
		$$
		
		We have the canonical surjections 
		$$
		\Sym_R(E) \xrightarrow{\phi_1} \Rees_R(E) \xrightarrow{\phi_2} \Rees_R(I) \subset R[t].
		$$
		Also, we can make the identification
		$$
		\phi_2\Big(\phi_1 \Big( \big((\mathbf{g}_1(\mathbf{e})) \cdots (\mathbf{g}_m(\mathbf{e}))\big)^l \Big) \Big) = \big((\mathbf{g}_1(\mathbf{f}t)) \cdots (\mathbf{g}_m(\mathbf{f}t))\big)^l \in \Rees_R(I),
		$$
		and from the birationality assumption we have that $\big((\mathbf{g}_1(\mathbf{f}t)) \cdots (\mathbf{g}_m(\mathbf{f}t))\big)^l \neq 0$.
		Hence, it follows that 
		$$
		\phi_1 \Big( \big((\mathbf{g}_1(\mathbf{e})) \cdots (\mathbf{g}_m(\mathbf{e}))\big)^l \Big) \phi_1\big(\lambda(\T)\big) = 0 \in \Rees_R(E)
		$$
		with $\phi_1 \Big( \big((\mathbf{g}_1(\mathbf{e})) \cdots (\mathbf{g}_m(\mathbf{e}))\big)^l \Big) \neq 0.$ Since $\Rees_R(E)$ is an integral domain, we get our claim $\phi_1\big(\lambda(\T)\big)=0$. 
	\end{proof}

\section{Rational maps in the projective plane with saturated base ideal}\label{section4}

In this section we focus on dominant rational maps $\FF:\PP^2 \dashrightarrow \PP^2$ whose base ideal $I$ is saturated and has dimension $1$. To emphasize our interest in these cases, we recall, for instance, that the base ideal of plane birational maps of degree $d\leq 4$ must be saturated (see \cite[Corollary 1.23]{Hassanzadeh_Simis_Cremona_Sat}).

A straightforward application of the Auslander-Buchsbaum formula yields that the conditions of $I$ being saturated and perfect are equivalent. 
Therefore, we will assume that $I$ has a Hilbert-Burch presentation (see e.g.~\cite[Theorem 20.15]{EISEN_COMM}). We adopt \autoref{nota_single_graded} with $r=2,$ and also the following conditions.
\begin{notation}		
	\label{nota_Hilbert_Burch}
	Assume that $I=(f_0,f_1,f_2) \subset R (=\kk[x_0,x_1,x_2])$ is saturated and $\dim(R/I)=1$.
	The presentation of $I$ is given by 
	\begin{equation}
	\label{Hilbert_Burch_presentation}
	0 \rightarrow R(-d-\mu_1)\oplus R(-d-\mu_2) \xrightarrow{\varphi} {R(-d)}^3 \rightarrow I \rightarrow 0,
	\end{equation}
	where $I$ is generated by the maximal minors of $\varphi$, $\mu_1+\mu_2=d$ and $\mu_1 \le \mu_2$. 
	The matrix of $\varphi,$ which we just denote by $\varphi,$ is 
	$$
	\varphi =  \left(
	\begin{array}{cc}
	a_{0,1} & a_{0,2} \\
	a_{1,1} & a_{1,2} \\
	a_{2,1} & a_{2,2} 
	\end{array}
	\right).
	$$
\end{notation}

The main result of this section is \autoref{thm:birat-mu1=1} where we derive a very simple birationality criterion for rational maps $\FF$ whose base ideal satisfy \autoref{Hilbert_Burch_presentation} with the additional assumption that $\mu_1=1$. This result is based on a complete description of the equations of the Rees algebra of $I$ in this setting, which is given in \autoref{Rees_Eq_mu_1}.

Before going further, we first notice that the degree of $\FF$ under our assumptions is connected to the couple of integers $(\mu_1,\mu_2)$ defined in  \autoref{Hilbert_Burch_presentation}.

\begin{proposition}\label{prop:degFrankA}
	Let $\FF:\PP^2 \dashrightarrow \PP^2$ be a dominant rational map with a dimension $1$ base ideal $I$ that is saturated. Then,
	$$
	\deg(\FF) \;\le\; \mu_1\mu_2
	$$
	with equality if and only if $I$ is locally a complete intersection at its minimal primes.
\end{proposition}
\begin{proof}
	The degree formula of \autoref{thm_degree_formula_base_points} gives us $\deg(\FF) = d^2 - e(\BB)$. 
	We also know that $\deg(\BB)\le e(\BB)$ and $\deg(\BB)=e(\BB)$ if and only if $I$ is locally a complete intersection at its minimal primes. Now, using the resolution \autoref{Hilbert_Burch_presentation} and a simple computation with Hilbert polynomials, we get 
	$$
	\deg(\BB) = {\text{P}}_{R/I}(t)= \binom{t+2}{2}-3\binom{t-d+2}{2}+\binom{t-d-\mu_1+2}{2}	+\binom{t-d-\mu_2+2}{2}	= d^2-\mu_1\mu_2.
	$$
	Therefore, we deduce that $\deg(\FF)\le \mu_1\mu_2$ and  $\deg(\FF)= \mu_1\mu_2$ if and only if $I$ is locally a complete intersection at its minimal primes.
\end{proof}

\subsection{Properties of saturated base ideals}

Below, we gather three technical results on some properties of the base ideal $I$ under our assumptions. We will need them in the sequel.

\begin{lemma}
	\label{lem_prop_I_sat}
	Assume that $\dim(R/I)=1$ and $I$ is saturated. Then, the following statements hold:
	\begin{enumerate}[(i)]

		\item $\HH_I^j(R) \neq 0$ if and only if $j=2$.
		\item $\Ass_R\left(\HH_{I}^2(R)\right)$ is a finite set and equal to 
		$$
		\Ass_R\left(\HH_{I}^2(R)\right) = \Ass_R\left( \Ext_R^2(R/I,R)\right) = \Ass_R\left( R/I \right).
		$$			
	\end{enumerate}
	\begin{proof}		
		$(i)$ From the Grothendieck vanishing theorem \cite[Theorem 6.1.2]{Brodmann_Sharp_local_cohom} we get that $\HH_I^j(R)=0$ for $j > 3$.
		The connection of $\grade(I)$ with local cohomology \cite[Theorem  6.2.7]{Brodmann_Sharp_local_cohom} implies that $\HH_I^j(R)=0$ for $j < 2$ and $\HH_I^2(R)\neq 0$.
		Finally, a graded version of the Lichtenbaum-Hartshorne theorem \cite[Theorem 14.1.16]{Brodmann_Sharp_local_cohom} yields $\HH_I^3(R)=0$.
		
		$(ii)$	 From \cite[Proposition 1.1$(b)$]{MARLEY_ASS_PRIMES_LOCAL_COHOM} we have that $\Ass_R\left(\HH_{I}^2(R)\right) = \Ass_R\left( \Ext_R^2(R/I,R)\right)$.
		The module $\Ext_R^2(R/I,R)$ is the so-called canonical module $\omega_{R/I}$, and its associated primes are given by the unmixed part $I^{\text{un}}$ of $I$ (see \cite[page 250, Lemma 1.9$(c)$]{SCHENZEL_NOTES}), that is
		$$
		\Ass_R\left( \Ext_R^2(R/I,R) \right) = \big\{ \pp \in \Ass_R(R/I) \mid \dim(R/\pp)=1 \big\}.
		$$		
		Since $\dim(R)=3$, we finally get that $I^{\text{un}}$ coincides with $I^{\text{sat}}=I$.
	\end{proof}
\end{lemma}

Using the present hypotheses we would like to better exploit the exact sequence
$$
0 \rightarrow \EEQ \rightarrow \Sym(I) \rightarrow \Rees(I) \rightarrow 0.
$$
We recall that the symmetric algebra can be easily described with the presentation $\autoref{Hilbert_Burch_presentation}$ of $I$, and its defining equations are given by
\begin{equation}
\label{eqs_Sym}
\left(g_1, g_2\right) = \left(y_0,y_1,y_2 \right) \cdot \varphi.
\end{equation}
Hence, we have an isomorphism 
$$
\Sym(I) \cong \AAA/\left(g_1, g_2\right).
$$
We also have that $\{g_1,g_2\}$ is a regular sequence in $\AAA$ (see \cite[Corollary 2.2]{SIMIS_VASC_SYM_INT}) and so the corresponding Koszul complex 
	\begin{equation}
	\label{koszul_complex_Sym}
	\mathbb{L}_{\bullet}: \quad 0 \rightarrow \AAA(-d,-2) \xrightarrow{\tiny\left(\begin{array}{cc}
		-g_2\\
		g_1
		\end{array}\right)} \AAA(-\mu_1,-1) \;\oplus\; \AAA(-\mu_2,-1) \xrightarrow{\left(g_1,g_2\right)} \AAA 
	\end{equation}
is exact.

\begin{lemma}
	\label{description_torsion_Sym}
	Assume that $\dim(R/I)=1,$ and $I$ is saturated. Then, the torsion submodule $\EEQ$ is described by the exact sequence
		$$
		0 \rightarrow \EEQ \rightarrow \HH_I^2\big(\AAA\big)(-d,-2) \xrightarrow{\tiny\left(\begin{array}{c}
			-g_2\\
			g_1
			\end{array}\right)} \HH_I^2\big(\AAA\big)(-\mu_1,-1) \;\oplus\; \HH_I^2\big(\AAA\big)(-\mu_2,-1).
		$$
	\begin{proof}
		We consider the double complex $\mathbb{L}_{\bullet} \otimes_{R} C_I^{\bullet}$. 
		Computing with the second filtration we obtain the spectral sequence
		$$
		{}^{\text{II}}E_2^{p,-q} = \begin{cases}
		\HH_{I}^p(\Sym(I)) \qquad \text{if } q = 0\\
		0 \qquad \qquad \qquad \quad \text{otherwise.}
		\end{cases}
		$$
		On the other hand, by using the first filtration we get that ${}^{\text{I}}E_1^{-p,q}=\HH_{I}^q(\mathbb{L}_p)$.
		Hence, from \autoref{lem_prop_I_sat}$(i)$, the only row that does not vanish in  ${}^{\text{I}}E_1^{\bullet,\bullet}$ is given by the complex
		$$
		\HH_{I}^2(\mathbb{L}_{\bullet}): \quad0 \rightarrow \HH_{I}^2\big(\AAA\big)(-d,-2) \rightarrow \HH_{I}^2\big(\AAA\big)(-\mu_1,-1) \;\oplus\; \HH_{I}^2\big(\AAA\big)(-\mu_2,-1) \rightarrow \HH_{I}^2\big(\AAA\big) \rightarrow 0.
		$$
		Thus we obtain 
		$$
		{}^{\text{I}}E_2^{-p,q} = \begin{cases}
		\HH_p\big(\HH_{I}^2(\mathbb{L}_{\bullet})\big) \qquad \text{if } q = 2\\
		0 \qquad \qquad \qquad \quad \text{otherwise.}
		\end{cases}
		$$
		Since both spectral sequences collapse,  from \autoref{torsion_multi_sym_alg} we get 
		$$
		\EEQ = \HH_{I}^0\big(\Sym(I)\big) \cong \HH_2\big(\HH_{I}^2(\mathbb{L}_{\bullet})\big),
		$$
		and so the assertion follows.
	\end{proof}
\end{lemma}

\begin{notation}
For  $z=x_i$ or $z=y_j$ and $F \in \AAA,$ we denote with $\deg_z(F)$ the maximal degree of the monomials of $F$ in terms of $z$.
\end{notation}

Using the presentation matrix $\varphi$ of $I$, we characterize when $I$ is of linear type.

\begin{lemma}
	\label{linear_type_characterization_sat_case}
	Assume that $\dim(R/I)=1$ and $I$ is saturated.
	Then, $I$ is of linear type if and only if $I_1(\varphi)$ is an $\mm$-primary ideal.
	\begin{proof}				
		Using \autoref{nota_Hilbert_Burch},  we have that $g_1 = a_{0,1}y_0+a_{1,1}y_1+a_{2,1}y_2$ and $g_2 = a_{0,2}y_0+a_{1,2}y_1+a_{2,2}y_2$.
		
		$(\Rightarrow)$
		Let us assume that $I_1(\varphi)$ is not $\mm$-primary.
		Then, we have that $I_1(\varphi) \supset I_2(\varphi)=I$ and $\HT(I_1(\varphi))=\HT(I)=2$.
		So the minimal primes of $I_1(\varphi)$ are contained in the set of associated primes of $I$.
		In particular, there exists some $\pp \in \Ass_R(R/I)$ with $I_1(\varphi) \subset \pp$.
		From \autoref{lem_prop_I_sat}$(ii)$ we have that $\pp \in \Ass_R\left(\HH_{I}^2(R) \right)$, and this implies the existence of an element $0\neq v \in \HH_{I}^2(R)$ that is annihilated by $I_1(\varphi)$.
		Considering $v$ as an element in $\HH_{I}^2(\AAA)$, we get $g_1\cdot v=g_2\cdot v=0$.
		Therefore, from \autoref{description_torsion_Sym} we obtain $\EEQ\neq0$.
		
		$(\Leftarrow)$
		Here we suppose that $I_1(\varphi)$ is $\mm$-primary.
		By contradiction, we assume  $\EEQ\neq0,$ and choose $0\neq w \in \EEQ$.
		Since  $\HH_{I}^2(\AAA) \cong \HH_{I}^2(R) \otimes_{\kk} S$,   $w$ can be written as $w = \sum_{i=1}^l v_i \otimes_{\kk} m_i$ where $v_i \in \HH_{I}^2(R)$ and $m_i$ is a monomial in $S$.
		For each $0 \le j \le 2$, we  have a unique decomposition
		$$
		w = w_j + w_j^*,
		$$
		where $w_j\neq 0$ is obtained by adding all the terms $v_i\otimes_{\kk} m_i$ such that the value of $\deg_{y_j}(m_i)$ is maximal.
		From the condition $g_1\cdot w=g_2\cdot w=0$,  we automatically get that $a_{j,1}\cdot w_j=a_{j,2}\cdot w_j = 0$.
		Therefore, we have obtained that $I_1(\varphi)$ is composed of zero divisors in $\HH_{I}^2(\AAA)$.
		By the isomorphism $\HH_{I}^2(\AAA) \cong \HH_{I}^2(R) \otimes_{\kk} S$ and  \autoref{lem_prop_I_sat}$(ii)$,  we have that $\Ass_R\left(\HH_{I}^2(\AAA)\right)=\Ass_R\left(\HH_{I}^2(R)\right)=\Ass_R(R/I)$.
		Finally, the prime avoidance lemma implies that $I_1(\varphi) \subset \pp$ for some $\pp \in \Ass_R(R/I)$,
		and this contradicts the fact that $I$ is saturated.
	\end{proof}
\end{lemma}

\begin{remark}
	As in \cite{Hassanzadeh_Simis_Cremona_Sat}, an alternative proof of \autoref{linear_type_characterization_sat_case} can be obtained from either \cite[Section 5]{HSV_TRENTO_SCHOOL} or \cite[Proposition 3.7]{Simis_Vasc_Syz_Conormal_Mod}.
	Indeed,
	one can note that $I$ is  locally a complete intersection at its minimal primes if and only if $I_1(\varphi)$ is an $\mm$-primary ideal.
	Therefore, the result follows from the fact that $I$ is an almost complete intersection.
\end{remark}

\subsection{An effective birationality criterion in the case $\mu_1=1$}

In this subsection, we focus on computing the defining equations of the Rees algebra in the case $\mu_1=1$ (\autoref{nota_Hilbert_Burch}).
As a corollary of this computation, we obtain a simple characterization of birationality in the particular case $\mu_1=1$ (\autoref{thm:birat-mu1=1}) by means of the Jacobian dual criterion (see \autoref{sec:jacdualcriterion}, but also \cite{AB_INITIO}). Our proof is inspired by the method used in \cite{COX_REES_MU_1}. We shall see that it is enough to treat the following special case.

\begin{notation}
	\label{notation_case_mu_1_x0x1}
	Assume that $\dim(R/I)=1$ and $I$ is saturated. 
	Suppose that the presentation matrix in \autoref{Hilbert_Burch_presentation} is of the form
	$$
	\varphi = \left(
	\begin{array}{cc}
	x_0 & p_0 \\
	-x_1 & p_1 \\
	0 & p_2
	\end{array}
	\right).
	$$
	Here we have that $g_1=x_0y_0-x_1y_1$ and $g_2=p_0y_0+p_1y_1+p_2y_2$.
\end{notation}

We now give a version of \autoref{description_torsion_Sym}  that uses the more amenable ideal $(x_0,x_1)$ as the support of the local cohomology modules.

\begin{lemma}
	\label{lem_mu_1_gens_local_cohom}
	Using \autoref{notation_case_mu_1_x0x1}, the following statements hold:
	\begin{enumerate}[(i)]
		\item $\EEQ=\HH_{(x_0,x_1)}^0\left(\Sym(I)\right)$.
		\item The torsion submodule $\EEQ$ is determined by the following exact sequence
		$$
		0 \rightarrow \EEQ \rightarrow \HH_{\xy}^2\big(\AAA\big)(-d,-2) \xrightarrow{\tiny\left(\begin{array}{c}
			-g_2\\
			g_1
			\end{array}\right)} \HH_{\xy}^2\big(\AAA\big)(-1,-1) \;\oplus\; \HH_{\xy}^2\big(\AAA\big)(-d+1,-1).
		$$
		\item Via this identification, we have that $\EEQ$ is generated by 
		\begin{equation*}
		\EEQ \;\cong\;  \AAA \cdot \Big\{ w_n   \;\mid\;  0 \le n \le d-2 \;\text{ and }\; g_2 \cdot w_n = 0 \Big\}				
		\end{equation*}
	\end{enumerate}
	where each $w_n$ is of the form
	$$
	w_n = \sum_{i=0}^{d-2-n}  \frac{1}{x_0^{i+1}x_1^{d-1-n-i}}y_0^{d-2-n-i}y_1^i \;\in\;  {\left[\HH_{\xy}^2(\AAA)(0,-2)\right]}_{n-d}.
	$$
	\begin{proof}
		$(i)$ From \autoref{description_torsion_Sym} we have that any  $F \in \EEQ$ can be written as
		$
		F = \sum_{i=1}^l a_i  \yy^{\gamma_i},
		$
		where each $a_i \in \HH_{I}^2(R)$, and satisfies $g_1\cdot F=g_2\cdot F=0$.
		Since $g_1=x_0y_0-x_1y_1$, we can conclude that there exists some $u>0$ such that $x_0^uF=x_1^uF=0$.
		From the fact that $I \subset (x_0,x_1)$, we get a neater description of $\EEQ$ given by 
		$$
		\EEQ = \HH_{\xy}^0\left(\EEQ\right) = \HH_{\xy}^0\left(\HH_{I}^0\left(\Sym(I)\right)\right) = \HH_{\xy}^0\left(\Sym(I)\right).
		$$
		
		$(ii)$ To obtain the required exact sequence we follow the same arguments as in the proof of \autoref{description_torsion_Sym}.
		We consider the double complex $\mathbb{L}_{\bullet} \otimes_{R} C_{\xy}^{\bullet}$, where $\mathbb{L}_{\bullet}$ is the Koszul complex of \autoref{koszul_complex_Sym}.
		Examining the spectral sequences corresponding to the first and second filtrations of $\mathbb{L}_{\bullet} \otimes_{R} C_{\xy}^{\bullet}$, we obtain
		\begin{equation*}
		\EEQ = \HH_{\xy}^0(\Sym(I)) \cong  \HH_2\big(\HH_{\xy}^2(\mathbb{L}_{\bullet})\big).
		\end{equation*}
		From this isomorphism we get the claimed exact sequence.
		
		$(iii)$ First we note that $\HH_{\xy}^2(\AAA) \cong \frac{1}{x_0x_1}\kk[x_0^{-1},x_1^{-1},x_2,y_0,y_1,y_2]$.
		In this part, we describe a set of generators of the kernel of the multiplication map
		\begin{equation}
		\label{mult_map_g_1}
		\HH_{\xy}^2\big(\AAA\big)(-d,-2) \xrightarrow{g_1} \HH_{\xy}^2\big(\AAA\big)(-d+1,-1).
		\end{equation}
		Using that $g_1=x_0y_0-x_1y_1$ does not depend on the variables $x_2$ and $y_2$, then a set of generators of the kernel of this map is given by just considering elements inside the subring $\frac{1}{x_0x_1}\kk[x_0^{-1},x_1^{-1},y_0,y_1]$.
		Let $F \in \frac{1}{x_0x_1}\kk[x_0^{-1},x_1^{-1},y_0,y_1]$, then we expand it as follows:
		$$
		F = \sum_{i=l}^m F_i y_0^{\beta_i}y_1^i
		$$
		where each $F_i \in \frac{1}{x_0x_1}\kk[x_0^{-1},x_1^{-1}]$.
		The condition $(x_0y_0-x_1y_1) F=0$ gives the relations 
		$$
		x_0F_ly_0^{\beta_l+1}y_1^l = 0, \;\; x_1F_my_0^{\beta_m}y_1^{m+1}=0, \quad \text{ and } \;\; \left(x_0F_iy_0^{\beta_i+1} - x_1F_{i-1}y_0^{\beta_{i-1}}\right)y_1^i =0 \;\;\text{ for }\;\; l+1 \le i \le m.
		$$
		We can easily conclude that a set of generators of the kernel of \autoref{mult_map_g_1} is given by elements of the form
		$$
		\frac{1}{x_0x_1^{m+1}}y_0^m \;+\; 		\frac{1}{x_0^2x_1^{m}}y_0^{m-1}y_1 \;+\; \cdots \;+\;  \frac{1}{x_0^{m+1}x_1}y_1^m 
		$$
		where $m \ge 0$.
		Therefore, to conclude we only need to take into account the shifting of $-d$ in the grading part corresponding with $R$, and intersect with the elements that are also annihilated by the other equation $g_2$.
	\end{proof}	
\end{lemma}

Now, we describe the process of computing the so-called Sylvester forms that have been successfully used in several papers like \cite{HONG_SIMIS_VASC_ELIM, COX_REES_MU_1, Hassanzadeh_Simis_Cremona_Sat}.

\begin{algorithm}
	\label{const_Sylvester_forms}
	Using \autoref{notation_case_mu_1_x0x1}, we compute iteratively the set of forms  
	$
	\Sylv,
	$
	as follows.
	\begin{enumerate}[(I)]
		\item Set $i = 0$, $F_0 = g_2$ and $\Sylv=\emptyset$. 
		\item While $F_i \in (x_0,x_1)$ we perform the following steps:
		\begin{enumerate}[(a)]
			\item Write $F_i$ in the convenient form 
			$
			F_i = {(F_i)}_{x_0}x_0 + {(F_i)}_{x_1}x_1
			$
			to get the equation 
			\begin{equation*}
			\left(\begin{array}{c}
			g_1\\
			F_i
			\end{array}\right)
			= \left(\begin{array}{cc}
			y_0 & -y_1 \\
			{(F_i)}_{x_0} & {(F_i)}_{x_1}
			\end{array}
			\right)
			\left(\begin{array}{c}
			x_0\\
			x_1
			\end{array}\right).					
			\end{equation*}
			Then, the $(i+1)$-th Sylvester form is computed with the determinant
			$$
			F_{i+1} = \det\left(\begin{array}{cc}
			y_0 & -y_1 \\
			{(F_i)}_{x_0} & {(F_i)}_{x_1}
			\end{array}
			\right).
			$$
			\item Set $\Sylv=\Sylv \cup \{F_{i+1}\}$.
			\item Set $i = i+1$.
		\end{enumerate}			
		\item Set $m=i$ and return the set of computed forms
		$
		\Sylv = \{F_1, \ldots, F_m\}.
		$
	\end{enumerate}
	We emphasize for later use that  $\bideg(F_i)=\left(d-1-i,i+1\right)$ for each $0\le i \le m$.
\end{algorithm}

The next lemma relates the torsion of the symmetric algebra with the Sylvester forms.

\begin{lemma}
	\label{lem_torsion_Sym_Sylvester_forms}
	In \autoref{const_Sylvester_forms}, for each $1\le i\le m$ the following statements hold:
	\begin{enumerate}[(i)]
		\item $\{g_1, F_i\}$ is a regular sequence.
		\item There is an isomorphism 
		\begin{equation*}
		{\Big[\big( 0 \;{:}_{\Sym(I)}\; {(x_0,x_1)}^i\big)\Big]}_{d-1-i} \cong {\left[ \frac{(g_1,F_i)}{(g	_1)}\right]}_{d-1-i}.
		\end{equation*}
	\end{enumerate}
	\begin{proof}
		The proof is obtained by induction on $i$.
		
		Let $i=1$. Since $\{x_0,x_1\}$ and $\{g_1,g_2\}$ are regular sequences,  from Wiebe's lemma (see e.g.~\cite[Proposition 3.8.1.6]{JOUANOLOU_INVARIANT_ELIM}) 
		we get the following exact sequence
		\begin{equation}
		\label{first_exact_seq_Wiebe}
		0 \rightarrow \AAA/(x_0,x_1) \xrightarrow{F_1} \AAA/(g_1,g_2) \xrightarrow{\tiny\left(\begin{array}{c}
			x_0\\
			x_1
			\end{array}\right)} {\left[\AAA/(g_1,g_2)\right]}^2,
		\end{equation}
		where $F_1$ is the first Sylvester form.
		Thus we have $F_1 \not\in (g_1,g_2)$, and since $g_1$ is an irreducible polynomial,  we get that $\{g_1,F_1\}$ is a regular sequence.
		From the fact that $\bideg(g_2)=(d-1,1)$, for any $v \in \AAA$ with $\deg_{\xx}(v) \le d-2$ the exact sequence \autoref{first_exact_seq_Wiebe} gives the following equivalences
		$$
		v \in \big((g_1,g_2) : (x_0,x_1)\big) \Longleftrightarrow v \in \left(g_1,g_2,F_1\right) \Longleftrightarrow v \in \left(g_1,F_1\right).
		$$
		In other words, we obtain the isomorphisms 
		$$
		{\Big[\big( 0 \;{:}_{\Sym(I)}\; (x_0,x_1)\big)\Big]}_{d-2} \cong {\left[ \frac{(g_1,g_2,F_1)}{(g_1,g_2)}\right]}_{d-2} \cong {\left[ \frac{(g_1,F_1)}{(g_1)}\right]}_{d-2}.
		$$
		Therefore, both conditions hold for $i=1$. 
		
		Let $2\le i \le m$. By induction we assume that conditions $(i)$ and $(ii)$ are satisfied for $i-1$.
		Again, from Weibe's lemma we get the exact sequence 
		\begin{equation}
		\label{second_exact_seq_Wiebe}
		0 \rightarrow \AAA/(x_0,x_1) \xrightarrow{F_{i}} \AAA/(g_1,F_{i-1}) \xrightarrow{\tiny\left(\begin{array}{c}
			x_0\\
			x_1
			\end{array}\right)} {\left[\AAA/(g_1,F_{i-1})\right]}^2,
		\end{equation}
		where $F_i$ is the $i$-th Sylvester form. 
		By the same previous argument, it is clear that $(g_1,F_i)$ is a regular sequence.
		Using the exactness of \autoref{second_exact_seq_Wiebe} and similar degree considerations, we have that 
		$$
		v \in \big((g_1,F_{i-1}) : (x_0,x_1)\big) \Longleftrightarrow v \in \left(g_1,F_{i-1},F_{i}\right) \Longleftrightarrow v \in \left(g_1,F_{i}\right)
		$$
		for any $v \in \AAA$ with $\deg_{\xx}(v)\le d-1-i$.
		Thus, we also have the isomorphisms
		$$
		{\left[\frac{\big((g_1,F_{i-1}) : (x_0,x_1)\big)}{(g_1,F_{i-1})}\right]}_{d-1-i} \cong {\left[ \frac{(g_1,F_{i-1},F_{i})}{(g_1,F_{i-1})}\right]}_{d-1-i} \cong {\left[ \frac{(g_1,F_{i})}{(g_1)}\right]}_{d-1-i}.
		$$
		Since $\deg_{\xx}(F_{i-1})=d-1-(i-1)$ and ${\left[\Sym(I)\right]}_{\le d-2} \cong {\left[\AAA/(g_1)\right]}_{\le d-2}$,  we get
		$$
		{\left[\frac{\big((g_1,F_{i-1}) : (x_0,x_1)\big)}{(g_1,F_{i-1})}\right]}_{d-1-i} \cong 		{\left[\left(\frac{(g_1,F_{i-1})}{(g_1)} \;{:}_{\Sym(I)}\; (x_0,x_1) \right)\right]}_{d-1-i}. 
		$$
		From the inductive hypothesis we already have
		\begin{equation*}
		{\Big[\big( 0 \;{:}_{\Sym(I)}\; {(x_0,x_1)}^{i-1}\big)\Big]}_{d-1-(i-1)} \cong {\left[ \frac{(g_1,F_{i-1})}{(g	_1)}\right]}_{d-1-(i-1)}.
		\end{equation*}
		By assembling these isomorphisms we conclude that the condition $(ii)$ 
		\begin{equation*}
		{\Big[\big( 0 \;{:}_{\Sym(I)}\; {(x_0,x_1)}^{i}\big)\Big]}_{d-1-i} \cong {\left[ \frac{(g_1,F_{i})}{(g_1)}\right]}_{d-1-i}
		\end{equation*}
		also holds for the form $F_{i}$.
		Therefore, we have that both conditions are satisfied for all the Sylvester forms.	
	\end{proof}
\end{lemma}

The following theorem gives explicit generators for the presentation of $\Rees(I)$.
It can be seen as a natural generalization of both \cite[Theorem 2.3]{COX_REES_MU_1} and \cite[Theorem 2.7$(i)$]{Hassanzadeh_Simis_Cremona_Sat}.

\begin{theorem}
	\label{Rees_Eq_mu_1}
	Let $\Sylv$ be the set of Sylvester forms computed in \autoref{const_Sylvester_forms}.
	Then, the defining equations of $\Rees(I)$ are minimally generated by 
	$$
	\{g_1,g_2\} \;\cup\; \Sylv.
	$$
	In particular, it is minimally generated in the bi-degrees
	$$
	(1,1), (d-1,1),(d-2,2),\ldots,(d-1-m,m+1).
	$$
	\begin{proof}
		Let $e=d-1-m$.
		In \autoref{lem_torsion_Sym_Sylvester_forms}$(ii)$ we proved that 
		$$
		{\Big[\big( 0 \;{:}_{\Sym(I)}\; {(x_0,x_1)}^{m}\big)\Big]}_{e} \cong {\left[ \frac{(g_1,F_{m})}{(g_1)}\right]}_{e},
		$$
		which implies that for any $j>0$ we have
		\begin{equation}
		\label{vanishing_torsion_Sym}
		{\Big[\big( 0 \;{:}_{\Sym(I)}\; {(x_0,x_1)}^{m+j}\big)\Big]}_{e-j} \cong 	{\left[\left(\frac{(g_1,F_{m})}{(g_1)} \;{:}_{\Sym(I)}\; {(x_0,x_1)}^{j} \right)\right]}_{e-j}.
		\end{equation}
		Since $F_{m} \not\in (x_0,x_1)$ and $g_1 \in (x_0,x_1)$, we deduce that the term on the right is always equal to zero. 
		
		From \autoref{lem_mu_1_gens_local_cohom}$(iii)$, a set of generators for $\EEQ$ is given by elements of the form
		$$
		w_n = \sum_{i=0}^{d-2-n}  \frac{1}{x_0^{i+1}x_1^{d-1-n-i}}y_0^{d-2-n-i}y_1^i \;\in\;  {\left[\HH_{\xy}^2(\AAA)(0,-2)\right]}_{n-d}.
		$$
		Hence, for any $j > 0$ we have that ${(x_0,x_1)}^{m+j} \cdot w_{e-j}=0$.
		The vanishing of the equation \autoref{vanishing_torsion_Sym} implies that $w_{e-j} \not\in \EEQ$ for all $j > 0$.
		Therefore, the elements $w_{d-2},w_{d-1},\ldots,w_{e}$ generate $\EEQ$.
		Using the isomorphisms of \autoref{lem_mu_1_gens_local_cohom}$(iii)$  and \autoref{lem_torsion_Sym_Sylvester_forms}$(ii)$,  we identify $w_{d-1-i}$ as a multiple of $F_i$, and this implies that $F_1,F_2,\ldots,F_m$ is also a set of generators of $\EEQ$.
		Finally, simple degree considerations yield that $\{g_1,g_2,F_1,F_2,\ldots,F_m\}$ is a minimal set of generators.
	\end{proof}
\end{theorem}

We are now ready for providing our birationality criterion. We notice that from \autoref{prop:degFrankA}, we have that the rational map $\FF$ is birational for $d\le2$ under our assumptions. Therefore, we only need to consider the cases $d\ge 3$. Before stating the main result we make the following point.

\begin{remark}
	\label{rem_linear_change_coordinates}
	In the presentation matrix $\varphi$ of \autoref{Hilbert_Burch_presentation}, if $\mu_1=1$ and $\HT(I_1(\varphi))=2$ then the vector space spanned by the linear forms of the first column has dimension $2$.  
	Therefore, in this case we can make a linear change of coordinates and assume that $\varphi$ is given as in \autoref{notation_case_mu_1_x0x1}.
\end{remark}

The following result covers a family of birational maps that include the classical de Jonqui\`eres maps (see e.g. \cite[\S 2.1]{Hassanzadeh_Simis_Cremona_Sat}).

\begin{theorem}\label{thm:birat-mu1=1}
	Let $\FF:\PP^2 \dashrightarrow \PP^2$ be a dominant rational map with a dimension $1$ base ideal $I$ that is saturated.
	Suppose that $\varphi$ in \autoref{Hilbert_Burch_presentation} satisfies $\mu_1=1$ and $d \ge 3$. 
	Then, $\FF$ is birational if and only if the following conditions are satisfied:
	\begin{enumerate}[(i)]
		\item $\HT(I_1(\varphi))=2$.
		\item  After the linear change of coordinates of \autoref{rem_linear_change_coordinates}, in \autoref{const_Sylvester_forms} we have $m=d-2$.             
	\end{enumerate}
	\begin{proof}
		After a linear change of coordinates the condition of birationality remains invariant. 
		From the Jacobian dual criterion (\cite[Theorem 2.18]{AB_INITIO} or \autoref{jacob_main_crit}) we have that $\FF$ is birational if and only if there is another equation of bi-degree $(1,*)$, and by \autoref{Rees_Eq_mu_1} this is equivalent to $m=d-2$.
	\end{proof}
\end{theorem}


\begin{bibdiv}
\begin{biblist}

\bib{ACHILLES_MANARESI_J_MULT}{article}{
      author={Achilles, R\"udiger},
      author={Manaresi, Mirella},
       title={Multiplicity for ideals of maximal analytic spread and
  intersection theory},
        date={1993},
     journal={J. Math. Kyoto Univ.},
      volume={33},
      number={4},
       pages={1029\ndash 1046},
}

\bib{ACBook}{book}{
      author={Alberich-Carrami\~nana, Maria},
       title={Geometry of the plane cremona maps},
      series={Lecture Notes in Mathematics},
   publisher={Springer-Verlag Berlin Heidelberg},
        date={2002},
      volume={1769},
}

\bib{ATIYAH_MACDONALD}{book}{
      author={Atiyah, M.~F.},
      author={Macdonald, I.~G.},
       title={Introduction to commutative algebra},
   publisher={Addison-Wesley Publishing Co., Reading, Mass.-London-Don Mills,
  Ont.},
        date={1969},
}

\bib{BOTBOL_IMPLICIT_SURFACE}{article}{
      author={Botbol, Nicol\'as},
       title={The implicit equation of a multigraded hypersurface},
        date={2011},
     journal={J. Algebra},
      volume={348},
       pages={381\ndash 401},
}

\bib{EFFECTIVE_BIGRAD}{article}{
      author={Botbol, Nicol\'as},
      author={Bus\'e, Laurent},
      author={Chardin, Marc},
      author={Hassanzadeh, Seyed~Hamid},
      author={Simis, Aron},
      author={Tran, Quang~Hoa},
       title={Effective criteria for bigraded birational maps},
        date={2017},
     journal={J. Symbolic Comput.},
      volume={81},
       pages={69\ndash 87},
}

\bib{Brodmann_Sharp_local_cohom}{book}{
      author={Brodmann, M.~P.},
      author={Sharp, R.~Y.},
       title={Local cohomology.},
     edition={Second},
      series={Cambridge Studies in Advanced Mathematics},
   publisher={Cambridge University Press, Cambridge},
        date={2013},
      volume={136},
        note={An algebraic introduction with geometric applications},
}

\bib{BRUNS_HERZOG}{book}{
      author={Bruns, Winfried},
      author={Herzog, J\"urgen},
       title={Cohen-{M}acaulay rings},
     edition={2},
      series={Cambridge Studies in Advanced Mathematics},
   publisher={Cambridge University Press},
        date={1998},
}

\bib{LAURENT_CHARDIN_IMPLICIT}{article}{
      author={Bus\'e, Laurent},
      author={Chardin, Marc},
       title={Implicitizing rational hypersurfaces using approximation
  complexes},
        date={2005},
     journal={J. Symbolic Comput.},
      volume={40},
      number={4-5},
       pages={1150\ndash 1168},
}

\bib{Laurent_Jouanolou_Closed_Image}{article}{
      author={Bus\'e, Laurent},
      author={Jouanolou, Jean-Pierre},
       title={On the closed image of a rational map and the implicitization
  problem},
        date={2003},
     journal={J. Algebra},
      volume={265},
      number={1},
       pages={312\ndash 357},
}

\bib{Yairon}{misc}{
      author={Cid-Ruiz, Yairon},
       title={{\it MultiGradedRationalMap,} a package designed for macaulay2},
        note={Available at
  \url{http://www2.macaulay2.com/Macaulay2/doc/Macaulay2-1.15/share/doc/Macaulay2/MultiGradedRationalMap/html/}},
}

\bib{DMOD}{article}{
      author={Cid-Ruiz, Yairon},
       title={A ${D}$-module approach on the equations of the {R}ees algebra},
        date={2017},
     journal={to appear in J. Commut. Algebra},
        note={arXiv:1706.06215},
}

\bib{SAT_FIB_PERF_HT_2}{article}{
      author={Cid-Ruiz, Yairon},
       title={Multiplicity of the saturated special fiber ring of height two
  perfect ideals},
        date={2020},
     journal={Proc. Amer. Math. Soc.},
      volume={148},
      number={1},
       pages={59\ndash 70},
}

\bib{MULT_GOR_HT_3}{article}{
      author={Cid-Ruiz, Yairon},
      author={Mukundan, Vivek},
       title={Multiplicity of the saturated special fiber ring of height three
  gorenstein ideals},
        date={2019},
     journal={arXiv preprint arXiv:1909.13633},
}

\bib{COX_REES_MU_1}{article}{
      author={Cox, David},
      author={Hoffman, J.~William},
      author={Wang, Haohao},
       title={Syzygies and the {R}ees algebra},
        date={2008},
     journal={J. Pure Appl. Algebra},
      volume={212},
      number={7},
       pages={1787\ndash 1796},
}

\bib{Cutkosky_Ha_Asymp}{article}{
      author={Cutkosky, Steven~Dale},
      author={H\`a, Huy~T\`ai},
      author={Srinivasan, Hema},
      author={Theodorescu, Emanoil},
       title={Asymptotic behavior of the length of local cohomology},
        date={2005},
     journal={Canad. J. Math.},
      volume={57},
      number={6},
       pages={1178\ndash 1192},
}

\bib{dolgachev-notes}{unpublished}{
      author={Dolgachev, Igor},
       title={Lectures on cremona transformations},
        date={2016},
        note={\url{http://www.math.lsa.umich.edu/~idolga/cremonalect.pdf}},
}

\bib{AB_INITIO}{article}{
      author={Doria, A.~V.},
      author={Hassanzadeh, S.~H.},
      author={Simis, A.},
       title={A characteristic-free criterion of birationality},
        date={2012},
     journal={Adv. Math.},
      volume={230},
      number={1},
       pages={390\ndash 413},
}

\bib{EISEN_COMM}{book}{
      author={Eisenbud, David},
       title={Commutative algebra with a view towards algebraic geometry},
      series={Graduate Texts in Mathematics, 150},
   publisher={Springer-Verlag},
        date={1995},
}

\bib{EISENBUD_ULRICH_ROW_IDEALS}{article}{
      author={Eisenbud, David},
      author={Ulrich, Bernd},
       title={Row ideals and fibers of morphisms},
        date={2008},
     journal={Michigan Math. J.},
      volume={57},
       pages={261\ndash 268},
        note={Special volume in honor of Melvin Hochster},
}

\bib{FLENNER_O_CARROLL_VOGEL}{book}{
      author={Flenner, H.},
      author={O'Carroll, L.},
      author={Vogel, W.},
       title={Joins and intersections},
      series={Springer Monographs in Mathematics},
   publisher={Springer-Verlag, Berlin},
        date={1999},
}

\bib{FULTON_INTERSECTION_THEORY}{book}{
      author={Fulton, William},
       title={Intersection theory},
     edition={Second},
      series={Ergebnisse der Mathematik und ihrer Grenzgebiete. 3. Folge. A
  Series of Modern Surveys in Mathematics [Results in Mathematics and Related
  Areas. 3rd Series. A Series of Modern Surveys in Mathematics]},
   publisher={Springer-Verlag, Berlin},
        date={1998},
      volume={2},
}

\bib{GORTZ_WEDHORN}{book}{
      author={G\"{o}rtz, Ulrich},
      author={Wedhorn, Torsten},
       title={Algebraic geometry {I}},
      series={Advanced Lectures in Mathematics},
   publisher={Vieweg + Teubner, Wiesbaden},
        date={2010},
        ISBN={978-3-8348-0676-5},
         url={https://doi.org/10.1007/978-3-8348-9722-0},
        note={Schemes with examples and exercises},
}

\bib{GRADED_RINGS_I}{article}{
      author={Goto, Shiro},
      author={Watanabe, Keiichi},
       title={On graded rings. {I}},
        date={1978},
     journal={J. Math. Soc. Japan},
      volume={30},
      number={2},
       pages={179\ndash 213},
}

\bib{GRADED_RINGS_II}{article}{
      author={Goto, Shiro},
      author={Watanabe, Keiichi},
       title={On graded rings. {II}. ({${\bf Z}^{n}$}-graded rings)},
        date={1978},
     journal={Tokyo J. Math.},
      volume={1},
      number={2},
       pages={237\ndash 261},
}

\bib{MACAULAY2}{misc}{
      author={Grayson, Daniel~R.},
      author={Stillman, Michael~E.},
       title={Macaulay2, a software system for research in algebraic geometry},
        note={Available at \url{http://www.math.uiuc.edu/Macaulay2/}},
}

\bib{HACON}{article}{
      author={Hacon, Christopher~D.},
       title={Effective criteria for birational morphisms},
        date={2003},
     journal={J. London Math. Soc. (2)},
      volume={67},
      number={2},
       pages={337\ndash 348},
}

\bib{HARTSHORNE}{book}{
      author={Hartshorne, Robin},
       title={Algebraic geometry},
   publisher={Springer-Verlag, New York-Heidelberg},
        date={1977},
        note={Graduate Texts in Mathematics, No. 52},
}

\bib{Hassanzadeh_Simis_Cremona_Sat}{article}{
      author={Hassanzadeh, Seyed~Hamid},
      author={Simis, Aron},
       title={Plane {C}remona maps: saturation and regularity of the base
  ideal},
        date={2012},
     journal={J. Algebra},
      volume={371},
       pages={620\ndash 652},
}

\bib{HASSANZADEH_SIMIS_DEGREES}{article}{
      author={Hassanzadeh, Seyed~Hamid},
      author={Simis, Aron},
       title={Bounds on degrees of birational maps with arithmetically
  {C}ohen-{M}acaulay graphs},
        date={2017},
     journal={J. Algebra},
      volume={478},
       pages={220\ndash 236},
}

\bib{HSV_Approx_Complexes_I}{article}{
      author={Herzog, J.},
      author={Simis, A.},
      author={Vasconcelos, W.~V.},
       title={Approximation complexes of blowing-up rings},
        date={1982},
     journal={J. Algebra},
      volume={74},
      number={2},
       pages={466\ndash 493},
}

\bib{HSV_Approx_Complexes_II}{article}{
      author={Herzog, J.},
      author={Simis, A.},
      author={Vasconcelos, W.~V.},
       title={Approximation complexes of blowing-up rings. {II}},
        date={1983},
     journal={J. Algebra},
      volume={82},
      number={1},
       pages={53\ndash 83},
}

\bib{HSV_TRENTO_SCHOOL}{incollection}{
      author={Herzog, J.},
      author={Simis, A.},
      author={Vasconcelos, W.~V.},
       title={Koszul homology and blowing-up rings},
        date={1983},
   booktitle={Commutative algebra ({T}rento, 1981)},
      series={Lecture Notes in Pure and Appl. Math.},
      volume={84},
   publisher={Dekker, New York},
       pages={79\ndash 169},
}

\bib{Herzog_hilb_polynom}{article}{
      author={Herzog, J\"urgen},
      author={Puthenpurakal, Tony~J.},
      author={Verma, Jugal~K.},
       title={Hilbert polynomials and powers of ideals},
        date={2008},
     journal={Math. Proc. Cambridge Philos. Soc.},
      volume={145},
      number={3},
       pages={623\ndash 642},
}

\bib{HWH}{article}{
      author={Hoffman, J.~William},
      author={Wang, Hao~Hao},
       title={Curvilinear base points, local complete intersection and {K}oszul
  syzygies in biprojective spaces},
        date={2006},
     journal={Trans. Amer. Math. Soc.},
      volume={358},
      number={8},
       pages={3385\ndash 3398},
}

\bib{HONG_SIMIS_VASC_ELIM}{article}{
      author={Hong, Jooyoun},
      author={Simis, Aron},
      author={Vasconcelos, Wolmer~V.},
       title={On the homology of two-dimensional elimination},
        date={2008},
     journal={J. Symbolic Comput.},
      volume={43},
      number={4},
       pages={275\ndash 292},
}

\bib{HULEK_KATZ_SCHREYER_SYZ}{article}{
      author={{Hulek}, Klaus},
      author={{Katz}, Sheldon},
      author={{Schreyer}, Frank-Olaf},
       title={{Cremona transformations and syzygies.}},
        date={1992},
        ISSN={0025-5874; 1432-1823/e},
     journal={{Math. Z.}},
      volume={209},
      number={3},
       pages={419\ndash 443},
}

\bib{HYRY_MULTIGRAD}{article}{
      author={Hyry, Eero},
       title={The diagonal subring and the {C}ohen-{M}acaulay property of a
  multigraded ring},
        date={1999},
     journal={Trans. Amer. Math. Soc.},
      volume={351},
      number={6},
       pages={2213\ndash 2232},
}

\bib{JEFFRIES_MONTANO_VARBARO}{article}{
      author={Jeffries, Jack},
      author={Monta\~no, Jonathan},
      author={Varbaro, Matteo},
       title={Multiplicities of classical varieties},
        date={2015},
     journal={Proc. Lond. Math. Soc. (3)},
      volume={110},
      number={4},
       pages={1033\ndash 1055},
}

\bib{JOUANOLOU_INVARIANT_ELIM}{article}{
      author={Jouanolou, Jean-Pierre},
       title={Aspects invariants de l'\'elimination},
        date={1995},
     journal={Adv. Math.},
      volume={114},
      number={1},
       pages={1\ndash 174},
}

\bib{KAP}{book}{
      author={Kaplansky, Irving},
       title={Commutative rings},
     edition={Revised},
   publisher={The University of Chicago Press, Chicago, Ill.-London},
        date={1974},
}

\bib{Kleiman_geom_mult}{article}{
      author={Kleiman, Steven},
      author={Thorup, Anders},
       title={A geometric theory of the {B}uchsbaum-{R}im multiplicity},
        date={1994},
     journal={J. Algebra},
      volume={167},
      number={1},
       pages={168\ndash 231},
}

\bib{KPU_NORMAL_SCROLL}{article}{
      author={Kustin, Andrew},
      author={Polini, Claudia},
      author={Ulrich, Bernd},
       title={Rational normal scrolls and the defining equations of {R}ees
  algebras},
        date={2011},
     journal={J. Reine Angew. Math.},
      volume={650},
       pages={23\ndash 65},
}

\bib{KPU_blowup_fibers}{article}{
      author={Kustin, Andrew},
      author={Polini, Claudia},
      author={Ulrich, Bernd},
       title={Blowups and fibers of morphisms},
        date={2016},
     journal={Nagoya Math. J.},
      volume={224},
      number={1},
       pages={168\ndash 201},
}

\bib{KPU_BIGRAD_STRUCT}{article}{
      author={Kustin, Andrew},
      author={Polini, Claudia},
      author={Ulrich, Bernd},
       title={The bi-graded structure of symmetric algebras with applications
  to {R}ees rings},
        date={2017},
     journal={J. Algebra},
      volume={469},
       pages={188\ndash 250},
}

\bib{KPU_GOR3}{article}{
      author={Kustin, Andrew},
      author={Polini, Claudia},
      author={Ulrich, Bernd},
       title={The equations defining blowup algebras of height three
  {G}orenstein ideals},
        date={2017},
     journal={Algebra Number Theory},
      volume={11},
      number={7},
       pages={1489\ndash 1525},
}

\bib{MARLEY_ASS_PRIMES_LOCAL_COHOM}{article}{
      author={Marley, Thomas},
       title={The associated primes of local cohomology modules over rings of
  small dimension},
        date={2001},
     journal={Manuscripta Math.},
      volume={104},
      number={4},
       pages={519\ndash 525},
}

\bib{MICALI_REES}{inproceedings}{
      author={Micali, Artibano},
       title={Sur les algebres universelles},
        date={1964},
   booktitle={Annales de l'institut fourier},
      volume={14},
       pages={33\ndash 87},
}

\bib{PAN_RUSSO}{article}{
      author={Pan, Ivan},
      author={Russo, Francesco},
       title={Cremona transformations and special double structures},
        date={2005},
     journal={Manuscripta Math.},
      volume={117},
      number={4},
       pages={491\ndash 510},
}

\bib{SIMIS_PAN_JONQUIERES}{article}{
      author={Pan, Ivan},
      author={Simis, Aron},
       title={Cremona maps of de {J}onqui\`eres type},
        date={2015},
     journal={Canad. J. Math.},
      volume={67},
      number={4},
       pages={923\ndash 941},
}

\bib{SIMIS_RUSSO_BIRAT}{article}{
      author={Russo, Francesco},
      author={Simis, Aron},
       title={On birational maps and {J}acobian matrices},
        date={2001},
     journal={Compositio Math.},
      volume={126},
      number={3},
       pages={335\ndash 358},
}

\bib{SCHENZEL_NOTES}{incollection}{
      author={Schenzel, Peter},
       title={On the use of local cohomology in algebra and geometry},
        date={1998},
   booktitle={Six lectures on commutative algebra ({B}ellaterra, 1996)},
      series={Progr. Math.},
      volume={166},
   publisher={Birkh\"auser, Basel},
       pages={241\ndash 292},
}

\bib{SEDERBERG20161}{article}{
      author={Sederberg, Thomas~W.},
      author={Goldman, Ronald~N.},
      author={Wang, Xuhui},
       title={Birational 2d free-form deformation of degree $1\times n$},
        date={2016},
     journal={Computer Aided Geometric Design},
      volume={44},
       pages={1\ndash 9},
}

\bib{SEDERBERG20151}{article}{
      author={Sederberg, Thomas~W.},
      author={Zheng, Jianmin},
       title={Birational quadrilateral maps},
        date={2015},
     journal={Computer Aided Geometric Design},
      volume={32},
       pages={1\ndash 4},
}

\bib{Simis_cremona}{article}{
      author={Simis, Aron},
       title={Cremona transformations and some related algebras},
        date={2004},
     journal={J. Algebra},
      volume={280},
      number={1},
       pages={162\ndash 179},
}

\bib{SIM_ULRICH_VASC_JACOBIAN}{article}{
      author={Simis, Aron},
      author={Ulrich, Bernd},
      author={Vasconcelos, Wolmer~V.},
       title={Jacobian dual fibrations},
        date={1993},
     journal={Amer. J. Math.},
      volume={115},
      number={1},
       pages={47\ndash 75},
}

\bib{Sim_Ulr_Vasc_mult}{article}{
      author={Simis, Aron},
      author={Ulrich, Bernd},
      author={Vasconcelos, Wolmer~V.},
       title={Codimension, multiplicity and integral extensions},
        date={2001},
     journal={Math. Proc. Cambridge Philos. Soc.},
      volume={130},
      number={2},
       pages={237\ndash 257},
}

\bib{Rees_alg_modules}{article}{
      author={Simis, Aron},
      author={Ulrich, Bernd},
      author={Vasconcelos, Wolmer~V.},
       title={Rees algebras of modules},
        date={2003},
     journal={Proc. London Math. Soc. (3)},
      volume={87},
      number={3},
       pages={610\ndash 646},
}

\bib{SIMIS_VASC_SYM_INT}{article}{
      author={Simis, Aron},
      author={Vasconcelos, Wolmer~V.},
       title={On the dimension and integrality of symmetric algebras},
        date={1981},
        ISSN={0025-5874},
     journal={Math. Z.},
      volume={177},
      number={3},
       pages={341\ndash 358},
}

\bib{Simis_Vasc_Syz_Conormal_Mod}{article}{
      author={Simis, Aron},
      author={Vasconcelos, Wolmer~V.},
       title={The syzygies of the conormal module},
        date={1981},
     journal={Amer. J. Math.},
      volume={103},
      number={2},
       pages={203\ndash 224},
}

\bib{stacks-project}{misc}{
      author={{Stacks project authors}, The},
       title={The stacks project},
         how={\url{https://stacks.math.columbia.edu}},
        date={2019},
}

\bib{STURMFELS_GROBNER}{book}{
      author={Sturmfels, Bernd},
       title={Gr\"obner bases and convex polytopes},
      series={University Lecture Series},
   publisher={American Mathematical Society, Providence, RI},
        date={1996},
      volume={8},
}

\bib{VASC_EQ_REES}{article}{
      author={Vasconcelos, Wolmer~V.},
       title={On the equations of {R}ees algebras},
        date={1991},
     journal={J. Reine Angew. Math.},
      volume={418},
       pages={189\ndash 218},
}

\end{biblist}
\end{bibdiv}

\end{document}